\documentclass[11pt,reqno]{amsart}

\usepackage[dvips]{graphicx}
\usepackage{amssymb, latexsym, amsmath, amscd, array, amsthm, epsfig, pstricks}
\usepackage[all]{xy}
\usepackage{bbm}

\usepackage[pagebackref=true]{hyperref}

\usepackage{breakurl}   %this allows for url breaks in pdf

\usepackage{mathrsfs}

\usepackage{fontenc}
\usepackage{amsmath,amscd}
\usepackage{latexsym}
\usepackage{enumerate, enumitem}

\newcommand{\Z}{{\mathbb Z}}
\newcommand{\R}{{\mathbb R}}
\newcommand{\Q}{{\mathbb Q}}

\newcommand{\RP}{{\mathbb RP}}

\newcommand{\HP}{{\mathbb HP}}

\newcommand{\Isom}{{\rm Isom}}

\newcommand{\diam}{{\rm diam}}

\newcommand{\inj}{{\rm inj}}

\newcommand{\vol}{{\rm vol}}

\newcommand{\w}{{\rm W}}
\newcommand{\UW}{{\rm UW}}
\newcommand{\HS}{{\rm HS}}
\newcommand{\length}{{\rm length}}
\newcommand{\FR}{{\rm FillRad}}
\newcommand{\FillRad}{{\rm FillRad}}
\newcommand{\FH}{{\rm FH}}
\newcommand{\FHb}{\overline{\rm FH}}

\numberwithin{equation}{section}

\newtheorem{theorem}{Theorem}[section]
\newtheorem{proposition}[theorem]{Proposition}

\newtheorem{lemma}[theorem]{Lemma}
\newtheorem{claim}[theorem]{Claim}

\theoremstyle{definition}
\newtheorem{definition}[theorem]{Definition}
\newtheorem{example}[theorem]{Example}
\newtheorem{remark}[theorem]{Remark}

\long\def\forget#1\forgotten{} 

\begin{document}

\title{Sweepouts of closed Riemannian manifolds}

\author[A.~Nabutovsky]{Alexander Nabutovsky}
\author[R.~Rotman]{Regina Rotman}
\author[S.~Sabourau]{St\'ephane Sabourau}

\thanks{Partially supported by the ANR project Min-Max (ANR-19-CE40-0014).}

\address{Department of Mathematics, University of Toronto, 40 St. George Street, Toronto, ON M5S 2E4, Canada} 

\email{alex@math.utoronto.ca}

\address{Department of Mathematics, University of Toronto, 40 St. George Street, Toronto, ON M5S 2E4, Canada} 

\email{rina@math.utoronto.ca}

\address{\parbox{\linewidth}{LAMA, Univ Paris Est Creteil, Univ Gustave Eiffel, CNRS, F-94010, Cr\'eteil, France \\
CRM (UMI3457), CNRS, Montr\'eal, QC H3C 3J7, Canada}}

\email{stephane.sabourau@u-pec.fr}

\subjclass[2020]
{Primary 53C23; Secondary 53C20}

\begin{abstract}
We show that for every closed Riemannian manifold there exists a continuous family of $1$-cycles (defined as finite collections of disjoint closed curves) 
parametrized by a sphere and sweeping out the whole manifold  so that the lengths of all connected closed curves are bounded in terms of the volume (or the diameter) and the dimension~$n$ of the manifold, when $n \geq 3$. 
An alternative form of this result involves a modification of Gromov's definition of waist of sweepouts, where the space of parameters can be any finite polyhedron (and not necessarily a pseudomanifold). 
We demonstrate that the so-defined  polyhedral $1$-dimensional waist of a closed Riemannian manifold is equal to its filling radius up to at most a constant factor.
We also establish upper bounds for the  polyhedral $1$-waist of some homology classes in terms of the volume or the diameter of the ambient manifold.
In addition, we provide generalizations of these results for sweepouts by polyhedra of higher dimension using the homological filling functions.
Finally, we demonstrate that the filling radius and the hypersphericity of a closed Riemannian manifold can be  arbitrarily far apart.
\end{abstract}

\maketitle

\section{Introduction}

One can define a {\it $p$-slicing} of an $n$-dimensional manifold (or a pseudomanifold) $N$ 
as a collection of inverse images $h^{-1}(t)$ of points under a ``nice" mapping $h:N^n \to T^{n-p}$ to a lower dimensional (pseudo)manifold,
where $t$ runs over~$T^{n-p}$. Here, the ``niceness" of $h$ must imply that all $h^{-1}(t)$ are subpolyhedra of $N^n$ of dimension $\leq p$.
(Note, however, that even for a nice map~$h$, the fibers~$h^{-1}(t)$ need not automatically form a continuous family of subpolyhedra under
any reasonable choice of topology.) %(and the assumptions on~$h$ imply that $h^{-1}(t)$ is a continuous family of $p$-cycles). 
One can define a {\it $p$-sweepout} of $M$ as the image of a slicing of a (pseudo)manifold~$N$ under a topologically nontrivial continuous map~$\varphi:N \to M$. 
Here, one possible meaning of ``topologically nontrivial" is that the image under $\varphi$ of the fundamental homology class of~$N$ with $\mathbb{Z}$ or $\mathbb{Z}_2$ coefficients is the fundamental homology class of~$M$. Obviously, a slicing is a sweepout (corresponding to the identity map
of $M$), but not vice versa.
More generally, one can define a $p$-sweepout of a $k$-dimensional homology class $a$ of $M$ from a continuous map $\varphi:N \to M$ with $\varphi_*([N])=a$.
%More generally, given a $k$-dimensional homology class $a$ of $M$, and choosing a $k$-dimensional pseudomanifold $N$ such that $\varphi_*([N])=a$,
%%is a prescribed nontrivial homology class $a$ in~$H_*(M)$, and
%we can define $p$-sweepouts of $a$.
Taking the infimum of the $p$-volumes of the fibers $h^{-1}(t)$ over all maps $h:N \to T^{n-p}$ and $\varphi:N \to M$ such that $\varphi_*([N])=[M]$ (respectively, $\varphi_*([N])$ represents a prescribed lower dimensional class $a$ of $M$), one arrives to the definition of the {\it $p$-waist} of~$M$ (respectively, of a homology class~$a$ of~$M$).
This definition was introduced by Gromov; see~\cite[\S15]{Gromov101} and~\cite[\S6]{gro}. 

Among other things we analyze a version of this concept when the space $T^{n-p}$ is not required to be a pseudomanifold, but can be an arbitrary finite
polyhedron.
%Note that in these papers by Gromov $T^{n-p}$ is assumed to be a submanifold, but
%in the present paper $T^{n-p}$ is allowed to be any finite polyhedron (see Definition~\ref{def:width} below).
%\medskip
%Given a $p$-sweepout (in particular, a $p$-slicing), 
%one can calculate the maximal volume of a 
%$p$-cycle, and then take the minimum over all
Consider a $p$-sweepout of $M$ (or, more generally, of a homology class $a$ of $M$ with coefficients in $G=\mathbb{Z}$ or $\mathbb{Z}_2$) whose fibers $h^{-1}(t)$ are $p$-dimensional cycles that continuously depend on $t$. We can regard $\varphi(h^{-1}(t))$ as 
$p$-cycles of~$M$. (If $h^{-1}(t)$ is empty, then the $1$-cycle is, by definition, zero.) 
Therefore, the maps $h:N \to T^{n-p}$ and $\varphi:N \to M$ induce a
continuous map from the pseudomanifold $T^{n-p}$ to the space $\mathcal{Z}_p(M;G)$ of $p$-cycles of $M$ with coefficients in $G$. %=\mathbb{Z}$ or $\mathbb{Z}_2$.
 This map sends the fundamental homology class of $T^{n-p}$ to an $(n-p)$-dimensional homology class of $\mathcal{Z}_p(M;G)$. If this homology class is nontrivial,
we can take the minimax value of the $p$-volume over $p$-cycles, where the supremum is taken over the space of all cycles in a sweepout, and the infimum over all
sweepouts $h:N \to T^{n-p}$ and $\varphi:N \to M$ corresponding to the considered homology class of $\mathcal{Z}_p(M;G)$. The resulting quantities (called the {\it widths} of~$M$) were studied in many important papers. Properties of widths were crucial for a recent progress in geometric analysis including proofs of the Willmore conjecture by F.~Marques and A.~Neves, and the Yau conjecture by A.~Song. 
If $T^{n-p}=S^{n-p}$, then the construction gives rise to an $(n-p)$-dimensional homotopy class of $\mathcal{Z}_p(M;G)$. As $\pi_{n-p}(\mathcal{Z}_p(M;G))$ is
isomorphic to $H_n(M;G)$ (see~\cite{alm}), we obtain a corresponding homology class of~$M$. Looking at Almgren's proof in \cite{alm}, it is easy to see that
this is the same homology class that was used in the definition of the $p$-sweepout.

There are many different versions of $p$-waists in the literature that are referred to as width, waist, diastole, etc.
One can consider all sweepouts, or only slicings, consider different notions of ``topologically nontrivial", or impose different
restrictions on the spaces of parameters parameterizing sweepouts. 
%ne can drop the requirement that sweepouts are by cycles,
%allowing more general types of fibers~$h^{-1}(t)$ (\eg, polyhedra, or even arbitrary compact sets).
Also, one can measure the size of the fibers $h^{-1}(t)$ not as their volume, but using a  different functional, such as the diameter. 
For example, if we consider only slicings (so $\varphi$ is the identity map) by arbitrary not necessarily polyhedral fibers~$h^{-1}(t)$ and measure their size using the diameter, we obtain the notion of {\it Urysohn width}. 
It was first proven by L.~Guth, that if the codimension $p$ equals~$1$, then the Urysohn width of a closed Riemannian manifold $M$ can be majorized in terms of only the volume and the dimension of $M$; see~\cite{guth17}, \cite{LLNR}, \cite{pap}, \cite{nab}, \cite{sab20} for this and related results.
In this paper, we will consider the case, when $T^{n-p}$ is a pseudomanifold as in Gromov's definition, but the functional of interest
is the maximal $p$-volume of a connected component of the fiber $h^{-1}(t)$ (instead of the $p$-volume of the whole  $h^{-1}(t))$. Also, we are going to
do this in the most interesting case, when $T^{n-p}=S^{n-p}$, and for maps~$h:N \to T^{n-p}$ such that $h^{-1}(t)$ is a continuous family of $p$-cycles.
%Finally, we do not insist on $h^{-1}(t)$ to be $1$-cycles but merely $1$-dimensional finite polyhedra (i.e. graphs).
Thus, the resulting notion could be regarded as a version of the widths of $M$ (for a different functional).
%If one considers slicing of closed Riemannian surfaces by $1$-parametric families of $1$-cycles parameterized by $S^1$, and measures their length, then the resulting slicing waist is also called {\it diastole}.

\medskip

%There are several ways to define the width of a closed manifold~$M$ depending on the notion of sweepout considered and the way their size is measured.
%For instance, one could define the width by minimizing the maximal volume of the fibers over all Morse functions on~$M$.
There is a number of papers with interesting and important upper and lower bounds for various versions of waists/widths.
Here are some of them for the most studied situation when we take the min-max of the volume over all slicings of a Riemannian manifold (slicing waists).

\medskip

\noindent {\bf Upper bounds.} 
The first bound in this direction applies to closed Riemannian surfaces and involves only the area and the genus of the surface; see~\cite{BS10}.
In higher dimension, a similar bound in terms of the volume holds for closed Riemannian $n$-manifold~$M$ with nonnegative Ricci curvature, which leads to the existence of a closed minimal hypersurface (with a singular set of Hausdorff dimension at most~$n-8$) whose $(n-1)$-volume is bounded in terms of the volume of~$M$;
%of volume at most $c_n \, \vol(M)^{\frac{n-1}{n}}$, for some constant~$c_n$ depending only on~$n$; 
see~\cite{GL17}, \cite{sab17}.
Further estimates for closed Riemannian manifolds with Ricci curvature bounded below can be found in~\cite{GL17}.
For closed $3$-manifolds~$M^3$ with positive Ricci curvature, the bound in terms of the volume holds for the min-max length of the fibers of the maps from~$M$ to the plane; ; see~\cite{LZ18}.
%$M \to \R^2$; see~\cite{LZ18}.
Without any curvature condition, these results fail; see~\cite{PS} and~\cite{pap} for counterexamples.

\medskip

\noindent {\bf Lower bounds.} 
A lower bound for slicing waists of the form ${\rm const}_n \, \FillRad(M^n)$ can be found in Appendix~1 of~\cite[\S6]{gro}, where $\FillRad$ is the filling radius; see Definition~\ref{def:FR}.
In a different direction the exact lower bound for the volume of fibers of maps between round spheres of different dimensions is the content of Gromov's famous waist theorem; see~\cite{gromovwaist}. 
There are many generalizations for other spaces; see~\cite{klartag} for a highly nontrivial generalization of Gromov's waist theorem for cubes and \cite{ahk} for a survey on the subject.

\forget
Sweepout estimates are min-max volume estimates over families of cycles satisfying nontrivial topological conditions.
These min-max values are referred to as width, waist, diastole, slicing, etc. in the literature.
They generalize systolic inequalities and are closely related to minimal surface theory and the geometric calculus of variations.
Their study has recently been the object of numerous works and led to major advances in geometry, including the resolutions of the Willmore conjecture and the Yau conjectures. 

\medskip

There are several ways to define the width of a closed manifold~$M$ depending on the notion of sweepout considered and the way their size is measured.
For instance, one could define the width by minimizing the maximal volume of the fibers over all Morse functions on~$M$.
With this definition, the area of every closed surface with a Riemannian metric is bounded from below by the square of the width, up to a metric-independent multiplicative constant; see~\cite{BS10}.
In higher dimension, a similar universal volume lower bound holds for closed Riemannian $n$-manifolds~$M$ with nonnegative Ricci curvature, which leads to the existence of a closed minimal hypersurface (with a singular set of Hausdorff dimension at most~$n-8$) of volume at most $c_n \, \vol(M)^{\frac{n-1}{n}}$, for some constant $c_n$ depending only on~$n$; see~\cite{GL17}, \cite{sab17}.
For closed $3$-manifolds~$M$ with positive Ricci curvature, a similar universal volume lower bound also holds with the min-max length of the fibers of the maps $M \to \R^2$; see~\cite{LZ18}.
Without any curvature condition, these results fail; see~\cite{PS} and~\cite{pap} for counterexamples.

\medskip

Another way to define the width would be to minimize the maximal diameter of the fibers over all continuous functions from~$M$ to a simplicial $(n-p)$-complex.
This notion of width, also known as Urysohn width, provides a universal lower bound on the volume of every closed Riemannian $n$-manifold for $p=1$; see~\cite{guth17}, \cite{LLNR}, \cite{pap}, \cite{nab}.

\medskip

In these definitions, the sweepouts consist of fibers of maps and therefore are pairwise disjoint.
For a more general notion of sweepout allowing cycles to intersect each other, one can consider maps $X \to {Z}_p(M)$ from a finite simplicial complex~$X$ to the $p$-cycle space~${Z}_p(M)$ of~$M$ studied by Almgren in~\cite{alm}.
Combined with results from geometric measure theory, this notion is well designed for applications and is at the cornerstone of major results in minimal surface theory.
One issue with this definition though is that it does not include the classical sweepout of the ``three-legged starfish'' sphere composed of single loops and parameterized by the tripod tree (as it is not continuous on the one-cycle space endowed with any natural topology).
\forgotten

%For a more general notion of sweepout allowing cycles to intersect each other, one can consider maps $X \to \mathcal{Z}_p(M)$ from a finite simplicial complex~$X$ to the $p$-cycle space~$\mathcal{Z}_p(M)$ of~$M$ studied by Almgren in~\cite{alm}.
%Combined with results from geometric measure theory, this notion is well designed for applications and is at the cornerstone of major results in minimal surface theory.
%One issue with this definition though is that it does not include the classical sweepout of the ``three-legged starfish'' sphere composed of single loops and parameterized by the tripod tree (as it is not continuous on the one-cycle space endowed with any natural topology).

\subsection{Homology $1$-waist bounds}

In this paper, we consider two different notions of sweepout and waist similar but somewhat different to those introduced by Gromov in~\cite[\S6]{gro}; see Definitions~\ref{def:width} and~\ref{def:width2}. 
%Instead, we consider different notions of sweepout and width which were introduced by Gromov in~\cite[\S6]{gro}; see Definitions~\ref{def:width} and~\ref{def:width2}.
With these notions, universal upper bounds on the waist of one-parameter families of one-cycles sweeping out essential surfaces of closed Riemannian manifolds were obtained in~\cite{sab}.
%In particular, it was proved that given any Riemannian metric on the complex projective space~$\CP^n$, there exists a one-parameter family of one-cycles on~$\CP^n$ sweeping out its essential sphere~$\CP^1$ whose one-cycle length is at most~$c_n \, \vol(\CP^n)^\frac{1}{2n}$, for some constant~$c_n$ depending only on~$n$.
Our first theorem extends this result to the waist of multi-parameters families of one-cycles sweeping out any closed Riemannian manifold.

%In this article, we give a complete picture of the situation for one-cycle sweepouts extending the results of~\cite{sab}.
%More specifically, we show that the width defined from families of one-cycles entirely sweeping out any closed manifold with a Riemannian metric (and not only essential surfaces when they exist) provides a universal lower bound both on the volume and the diameter of the manifold.

\medskip

Before stating the precise result, we need to introduce the following notion of $p$-waist.

\begin{definition} \label{def:width} Let $M$ be a closed $n$-manifold. 
A polyhedral \emph{$p$-sweepout} of~$M$ is a family 
% of finite $p$-dimensional subpolyhedra of $M$
%simplicial $p$-complex in~$M$
\[
\varphi[h^{-1}(t)] \subseteq M
\]
with $t \in T$, where $h:N \to T$ is a continuous map from a closed $n$-pseudomanifold~$N$ to a finite $(n-p)$-dimensional polyhedron~$T$ such that
all fibers $h^{-1}(t)$ are $p$-subpolyhedra of $N$, and $\varphi:N \to M$ is a continuous degree one map.
%Such a sweepout is \emph{homologically substantial(} if $\varphi$ is a degree one map. 
That is,
\[
\varphi_*([N]) = [M] \in H_n(M)
\]
where the homology coefficients are in~$\Z$ if $M$ is orientable, and in~$\Z_2$ otherwise.
%(We refer to Definition~\ref{def:pseudomanifold} for the definition of a pseudomanifold.)
Define the \emph{homology $p$-waist} of a closed Riemannian manifold~$M$ as
\begin{equation} \label{eq:Wp}
\w_p(M) = \inf_{\raisebox{-3pt}{\scriptsize $\varphi,h$}} \, \sup_{t \in T} \,  \vol_p(\varphi_{|h^{-1}(t)})
\end{equation}
where the infimum is taken over all $n$-pseudomanifolds~$N$, all simplicial $(n-p)$-complexes~$T$, and all maps $\varphi:N \to M$ and $h:N \to T$ defining a polyhedral $p$-sweepout of~$M$.
Here, the notation $\vol_p(\varphi_{|h^{-1}(t)})$ stands for the volume of the map~$\varphi$ restricted to the fiber~$h^{-1}(t) \subseteq N$, and not merely the volume of its image (which might be smaller).
If such sweepouts do not exist, we let \mbox{$\w_p(M)=0$}.
When $p=1$, we simply write $\w(M)=\w_1(M)$.
\end{definition}

The only difference between this definition and Gromov's is that we do not require that $T$ is a pseudomanifold. This distinction can be illustrated by the following example. 

\begin{example} \label{ex:fish1}
Let $M$ be a Riemannian $2$-sphere that looks like a ``thin"
$2$-dimensional three-legged starfish. We can choose~$T$ as a tripod, that is, the union of three closed intervals intersecting
at a common endpoint, and polyhedral $1$-sweepout of $M$ by 
$1$-cycles, most of which are very short closed curves running around individual tentacles.
Only one of them, namely, the inverse image of the center of the tripod looks like the $\theta$-graph with two vertices and three edges connecting these vertices. (Each pair of
these three edges forms a closed curve around one of the tentacles that appears as the limit
of closed curve above inner points of the corresponding leg of tripod~$T$.) Note that this polyhedral sweepout would not be allowed in Gromov's definition. Also, note that if one would consider the inverse images of a point of~$T$ as a function
from~$T$ to the space of currents on the $2$-sphere,
this function will {\it not} be continuous. Indeed,
the inverse image of the center of the tripod will be the $\theta$-graph that consists of thee arcs. When we approach the center along each ray, the inverse images of points will consist of two arcs,
and will converge to a subset of the $\theta$-graph that consists of two (out of three) arcs. Note, that this type of discontinuity
would be impossible if $h$ were a map to a manifold ({\it e.g.}, a sphere) such that for each $t$, the fiber $h^{-1}(t)$ is a cycle.
\end{example}

With this notion of waist, we can prove the following homology $1$-sweepout estimates.

\begin{theorem} \label{theo:A}
Let $M$ be a closed $n$-manifold.
Then every Riemannian metric on~$M$ satisfies 
\begin{align*}
\FillRad (M) & \geq c_n\, \w(M) \\
\vol(M) & \geq c'_n \, \w(M)^n \\
\diam(M) & \geq c''_n \, \w(M)
\end{align*}
for some explicit positive constants~$c_n$, $c'_n$ and~$c_n''$ depending only on~$n$.
\end{theorem}

Here, $\FillRad(M)$ is the filling radius introduced by Gromov in~\cite{gro83}; see Definition~\ref{def:FR} below. Gromov proved that $\FillRad(M^n)\leq const(n) \, \vol(M)^{\frac{1}{n}}$.
Later, it was established in~\cite{nab} that one can take $const(n)=n$. Also, M.~Katz proved that $\FillRad (M)\leq \frac{1}{3} \diam(M)$; see~\cite{katz83}.
Thus, the last two inequalities follow from the first one.
This result is geometrically appealing, and can be used to demonstrate that $\FillRad(M)$ is equal to $\w(M)$ up to at most a constant factor (see Theorem~\ref{theo:C}), thus
leading to some geometric intuition about $\FillRad(M)$. 

\medskip

%The proof of Theorem~\ref{theo:A} %and Theorem~\ref{theo:B} relies on the notion of filling radius.
%introduced by Gromov~\cite{gro83} to establish systolic inequalities; see Definition~\ref{def:FR}.
%More precisely, we show that the filling radius of a closed Riemannian $n$-manifold is bounded from below by its homology/homotopy $1$-waist, up to a multiplicative constant depending only on~$n$; see Theorem~\ref{theo:FR1} and Theorem~\ref{theo:FR2}.
The classical approach to obtain lower bounds on the filling radius is to argue by contradiction and construct a retraction, one simplex at a time, from a pseudomanifold~$P$ bounding~$M$ onto its boundary. %at a fixed Gromov-Hausdorff distance onto~$M$.
However, such a construction is not always possible in general.
A different path was taken in~\cite{sab}, where a retraction from a different filling~$Q$ was constructed by considering all the simplices lying in the $2$-skeleton of~$Q$ at the same time (and not only one at a time) and by proceeding by induction on the dimension of the higher-dimensional skeleta of~$Q$ from there, using a topological assumption on the manifold.
In the proof of Theorem~\ref{theo:A}, where the manifold~$M$ is arbitrary, we take yet a different approach.
In particular, our construction does not proceed by induction on the skeleta of the filling.
Instead, we construct a pseudomanifold~$N$ homologuous to~$M$ and a map $N \to M$ non-homologuous to the identity map by considering all simplices of the filling at the same time without arguing by induction in order to derive a contradiction.

\medskip

It would be interesting to know whether the one-cycle sweepout estimates of Theorem~\ref{theo:A} hold for sweepouts made of pairwise disjoint one-cycles, that is, when the maps $\varphi:N \to M$ in Definition~\ref{def:width} are required to be diffeomorphisms.
In the case of surfaces, this would yield a positive answer to Bers' pants decomposition problem, which may or may not be true.

\subsection{Sweepouts and geometric measure theory} \label{subsec:gmt}

From the point of view of geometric measure theory, one would prefer a situation where
\begin{itemize}
    \item all inverse images $h^{-1}(t)$ are $1$-cycles on~$N$ (or, more precisely, finite collections of piecewise smooth closed curves on~$N$ so that their images~$\varphi(h^{-1}(t))$ are $1$-cycles on~$M$);
    \item $T=S^{n-1}$;
    \item the map $S^{n-1}\to \mathcal{Z}_1(M;G)$ with $G=\mathbb{Z}$ if $M$ is orientable, and $G=\mathbb{Z}_2$ otherwise, that sends every $t\in T$ to
$\varphi(h^{-1}(t))$ is continuous with respect to the flat topology on $\mathcal{Z}_1(M;G)$.
\end{itemize}
%This map will be automatically continuous and
%its homotopy class will be sent to $[M]$ under the Almgren isomorphism.

%
%like to construct
%a continuous map $\bar h$ from a pseudomanifold $T$ to the space $Z_1(M;G)$ of $1$-dimensional cycles on $M$ with coefficients in $G=\mathbb{Z}$ or $\mathbb{Z}_2$, where the image of each point is a ``geometric" $1$-cycle composed of a finite number of closed piecewise smooth curves. Moreover, one would like to be in a situation, when $T=S^{n-1}$, and  this map is homologically significant, \ie, the map~$\bar h_*$ takes the generator of $\pi_{n-1}(T)$ to
%the $(n-1)$-homotopy class of $Z_1(M;G)$ that corresponds to the fundamental homology class of $M$ with coefficients in $G$ via the Almgren isomorphism.
%Note that the slicing of $N$ by $h^{-1}(t)$
%it has two drawbacks. The first drawback is that the family of polyhedra in a sweepout
%are parameterized not by a (pseudo)manifold but by an arbitrary complex~$T$. The second drawback (from the point of view of geometric calculus of variations)
%is that the resulting map of $T$ to the space of $k$-dimensional varifolds (or even the space of $k$-currents)  is NOT continuous.
%Both drawbacks can be fixed (if desired), but there will be a price to pay.

%\medskip

To achieve this goal in the case when $h^{-1}(t)$ is already a $1$-cycle for each~$t$, %(as it is the case in Definition~\ref{def:width}), 
we can first replace the map $h:N \to T$ with a continuous map $\bar{h}:N \to S^{n-p}$ obtained as the composition of $h:N \to T$ with a finite-to-one continuous map $T \to S^{n-p}$.
In this case, the $p$-sweepouts are parameterized by~$S^{n-p}$ and the $p$-waist of~$M$ can be defined by minimizing the maximal volume of the map~$\varphi$ restricted to the 
connected components of the fibers of~$\bar{h}$. (Note that we cannot hope to have a control over the cardinality of the inverse images of the many-to-one map to the sphere. Therefore, the best we can hope for is to control the volume (length) of the individual connected components.)
That is,
\begin{equation} \label{eq:wp}
\bar{\w}_p(M) = \inf_{\raisebox{-3pt}{\scriptsize $\varphi,\bar{h}$}} \, \sup_{t \in S^{n-p}} \,  \max_{\raisebox{-3pt}{\scriptsize $C \subseteq \bar{h}^{-1}(t)$}} \, \vol_p(\varphi_{|C})
\end{equation}
where the infimum is taken over all maps $\varphi:N \to M$ and $\bar{h}: N \to S^{n-p}$ as above and the maximum is taken over all connected components~$C$ of~$\bar{h}^{-1}(t)$. 
%In other words, we combine slices $h^{-1}(t)$ into finite collections, resulting in the\element of $\pi_{n-p}(Z_p(N;G))$ corresponding
%to the fundamental homology class of $[N]$ via Almgren's isomorphism. Then, $\varphi_*$ maps this homotopy class of $Z_p(N;G)$ to a homotopy class of %$Z_p(M;G)$.

\medskip

Note that, vice versa, given a PL-map $\bar{h}:N \to S^{n-p}$, one can define the space $T^{n-p}$ of connected components
of all inverse images $\bar{h}^{-1}(t)$ for every $t\in S^{n-p}$. 
%This space $T^{n-p}$ is the image of $N$ under the quotient map $h:N \to T^{n-p}$ that sends each
This gives rise to a map $h:N \to T^{n-p}$ sending every point $x\in N$ to its connected component in $\bar{h}^{-1}(\bar h(x))$.
There exists a map $\psi:T^{n-p}\to S^{n-p}$ that sends each point of $t \in T^{n-p}$ to the corresponding value of~$t$ under~$h$ such that $\bar{h}=\psi \circ h$. 
Obviously, the fibers of~$h$ in~$N$ are connected and coincide with the connected components of the fibers of~$\bar{h}$. As we will see below, the map 
that sends each point $t\in T$ to the $1$-cycle $\varphi(\bar h^{-1}(t))$ need not be continuous, so one will need some extra care in constructing the finite-to-one map $\psi:T^{n-p} \to S^{n-p}$ to ensure the continuity of this map.

\medskip

%Yet, the map of $S^{n-p}$ to the space of, say, $p$-dimensional integral polyhedral chains on $M$ is, in general, still not continuous. The reason for the lack of continuity
%is just that, in general,  the family of fibers $\bar h^{-1}(s)$  of a ``generic" PL-map $\bar h: N\to S^{n-p}$ changes discontinuously under
%any reasonable choice of topology. Further, sets $h^{-1}(t)$ need not be cycles. For example, if $p=1$ our proof of Theorem~\ref{theo:A} produces a map~$h$ such that $h^{-1}(t)$ are graphs with vertices of
%odd degree for some $t$, and cannot be interpreted as $1$-cycles.

As pointed out before, the first step is to alter $h:N \to T^{n-1}$ so that all $h^{-1}(t)$ become $1$-cycles.
To illustrate our approach consider the following example

\begin{example} \label{ex:fish2}
Consider the three-legged starfish $2$-sphere $M=N$ mapped to the tripod $T$ as described in Example~\ref{ex:fish1}. The inverse
image of the centre $c$ of the tripod is a collection of three arcs connecting two points on the sphere (a $\theta$-graph), which is not a cycle. The inverse images of the points
on each ray of the tripod are closed curves hugging the legs of the sphere. As a point on a ray of the tripod approaches the center, its inverse image approaches the union
of two of the three arcs forming the $\theta$-graph. 
%So, we have a discontinuity as well. The following change in our construction helps to get rid of
%these problems: 
Replace the three arcs in the inverse image of $c$ by pairs of arcs with the same endpoints as the original arc, running very close to the original arc. Each pair forms a digon bounding a thin disk. The boundary of each disk can be contracted to a point inside the disk via concentric simple loops.
%non-self-intersecting and pairwise non-intersecting curves.
This contraction corresponds to a map of each thin disk to a small interval such that the curves during the contracting homotopy are inverse images
of the points of the small interval. Combining these three homotopies, we obtain a map of three thin disks to a small tripod. The inverse image of the center of this small tripod is a collection of six arcs. Now, note that these six arcs can be grouped into three pairs of arcs so that each pair forms a simple closed curve ``hugging" one of three
long legs of the three-legged star-fish. Each of these three simple closed curves can be also contracted to a point via concentric simple loops 
%non-self-intersecting and pairwise non-intersecting curves 
along the corresponding long leg of the three-legged starfish. Combining these three contracting homotopies, we obtain
a map of the $2$-sphere minus the three thin disks to another tripod. (This map will be very close to the original map of the whole $2$-sphere to the tripod).
The inverse image of the centre of the tripod under this new map is the same collection of six arcs (or, three petals). 
We can glue these two tripods into one hexapod by identifying their centers and define a map from the $2$-sphere to the hexapod by combining the two maps defined on the union of the three thin disks and its complement to the tripods forming the hexapod.
The maximal length of a fiber is (almost) twice the maximal length of a fiber in the original map to
a tripod and every fiber now is a $1$-cycle. (Observe that an elaboration of this idea can be used to enhance Theorem~\ref{theo:A} by demanding
that all fibers $h^{-1}(t)$ are $1$-cycles, if desired. This will follow from an argument used to prove our next theorem below.)

\medskip

\begin{figure}[!htbp]
\centering
\includegraphics[width=9cm]{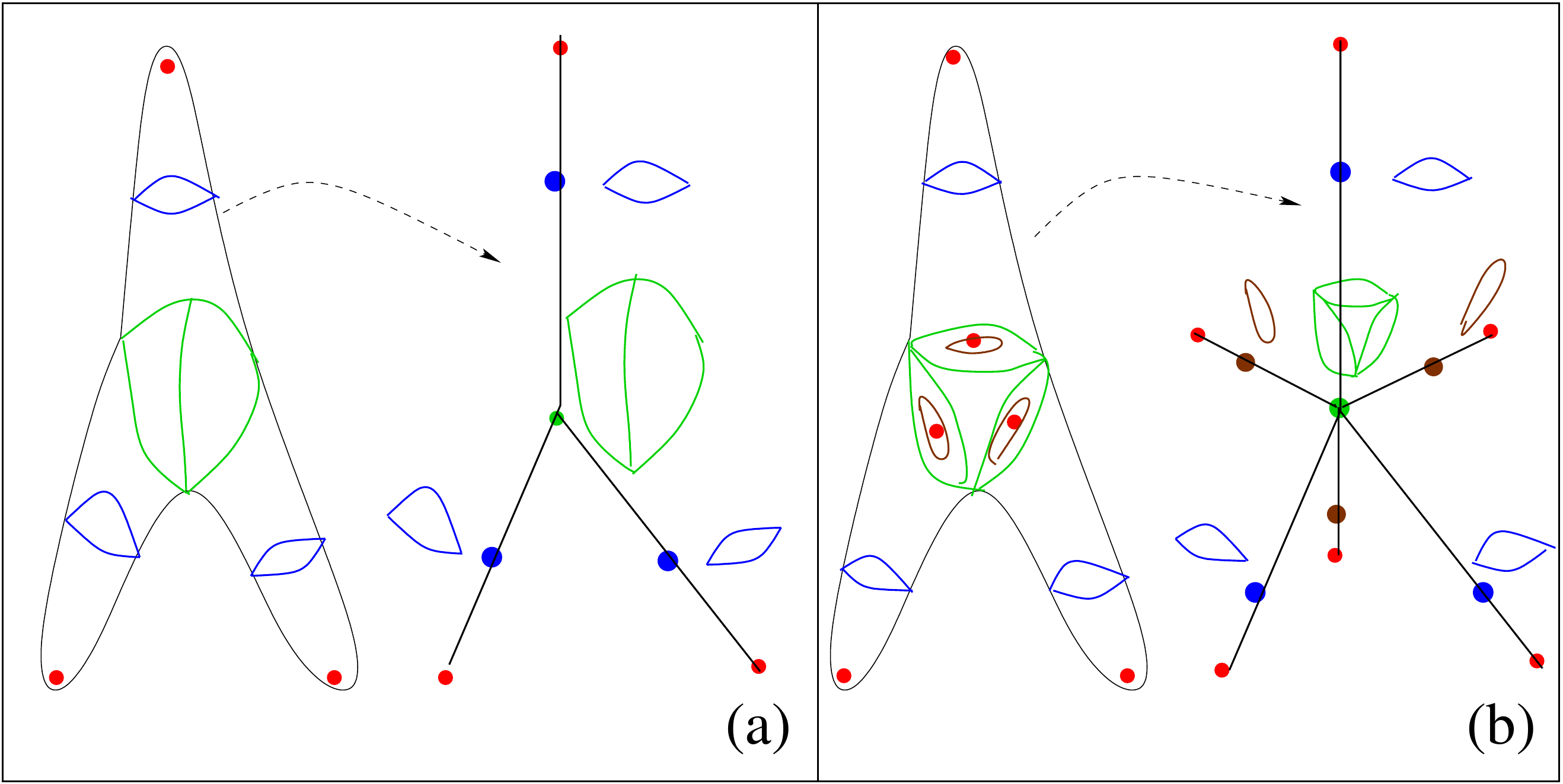}
\caption{Sweeping out a three-legged sphere}
\label{fig:sm3}
\end{figure}

In this construction, the inverse images of the points of the hexapod do not form a continuous family of $1$-cycles. We observe exactly the same discontinuity near the
center of the hexapod as the discontinuity near the center of the tripod for the original map. Of course, this discontinuity disappears once we take
a composition of the map to the hexapod $H$ with  an appropriate continuous finite-to-one map $H\to S^1$.
Here is a description of such a finite-to-one map that will be used below.
First, consider the following map from the hexapod to the interval $[0,1]$ where the center of the hexapod
is mapped to $\frac{1}{2}$, each ray of the first tripod is (linearly) mapped to $[\frac{1}{2},1]$, and each ray
of the second (small) tripod is mapped to $[0,\frac{1}{2}]$. Now, the inverse images of both endpoints of~$[0,1]$ are finite collections of points, while the inverse
image of each point of $(0,1)\setminus \{\frac{1}{2}\}$ is a collection of three simple loops. It is easy to see that the inverse images form
a continuous family of $1$-cycles (for the flat topology). Regarding these $1$-cycles as points in~$\mathcal{Z}_1(M;G)$, we see that both endpoints of $[0,1]$ are mapped
to the zero cycle, and our map can be factorized through $S^1=[0,1]/\{0,1\}$. One can easily check the continuity of the corresponding map of $S^1$ to the space of
$1$-cycles on the three-legged starfish.

Now, we are going to give an example illustrating that not every finite-to one map from the hexapod to~$S^1$ yields a continuous map to
the space of $1$-cycles. Consider a map that sends all vertices of degree~$1$ of the hexapod~$H$ to a point $a\in S^1$, the degree~$6$ vertex of~$H$ to another point
$b\in S^1$, and every edge of~$H$ to the same arc of~$S^1$ connecting $a$ and~$b$. The points of~$S^1$ not in the image of the map $H\to S^1$
 correspond to the zero $1$-cycle, and so is~$a$. However, the image of~$b$ is the $1$-cycle formed of six arcs on the three-legged star-fish.
 Hence, a discontinuity at~$b$.
 \end{example}

%when given integral weight functions that can be different from $1$ or $-1$
%only for $s$ in a subpolyhedron of $S^{n-p}$ of codimension at least $1$. 
%Continuity here can be in the topology of the space of polyhedral $p$-chains (which is a subspace of the space of $p$-currents), but we prefer to strengthen this definition by
%demanding the continuity in the space of $p$-varifolds.

An elaboration of the ideas involved in Example~\ref{ex:fish2} can be combined with our construction in the proof of Theorem~\ref{theo:A'} to turn the family~$h^{-1}(t)$ of $1$-cycles, into a continuous family of $1$-cycles parameterized by~$S^{n-1}$ in the general case.
This leads us to introduce the following definition. (We refer to~\cite{morgan} for a general background in geometric measure theory, including the notions of currents and varifolds.)

\begin{definition} \label{def:Wtilde}
Let $M$ be a closed Riemannian $n$-manifold.
Define
\begin{equation} \label{eq:defWtilde}
\w'_1(M) = \inf_{\raisebox{-3pt}{\scriptsize $\Xi$}} \, \sup_{u \in S^{n-1}} \, \max_{\raisebox{-3pt}{\scriptsize $C \subseteq \Xi_u$}} \length(\varphi_{|C}).
\end{equation}
In this expression, the infimum is taken over the families~$\Xi$ of $1$-cycles (more precisely, of finite collections of closed curves) parameterized by~$S^{n-1}$ on a closed $n$-pseudomanifold~$N$, which are continuous both in the flat topology of the $1$-cycle space and in the weak topology of the $1$-varifold space, and whose image under a degree-one map $\varphi:N \to M$ represents the fundamental class~$[M]$ via the Almgren isomorphism $\pi_{n-1}(\mathcal{Z}_1(N;G)) \simeq H_n(M;G)$, see~\cite{alm}.
Furthermore, the maximum is taken over the connected components~$C$ of~$\Xi_u$, where $u \in S^{n-1}$.

Note that since $\Xi$ is continuous with respect to the weak topology of varifolds, the length of the image of~$\Xi_u$ varies continuously.
Clearly, $\w'_1(M)\geq \w_1(M)$.

\medskip

There is a non-equivalent but equally adequate and more geometric way to define~$\Xi$.
Consider the space~$\Gamma$ of piecewise smooth paths on $M$ endowed with the Lipschitz distance topology. 
A finite collection of piecewise smooth closed curves can be parameterized by $k$-tuples of paths of~$\Gamma$ whose endpoints match to form a $1$-cycle.
The distance between two such collections of closed curves can be defined as the infimum of the Lipschitz distance between their parametrizations as $k$-tuples of paths of~$\Gamma$.
(Here, the integer~$k$ must be the same for both collections of curves and we take the infimum over all~$k$.)
Formally, we also identify two collections of curves that differ by a union of closed curves reduced to finitely many points.
We can modify the definition of~$\w'_1(M)$ by taking the infimum over all families~$\Xi$ of finite collections of piecewise smooth closed curves which represent the fundamental class of~$M$ under the Almgren isomorphism.
\end{definition}

%Second, consider the set $\Xi$ of all finite collections of piece-wise smooth closed curves. To determine the distance between two such $1$-cycles we need to consider all possible parametrizations of both these cycles by $N$-tuples of paths so that the end points of paths are matched to create a $1$-cycle. (Here $N$ must be the same for both paths.) For each pair of such parametrization we can measure the distance between them in $P^N$. Then we take the infimum over all values of $N$ and all pairs of prametrizations by $N$-tuples of paths. Finally, one take a quotient map identifying each $1$-cycle with all unions of this cycle with  finite collections of points regarded as separate closed curves in the collection.

Our estimates (and their proofs) are valid for both choices of~$\Xi$ in the definition of~$\w'_1(M)$.

\medskip

%Define {\it admissible} PL-maps~$\bar h:N \to S^{n-1}$ as maps such that the fibers~$\bar h^{-1}(s)$ form
%a continuous family of $1$-cycles (more precisely, of finite collections of connected closed curves), where the space of $1$-cycles is endowed with the flat topology.
%In particular, this means that we are taking the quotient of our space of curves, where two curves that differ by one or several components
%that are constant curves (\ie, points) are identified.

%Now, define %$\w'_1(M)\geq \w_p(M)$ by formula  
%\begin{equation*} %\label{eq:**}
%\w'_1(M) = \inf_{\raisebox{-3pt}{\scriptsize $\varphi,\bar{h}$}} \, \sup_{t \in S^{n-1}} \,  \max_{\raisebox{-3pt}{\scriptsize %$C \subseteq \bar{h}^{-1}(t)$}} \, \length(\varphi_{|C})
%\end{equation*}
%with the restriction that $\bar h$ must be admissible. 
%Here, $\mass_p$ stands for the $p$-varifold mass, and for sweepouts considered below will just coincide with $vol_p$.
%(Note that using coefficients with absolute value greater than $1$ would
%increase $\w'_p(M)$.)
%Clearly, $\w'_1(M)\geq \w_1(M)$.

The following theorem extends the estimates of Theorem~\ref{theo:A}.

\begin{theorem} \label{theo:A'}
Let $M$ be a closed $n$-manifold.
Then every Riemannian metric on~$M$ satisfies 
\begin{align*}
%\FillRad(M) & \geq c_n \, \w'_1(M) \\
\vol(M) & \geq c_n \, \w'_1(M)^n \\
\diam(M) & \geq c'_n \, \w'_1(M)
\end{align*}
for some explicit positive constants~$c_n$ and $c'_n$ depending only on~$n$.
%If $M$ is a closed Riemannian $n$-manifold, then
%\begin{align*}
%\w'_1(M) & \leq c_1(n) \, \vol(M)^{1\over n} \\
%\w'_1(M) & \leq c_2(n) \, \diam(M)
%\end{align*}
%for some explicit positive constants~$c_1(n)$ and~$c_2(n)$ depending only on~$n$.
\end{theorem}

%\begin{remark} \label{rem:A'}
We will also show that if $n \geq 3$, then this theorem can be somewhat improved by changing the definition of~$\w'_1(M)$. 
In the current definition, we are taking the infimum of the maximal length of the $\varphi$-images of the connected components of a $1$-cycle.
We can instead take the infimum of the maximal length of the connected components of its image under~$\varphi$.
Here is the formal definition of this new invariant

%we are looking at the lengths of %weighted 
%the $\varphi$-images of the connected components of $1$-cycles from $\Xi$. %$\bar h^{-1}(s)$.
%One can further restrict the class
%of admissible maps~$\bar h$ in the definition of $\w'_1(M)$ by demanding 
%One can require that if $C_1$ and~$C_2$ are two different connected components of  a $1$-cycles from $\Xi$, then  %~$\bar h^{-1}(s)$, then $\varphi(C_1)$ and $\varphi(C_2)$ are disjoint. Thus, instead of taking the infimum of the maximal length of %(weighted) 
%the $\varphi$-images of the connected components of a $1$-cycle, %$\bar h^{-1}(s)$, 
%we can take the infimum of the maximal length of %(weighted)
%the connected components of its image under $\varphi$. ~$\varphi(\bar h^{-1}(s))$. 
\begin{equation} \label{eq:Wtilde2}
\w''_1(M)= \inf_{\raisebox{-3pt}{\scriptsize $\Lambda$}} \, \sup_{u \in S^{n-1}} \, \max_{\raisebox{-3pt}{\scriptsize $C \subseteq \Lambda_u$}} \length(C),  
\end{equation}
where the infimum is taken over all families $\Lambda$ of $1$-cycles (more precisely, finite collections of closed curves) parametrized by $S^{n-1}$ on~$M$ (not $N$ as before!), that are continuous in either of the two topologies from the definition of $\w''_1(M)$ and correspond to the fundamental homology class of $M$ under the Almgren isomorphism.

In addition, one can require that $1$-cycle family~$\Lambda$ arises from a slicing of~$N$ as in Definition~\ref{def:Wtilde}.
More precisely, one can assume that the $1$-cycle family~$\Lambda$ is given by the image of the inverse images~$h^{-1}(u)$ of a continuous map $h:N \to S^{n-1}$ defined on a closed $n$-pseudomanifold~$N$ under a degree one map $\varphi:N \to M$ such that the fibers $h^{-1}(u)$ with $u \in S^{n-1}$ define a family of $1$-cycles on~$N$ continuous with respect to both topologies involved in Definition~\ref{def:Wtilde}.
%following additional restriction on the considered (in $\inf$) family of $1$-cycles $\Lambda$: There exists a closed manifold $N$ and a continuous degree one map $\varphi: N\longrightarrow M$ such that $\Lambda$ is the image under a continuous map $\varphi:N\longrightarrow M$ of the family $h^{-1}(u)$, $u\in S^{m-1}$, for a continuous map $h:N\longrightarrow S^{n-1}$ such that 1) all $h^{-1}(u)$ are $1$-cycles; and 2) the family of $1$-cycles $h^{-1}(u)$ on $N$ is continuous with respect to either of two topologies introduced in the definition of $\w'_1(M)$.

So defined invariant~$\w''_1(M)$ has a very natural geometric meaning: It measures the maximal length of a connected component in an optimal sweepout of~$M$ by $1$-cycles, where the sweepouts are parametrized by the sphere of codimension one. 

\medskip

The following estimates also hold for this invariant.

\begin{theorem} \label{theo:A''}
Let $M$ be a closed $n$-manifold, $n\geq 3$.
Then every Riemannian metric on~$M$ satisfies 
\begin{align*}
%\FillRad(M) & \geq c_n \, \w''_1(M) \\
\vol(M) & \geq c_n \, \w''_1(M)^n \\
\diam(M) & \geq c'_n \, \w''_1(M)
\end{align*}
for some explicit positive constants~$c_n$ and $c'_n$ depending only on~$n$.
%If $M$ is a closed Riemannian $n$-manifold, then
%\begin{align*}
%\w'_1(M) & \leq c_1(n) \, \vol(M)^{1\over n} \\
%\w'_1(M) & \leq c_2(n) \, \diam(M)
%\end{align*}
%for some explicit positive constants~$c_1(n)$ and~$c_2(n)$ depending only on~$n$.
\end{theorem}

%, that continuously depend on parameter(s) even when the $1$-chains are being regarded as $1$-varifolds.
%Moreover, as we will see, the integer weights in the definition
%of the integral polyhedral $1$-chains will be always bounded by a constant~$C(n)$ depending only on the dimension of~$M$.
%\end{remark}

As a further comment on the definition of~$\w'_1(M)$, we consider the following example.

\begin{example}
Let $M=N=(S^2,{\rm can})$ and $\varphi:S^2 \to S^2$ be the identity map. 
Consider a very fine triangulation of~$M$. 
Let us construct a map $h:M \to [0,1]$ as follows.
The inverse image of~$0$ under~$h$ is the (finite) collection of the centers of the $2$-simplices of the triangulation. 
When $t$ grows from $0$ to~$1$, the preimage~$h^{-1}(t)$ is a collection of concentric triangles  connecting the center of each simplex to its boundary. When $t=1$, the inverse image of~$t$ is the $1$-skeleton of the triangulation. 
Thus, the fiber~$h^{-1}(t)$ is a $1$-cycle, except at $t=1$.
Still, the map $[0,1) \to \mathcal{Z}_1(M;\Z)$ taking $t \in [0,1)$ to~$h^{-1}(t)$ extends by continuity at $t=1$ by sending $1$ to the zero $1$-cycle.
The map $S^1 \to \mathcal{Z}_1(M;\Z)$ thus-defined is continuous in the flat topology of $1$-cycles (albeit not in the weak topology of $1$-varifolds) and induces the fundamental class of~$M$ under the Almgren isomorphism.
Furthermore, the length of the connected components of this family of $1$-cycles can be arbitrarily small.
This example would tend to show that $\w'_1(M)$ is trivial.
Let us recall however that the map $S^1 \to \mathcal{Z}_1(M;\Z)$ is not continuous in the weak topology of $1$-varifolds and that the $1$-cycle at $t=1$ is not of the form $\varphi(h^{-1}(t))$ for $t=1$, as we changed its value for continuity reasons.
Thus, this $1$-cycle family does not occur in the definition of~$\w'_1(M)$.
Actually, we will show in Theorem~\ref{theo:C} that $\w'_1(M)$ is always positive.
% Though the fiber $h^{-1}(t)$ is a $1$-cycle, the map $S^1 \to \mathcal{Z}_1(M;\Z)$ induced by~$h$ is not continuous due to the discontinuity at the image of~$1$ in $S^1 = [0,1]/\{0,1\}$. (Since $\varphi(h^{-1}(t))$ can be very long, the corresponding family of $1$-cycles does not satisfy the condition of the theorem.) Yet, a small change in the map $S^1\to \mathcal{Z}_1(M;\mathbb{Z}_1)$ makes it continuous with arbitrarily short connected components. It is sufficient to change  this map at just one point of~$S^1$, namely~$1$, by  defining it as the zero cycle. In fact, the same simple construction can be performed for each manifold~$M$. However, the resulting continuous map in $S^1\to \mathcal{Z}_1(M;G)$ is not what we looking for, as it is {\it not} defined as $\varphi(h^{-1}(t))$ for one value of~$t$.
%As we will see below, our $\w_1(M)$ is always positive.
\end{example}

\forget
\begin{remark}
The previous theorem means that every closed Riemannian $n$-manifold $M$ there exist a closed $n$-pseudomanifold~$N$, a degree~$1$ map $\varphi:N\to M$ and a continuous map \mbox{$h:N \to S^{n-1}$} such that
\begin{itemize}
    \item the images $\varphi(h^{-1}(t))$ form a continuous family of polyhedral $1$-cycles on~$N$, and the resulting map $S^{n-1}\to \mathcal{Z}_1(N;G)$
is continuous, where $G=\mathbb{Z}$ or $\mathbb{Z}_2$ depending on the orientability of~$M$. It also sends the fundamental homology class
of $S^{n-1}$ into the fundamental homology class of $M$. The corresponding map of $S^{n-1}$ into the space of $1$-varifolds on $M$ is also continuous;
\item for each $t\in S^{n-1}$ and each connected component $C_i(t)$ of~$h^{-1}(t)$, the length of $\varphi(C_i(t))$ does not exceed $c_n \, \FillRad(M)$ for some~$c_n >0$.
\end{itemize}
A previous remark, see Remark~\ref{rem:A'}, also implies that, if $n\geq 3$, then for each $t\in S^{n-1}$, the
images under $\varphi$ of the different connected components of $h^{-1}(t)$ are disjoint.
\end{remark}
\forgotten

\begin{remark}
Below, we will consider some generalizations and extensions of Theorem~\ref{theo:A}. 
It will be clear that they also hold true for the analogs of~$\w'_1(M)$ and $\w''_1(M)$.
We leave the (rather obvious) details to the reader.
\end{remark}

\subsection{Sweeping out lower-dimensional strata}

In the previous theorems, we consider $(n-1)$-parameter families of polyhedral $1$-chains sweeping out the whole manifold~$M$.
One may wonder whether one can extract $(k-1)$-parameter families of polyhedral $1$-chains sweeping out nontrivial $k$-dimensional homology classes of~$M$ or more generally essential $k$-complexes of~$M$ so that these $1$-chain sweepouts satisfy the same upper bounds as the (full) $1$-sweepout of Theorem~\ref{theo:A}.
Though examples can be found in~\cite{sab}, the existence of such sweepouts may not hold in general.
In our next result, we give topological conditions which ensure the existence of such sweepouts.
The existence of these sweepouts does not follow directly from the proof of Theorem~\ref{theo:A} and requires some new ideas.
In particular, we will need to change our main definition and make various changes in the proof of Theorem~\ref{theo:A}.

\medskip

First, let us introduce a more general notion of sweepout leading to a different notion of waist.

\begin{definition} \label{def:width2}
Let $\Phi:M \to K$ be a continuous map from a closed manifold~$M$ to a CW-complex~$K$.
A \emph{$\Phi$-homotopy $(p,k)$-sweepout} of~$M$ is a family
\[
\varphi[h^{-1}(t)] \subseteq M
\]
with $t \in T$, where $h:X \to T$ is a continuous map from a finite simplicial $(k+p)$-complex~$X$ to a finite simplicial $k$-complex~$T$ such that all fibers~$h^{-1}(t)$ are $p$-subpolyhedra of~$X$, and $\varphi:X \to M$ is a continuous map whose composition $\Phi \circ \varphi: X \to K$ is not homotopic to a map 
\[
X \overset{h}{\to} T \to K
\]
which factors out through~$h$.
Define the \emph{$\Phi$-homotopy $(p,k)$-waist} of a closed Riemannian manifold~$M$ as
\[
\w_{p,k}(M,\Phi) = \inf_{\raisebox{-3pt}{\scriptsize $\varphi,h$}} \sup_{t \in T} \,  \vol_p(\varphi_{|h^{-1}(t)})
\]
where the infimum is taken over all maps $\varphi:X \to M$ and $h:X \to T$  defining a $\Phi$-homotopy $(p,k)$-sweepout of~$M$.
%Here again, the notation $\length \, \varphi[h^{-1}(t)]$ designates the length of the map~$\varphi$ restricted to the graph~$h^{-1}(t) \subseteq N$.
If such sweepouts do not exist, we let \mbox{$\w_{p,k}(M,\Phi)=0$}.

We will be especially interested in the case where~$p=1$.
%Similarly to Definition~\ref{def:width}, 
As in~\eqref{eq:wp}, we can assume that the homotopy $(p,k)$-sweepouts are parameterized by the sphere~$S^k$ and that the $\Phi$-homotopy $(1,k)$-waist of~$M$ is defined by minimizing the maximal length of the map~$\varphi$ restricted to the connected components of the fibers of~$\bar{h}:X \to S^k$. 
\end{definition}

Sweepout estimates also hold with this notion of waist when~$p=1$.

\begin{theorem} \label{theo:B}
Fix $k \leq n-1$.
Let $M$ be a closed $n$-manifold and $\Phi:M \to K$ be a continuous map to a CW-complex~$K$ with $\pi_i(K) = 0$ for every $i \geq k+1$.
Suppose that $\Phi_*([M]) \neq 0 \in H_n(K;G)$ for some homology coefficient group~$G$.
Then every Riemannian metric on~$M$ satisfies 
\begin{align*}
\vol(M) & \geq c_n \, \w_{1,k}(M,\Phi)^n \\
\diam(M) & \geq c'_n \, \w_{1,k}(M,\Phi)
\end{align*}
for some explicit positive constants~$c_n$ and~$c'_n$ depending only on~$n$.
\end{theorem}

The following example illustrates the theorem.

\begin{example}
The main examples arise when $K$ is the Eilenberg-Maclane space $K(G,m)$, where $G$ is an abelian group and $n=qm$.
%$n=qm$, and the $k$-th homology group of $M$ with coefficients in $G$ is nontrivial. Here, $G$ is an abelian group of coefficients,  say, $\mathbb{Z}$ or $\mathbb{Z}_2$. 
Assume that the map $\Phi:M \to K$ represents a nonzero cohomology class
$c\in H^m(M;G)$ such that the $q$-th cup power of~$c$ is nonzero in $H^n(M;G)$. 
%The dimension of $X$ is $k=lm$.
%for some integer $l\geq 1$ which is assumed to be even, if $G\not =\mathbb{Z}_2$, and the $l$th cup power of $c$ is not zero in $H^k(M;G)$. 
Assume also that the map $h:X \to M$ defined on a closed $k$-pseudomanifold~$X$ represents a homology class $a\in H_{k}(M;G)$ dual to a nonzero multiple of the $l$-th cup power of~$c$, where $k=lm$, in the sense that $\langle c^l,a \rangle \neq 0$.
In this case, the map $\Phi \circ h: X \to K$ does not factor through a $(k-1)$-dimensional complex $T$ since it induces  a nontrivial homomorphism between the 
$k$-dimensional homology groups of $X$ and $K$. 
Thus, Theorem~\ref{theo:B} yields curvature-free upper bounds for the homotopy $(1,k)$-waist of some lower-dimensional homology classes of $M$.
%Yet  there exist examples of a different nature. One such example can be found in section 4 (Example 4.2).
\end{example}

%\begin{remark}
%I suspect we can further assume that the graphs $\varphi[h^{-1}(q)]$ are complete of degree at most~$k+2$ in the definition of a sweepout.
%Not sure it's worth the trouble though.
%\end{remark}

\subsection{Intrinsic geometric interpretation of the filling radius}
The filling radius of a closed Riemannian manifold~$M$ is defined in an extrinsic way from the Kuratowski embedding of~$M$ into~$L^\infty(M)$; see~Definition~\ref{def:FR}.
%by embedding~$M$ into~$L^\infty(M)$ through the Kuratowski embedding; see~Definition~\ref{def:FR}.
A different (more intrinsic) interpretation of the filling radius can be deduced from the filling radius estimate of Theorem~\ref{theo:FR1}.
More specifically, we show that the filling radius of a closed Riemannian manifold is roughly equal to its homology $1$-waist.

\begin{theorem} \label{theo:C}
There exist two explicit constant $c_n$ and~$C_n$ depending only on~$n$ such that every closed Riemannian $n$-manifold~$M$ satisfies 
\[
c_n \, \w(M) \leq \FR(M) \leq C_n \, \w(M).
\]
We can take $C_n=\frac{1}{2}$.
\end{theorem}

\begin{remark}
Since $\w(M)\leq \w'_1(M)\leq \w''_1(M)$, we can combine the lower bound in the previous theorem with %Theorem~\ref{theo:A'} (and its strengthening 
Theorem~\ref{theo:A''} to obtain the following alternative (and also imprecise up to a constant factor)
geometric interpretation of the filling radius when $n \geq 3$.
Up to at most a dimensional factor  $c(n)$, the filling radius of a closed Riemannian $n$-manifold~$M$ is equal to the maximal length %(or mass) 
of a connected component in an ``optimal" sweepout of $M$ by 
a continuous family of closed curves. Here, ``optimal" means that the sweepout (nearly) realizes the infimum of the minimal length. When $n=2$, this is still
true, but for a somewhat less geometrically intuitive definition of ``connected components of a sweepout" stemming from the definition of~$\w'_1(M)$.
(In this case, one looks at the images of connected components of a slicing of~$N$ under a degree one map $\varphi:N \to M$.)
% (or mass).
\end{remark}

As shown in Proposition~\ref{prop:HS-FR-UW}, the hypersphericity of a closed orientable Riemannian manifold is roughly bounded by its filling radius (and so by its Urysohn width).
For Riemannian $2$-spheres, these Riemannian invariants are roughly the same; see Section~\ref{sec:HS}.
Still, there are examples of manifolds where the hypersphericity and the Urysohn width can be arbitrarily far apart; see~\cite{guth05}.
Applying the filling radius estimates of Theorem~\ref{theo:C} to these examples, we can strengthen this result by showing that the same occurs between the hypersphericity and the filling radius.

\begin{theorem} \label{theo:cex}
There exists a sequence~$(g_i)$ of Riemannian metrics on~$S^4$ with arbitrarily small hypersphericity and filling radius bounded away from zero.
\end{theorem}

It would be interesting to determine whether we can replace $\w(M)$ with the Urysohn width in Theorem~\ref{theo:C} or whether the filling radius can be arbitrarily far apart from the Urysohn width as in Theorem~\ref{theo:cex}.

\subsection{Homology $p$-waist bounds}

The bounds in Theorem~\ref{theo:A} about homology $1$-waist can be extended to homology $p$-waist using the notion of homological filling function defined below.

\begin{definition} \label{def:FH}
The $k$-homological filling function of a closed Riemannian $n$-manifold~$M$ is a function $\FH_k:[0,\infty) \to [0,\infty]$ defined as
\[
\FH_k(v) = \sup_{\vol_k(\Sigma_0^k) \leq v} \inf \{ \vol_{k+1}(\Sigma^{k+1}) \mid \partial \Sigma^{k+1} = \Sigma_0^k \}
\]
where the supremum is taken over all closed $k$-pseudomanifolds~$\Sigma_0^k$ in~$M$ of volume at most~$v$ and the infimum is taken over all compact pseudomanifolds~$\Sigma^{k+1}$ in~$M$ with boundary~$\partial \Sigma^{k+1} = \Sigma_0^k$.
By convention, $\inf \emptyset = \infty$. This means that if the $k$-th homology group of $M$ is nontrivial, then $\FH_k(v)=\infty$ for all $v$ greater than some $v_0$. 
This threshold value $v_0$ can, however, be arbitrarily large in comparison with, say, $\vol(M)^{k\over n}$.
For example, consider $M=S^1\times S^2$ endowed with the product metric,
where $S^1$ has a very large length $L$, but the area of $S^2$ is just $1\over L$. Although $H_1(M)$ is nontrivial, $\FH_1(v)$ will be finite
for all $v<L$.

Observe that the homological filling function~$\FH_k$ is nondecreasing and that $\FH_k(v) \leq \alpha_k \, v^{k+1}$ for every $v$ small enough, where $\alpha_k$ is some constant depending only on~$k$ involved in the isoperimetric inequality.
It will be convenient to introduce $\FHb_k(v) = \FH_k(2(k+1) v)$.
Also, we define $\FH_k(\infty)=\FHb_k(\infty)=\infty$.
\end{definition}

The notion of homological filling function was considered in~\cite{NR06} to bound the least area of a (possibly singular) minimal surface and, more generally, the least mass of a nontrivial stationary integral $k$-varifold in a closed Riemannian manifold whose first $k-1$ homology groups are trivial.
The homological filling functions can be estimated by taking a simplicial approximation and minimizing the volume of a filling using a elementary linear algebra argument in connection with systems of linear equations with integer coefficients given by the boundary operator; see~\cite{NR06} for more details.
Strictly speaking the definition of homological filling functions in~\cite{NR06} was stated in terms of singular chains and not pseudomanifolds, but this leads to the same notion after desingularization.

\medskip

The following result provides an extension of Theorem~\ref{theo:A} to higher dimensional homology waists; see~\eqref{eq:Wp}.

\begin{theorem} \label{theo:E}
Let $M$ be a closed Riemannian $n$-manifold.
Then, for every positive integer~$p$,
\begin{align*}
\w_p(M) & < \frac{1}{2^{n-p+1}} \textstyle{\binom{n+1} {p}}^{-1} \, \FHb_{p-1} \circ \cdots \circ \FHb_1(C_n \, \vol(M)^\frac{1}{n}) \\
\w_p(M) & < \frac{1}{2^{n-p+1}} \textstyle{\binom{n+1} {p}}^{-1} \, \FHb_{p-1} \circ \cdots \circ \FHb_1(C'_n \, \diam(M))
\end{align*}
for some explicit positive constants~$C_n$ and~$C'_n$ depending only on~$n$.
\end{theorem}

The constants involved in the filling functions can be improved by following Remark~\ref{rem:improve}.
Note, that if all homology groups $H_i(M)$ for $i\in\{1,\ldots, p-1\}$
vanish, then the right-hand sides in both inequalities are always finite. However, if at least one of these homology groups is nontrivial, it is possible that one or both right-hand sides are $\infty$, and the inequality(ies) become trivial. \\

\noindent {\it Acknowledgment.}
This research has been partially supported by NSERC Discovery Grants RGPIN-2017-06068 and RGPIN-2018-04523 of the first two authors.
The third author would like to thank the Fields Institute and the Department of Mathematics at the University of Toronto for their hospitality where a large part of this work was done.

\section{Natural sweepouts of the standard cubical simplex} \label{sec:simplex}

In this section, we describe natural $p$-sweepouts of the standard cubical simplex defined as the fibers of a map from the cube to a complex of codimension~$p$. 
The reason we consider the standard cubical simplex instead of the standard simplex is because it is simpler to describe the natural sweepouts in this case. 
Similar constructions hold for the standard simplex.

\medskip

Let us start by describing a decomposition of the standard cube.
The case when $p=1$ is considered in Sections \ref{sec:FR1}, \ref{sec:FR1bis}, \ref{sec:FRhomotopy}, \ref{sec:UW} and~\ref{sec:HS}, while the general case is considered in Section~\ref{sec:pwidth}.
At first reading, one can assume that $p=1$.

\begin{definition}
Let $C^{n+1}=[-1,1]^{n+1}$ be the standard cubical $(n+1)$-simplex.
Fix an integer $p \geq 1$.
The $p$-skeleton~$(C^{n+1})^{(p)}$ of~$C^{n+1}$ is formed of the points of~$C^{n+1}$ all of whose coordinates except possibly $p$ of them are equal to~$\pm 1$.
That is,
\[
(C^{n+1})^{(p)} = \{ x \in C^{n+1} \mid \mbox{there exist } i_1,\dots,i_{n-p+1} \in \{1,\dots,n+1\} \mbox{ distinct such that } x_{i_k} = \pm 1 \}.
\]
The cubical $(n-p)$-complex~${Z}^{n-p} \subseteq C^{n+1}$ dual to~$(C^{n+1})^{(p)}$ is formed of the points of~$C^{n+1}$ with at least $p+1$ zero coordinates; see Figure~\ref{fig:1}.(c) and Figure~\ref{fig:3}.(c).
That is,
\[
{Z}^{n-p} = \{ x \in C^{n+1} \mid \mbox{there exist } i_1,\dots,i_{p+1} \in \{1,\dots,n+1\} \mbox{ distinct such that } x_{i_k} = 0 \}.
\]
Fix $\varepsilon \in (0,1)$.
The space
\[
{X}_{1,\varepsilon}^{n+1} = \{ x \in C^{n+1} \mid \mbox{there exist } i_1,\dots,i_{n-p+1} \in \{1,\dots,n+1\} \mbox{ distinct such that }  |x_{i_k}| \geq \varepsilon \}
\]
formed of the points of~$C^{n+1}$ with at most $p$ coordinate less than~$\varepsilon$ in absolute value is a tubular neighborhood of~$(C^{n+1})^{(p)}$.
Similarly, the space
\[
{X}_{2,\varepsilon}^{n+1} = \{ x \in C^{n+1} \mid \mbox{there exist } i_1,\dots,i_{p+1} \in \{1,\dots,n+1\} \mbox{ distinct such that } |x_{i_k}| \leq \varepsilon \}
\]
formed of the points of~$C^{n+1}$ with a least $p+1$ coordinates bounded by~$\varepsilon$ in absolute value is a tubular neighborhood of~${Z}^{n-p}$.
Both spaces~${X}_{1,\varepsilon}^{n+1}$ and~${X}_{2,\varepsilon}^{n+1}$ are endowed with a cubical structure, where the cubical simplices are bounded by the hyperplanes $x_i=\pm \varepsilon$ and $x_i= \pm 1$.
The cubical complexes ${X}_{1,\varepsilon}^{n+1}$ and~${X}_{2,\varepsilon}^{n+1}$ cover the cube~$C^{n+1}$ and intersect along a cubical $n$-complex 
\[
{Y}^n_\varepsilon = {X}_{1,\varepsilon}^{n+1} \cap {X}_{2,\varepsilon}^{n+1}
\]
which decomposes into a disjoint union
\begin{align}
{Y}^n_\varepsilon =  \sqcup_{k=p+1}^{n+1}  \{ x \in C^{n+1} \mid \mbox{there exist } i_1,\cdots,i_{k} \in \{1,\dots,n+1\} & \mbox{ distinct such that } \nonumber \\
 &  |x_{i_1}| \leq \varepsilon, \cdots, |x_{i_p}| \leq \varepsilon  \nonumber \\
 & |x_{i_{p+1}}| = \cdots = |x_{i_k}| = \varepsilon \nonumber \\
 & |x_i| > \varepsilon \mbox{ for every } i \neq i_1,\cdots, i_k \} \label{eq:Y}
\end{align}
of cubical $(n+p+1-k)$-complexes with $p+1 \leq k \leq n+1$.

Strictly speaking, the complexes ${X}_{i,\varepsilon}^{n+1}$, ${Y}^n_\varepsilon$ and~${Z}^{n-p}$ depend also on~$p$.
In order not to burden the notations, we keep the dependence on~$p$ of these complexes and the following constructions implicit.
\end{definition}

Let us define a natural $p$-sweepout of the standard cube.

\begin{definition} 
Let $\lambda_\varepsilon:[-1,1] \to [-1,1]$ be the odd piecewise linear function defined by
\[
\lambda_\varepsilon(t) =
\begin{cases}
\frac{t-\varepsilon}{1-\varepsilon} & \mbox{ if } t \in [\varepsilon,1] \\
0 & \mbox{ if } t \in [-\varepsilon,\varepsilon] \\
\frac{t+\varepsilon}{1-\varepsilon} & \mbox{ if } t \in [-1,-\varepsilon]
\end{cases}
\]
keeping $-1$, $0$, $1$ fixed and sending $[-\varepsilon,\varepsilon]$ to~$\{ 0 \}$.

Consider the map ${\theta}_\varepsilon:{Y}_\varepsilon^n \to {Z}^{n-p}$ defined by
\begin{equation} \label{eq:theta}
{\theta}_\varepsilon(x_1,\cdots,x_{n+1}) = (\lambda_\varepsilon(x_1),\cdots,\lambda_\varepsilon(x_{n+1})).
\end{equation}
By definition, every point of~${Y}_\varepsilon^n$ has at least $p+1$ coordinates bounded by~$\varepsilon$ in absolute value which are sent to~$0$ by~$\lambda_\varepsilon$.
This shows that the map~${\theta}_\varepsilon$ takes values in~${Z}^{n-p}$.
The preimage of~$z \in {Z}^{n-p}$ under~${\theta}_\varepsilon$ can be determined as follows.
Denote by $z_{i_1},\dots,z_{i_k}$ all the zero coordinates of~$z$.
Note that $k \geq p+1$.
By construction,
\begin{align*}
{\theta}_\varepsilon^{-1}(z) = \{ x \in {Y}_\varepsilon^n \mid \, & |x_{i_{s_1}}| \leq \varepsilon, \cdots, |x_{i_{s_p}}| \leq \varepsilon \mbox{ for } s_1,\dots,s_p \in \{1,\dots,k \} \mbox{ distinct} \\
 & |x_{i_s}| = \varepsilon \mbox{ for every } s \neq s_1,\dots,s_p \mbox{ in } \{1,\dots,k \} \\
 & x_j = \lambda_\varepsilon^{-1}(z_j) \mbox{ for every } j \neq i_1,\dots,i_k \}.
\end{align*}
Since ${\theta}_\varepsilon^{-1}(z)$ lies in~${Y}_\varepsilon^n$, all the coordinates $x_{i_1},\dots,x_{i_k}$ of $x \in {\theta}_\varepsilon^{-1}(z)$ are equal to~$\pm \varepsilon$, except possibly~$p$ of them.
Thus, the preimage~${\theta}_\varepsilon^{-1}(z)$ is a cubical $p$-complex isomorphic to the $p$-skeleton of the $k$-cube.
See Figure~\ref{fig:1} for a description of the sweepout of~$Y^n_\varepsilon$.

\begin{figure}[!htbp]
\centering
\includegraphics[width=12cm]{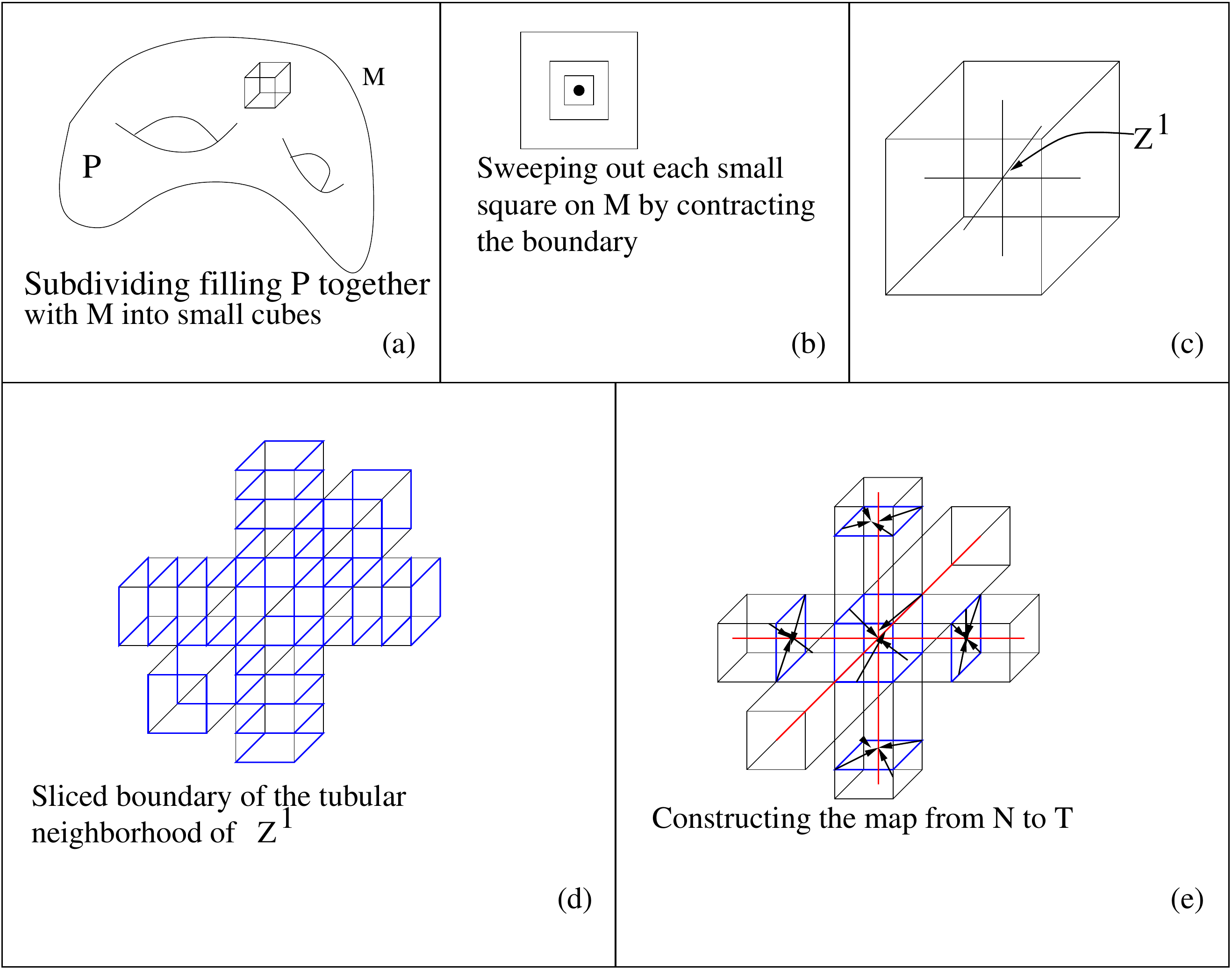}
\caption{Sweeping out $2$-dimensional
~$N$}%$Y^2_\varepsilon$}
\label{fig:1}
\end{figure}

%\medskip

The other map~$\Theta$ we need to define will be used only on~$C^n$.
For this reason, we carry our construction on~$C^n$ and not on~$C^{n+1}$.
The subsets~${Y}_\varepsilon^{n-1}$ with $\varepsilon \in (0,1)$ foliate $C^{n} \setminus ((C^{n})^{(p)} \cup {Z}^{n-p-1})$.
More precisely, they are the level sets of the continuous map
\[
\begin{array}{rccc}
{\Theta}: & C^{n} \setminus ((C^{n})^{(p)} \cup {Z}^{n-p-1}) & \longrightarrow & {Z}^{n-p-1} \times (0,1) \\
 & x & \longmapsto & ({\theta}_\varepsilon(x),\varepsilon)
\end{array}
\]
where $\varepsilon$ is given by $x \in {Y}_\varepsilon^{n-1}$.
Thus, the map~${\Theta}$ is given by~${\theta}_\varepsilon$ on each subset~${Y}_\varepsilon^{n-1}$.

Define the simplicial $(n-1)$-complex
\[
\hat{{Z}}^{n-p} = {\rm Cone}({Z}^{n-p-1}) = {Z}^{n-p-1} \times [0,1]/ {Z}^{n-p-1} \times \{ 1 \}
\]
where ${Z}^{n-p-1} \times \{ 1 \}$ is collapsed to a point~$\star$.

The map~${\Theta}$ extends to a continuous map still denoted by
\begin{equation} \label{eq:fhat}
{\Theta}:C^{n} \to \hat{{Z}}^{n-p}
\end{equation}
where ${\Theta}(x)=({\theta}_0(x),0)$ for every $x \in {Z}^{n-p-1} = {Y}_0^{n-1}$ and ${\Theta}(x)=\star$ for every $x \in (C^{n})^{(p)}$.

The fibers of~${\Theta}$ define a \emph{natural $p$-sweepout} of the $n$-cube which is invariant by the group of symmetries of~$C^{n}$.

\begin{remark}
Loosely speaking, when $n=3$ and $p=1$, the boundary of the neighborhood of the dual to the~$1$ skeleton of the $3$-cube varies with~$\varepsilon$ from two extremes where it collapses to the $1$-skeleton of the $3$-cube or its dual~$Z^1$.
Deforming the slicing of~$Y^2_\varepsilon$ as $\varepsilon$ varies induces a natural sweepout of the $3$-cube; see Figure~\ref{fig:4}.(b).
\end{remark}

\begin{figure}[!htbp]
\centering
\includegraphics[width=14cm]{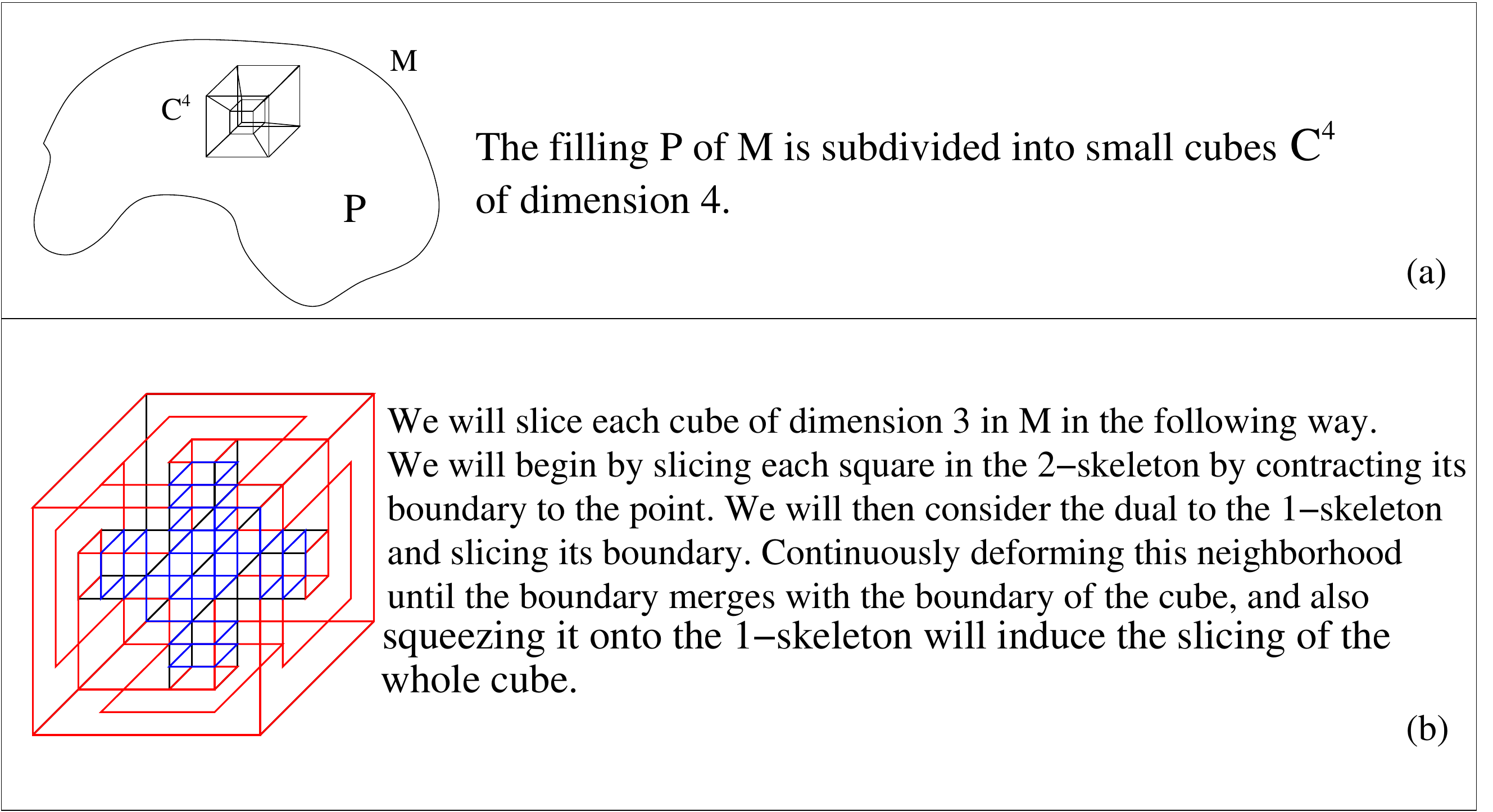}
\caption{Filling and sweeping-out $3$-dimensional $M$}
\label{fig:4}
\end{figure}

\begin{figure}[!htbp]
\centering
\includegraphics[width=14cm]{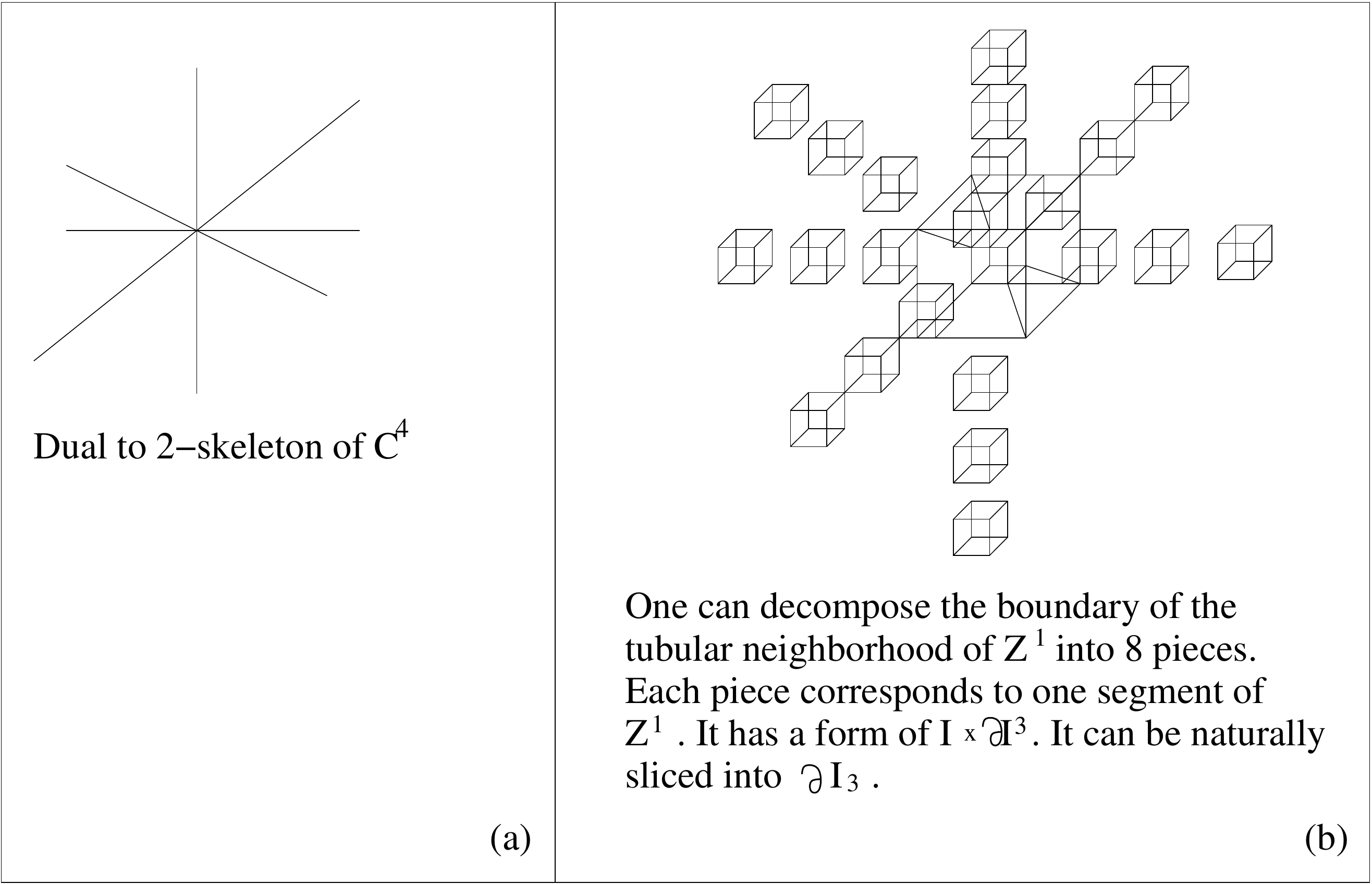}
\caption{Sweeping out the $4$-cube; $p$=2}
\label{fig:5}
\end{figure}

\begin{figure}[!htbp]
\centering
\includegraphics[width=14cm]{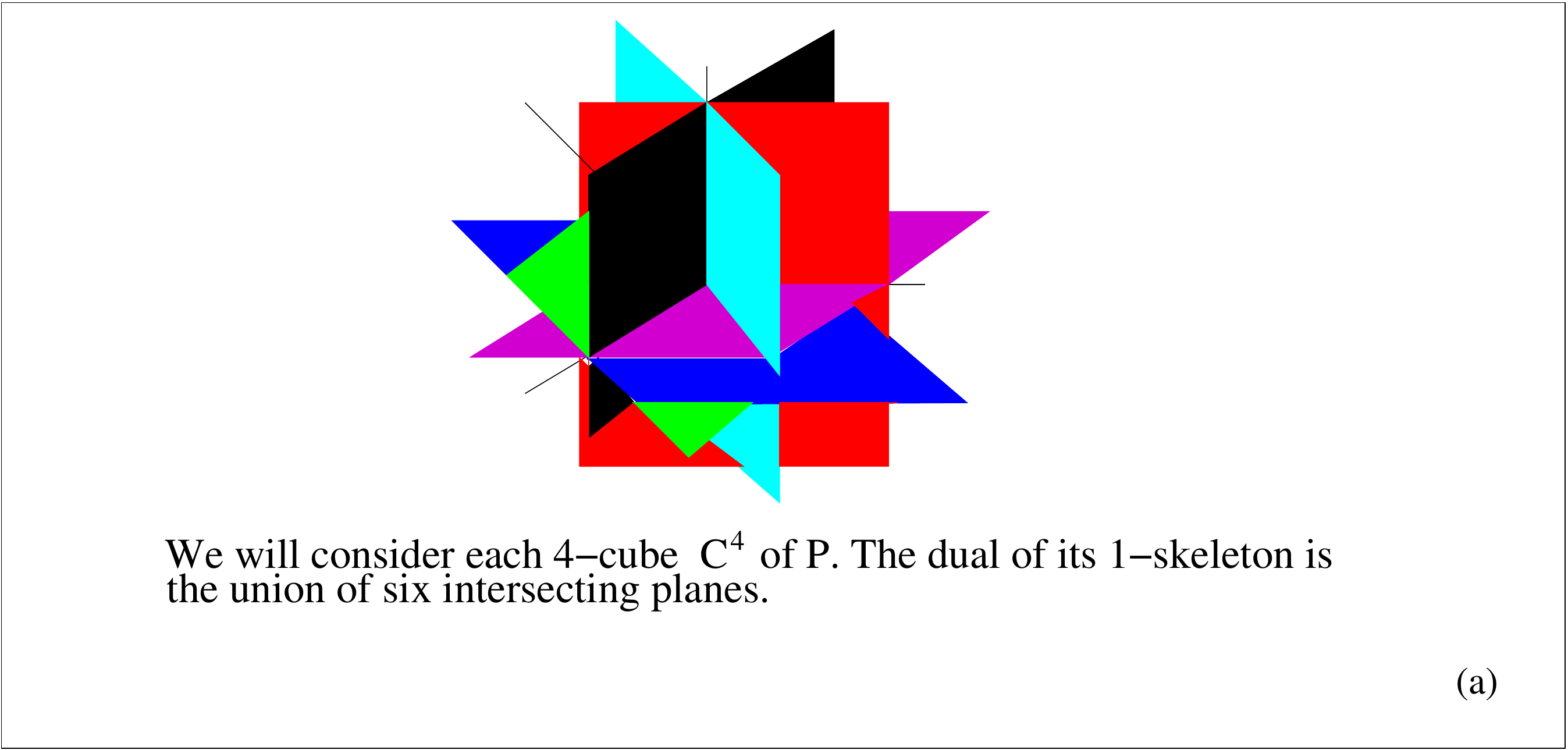}
\includegraphics[width=14cm]{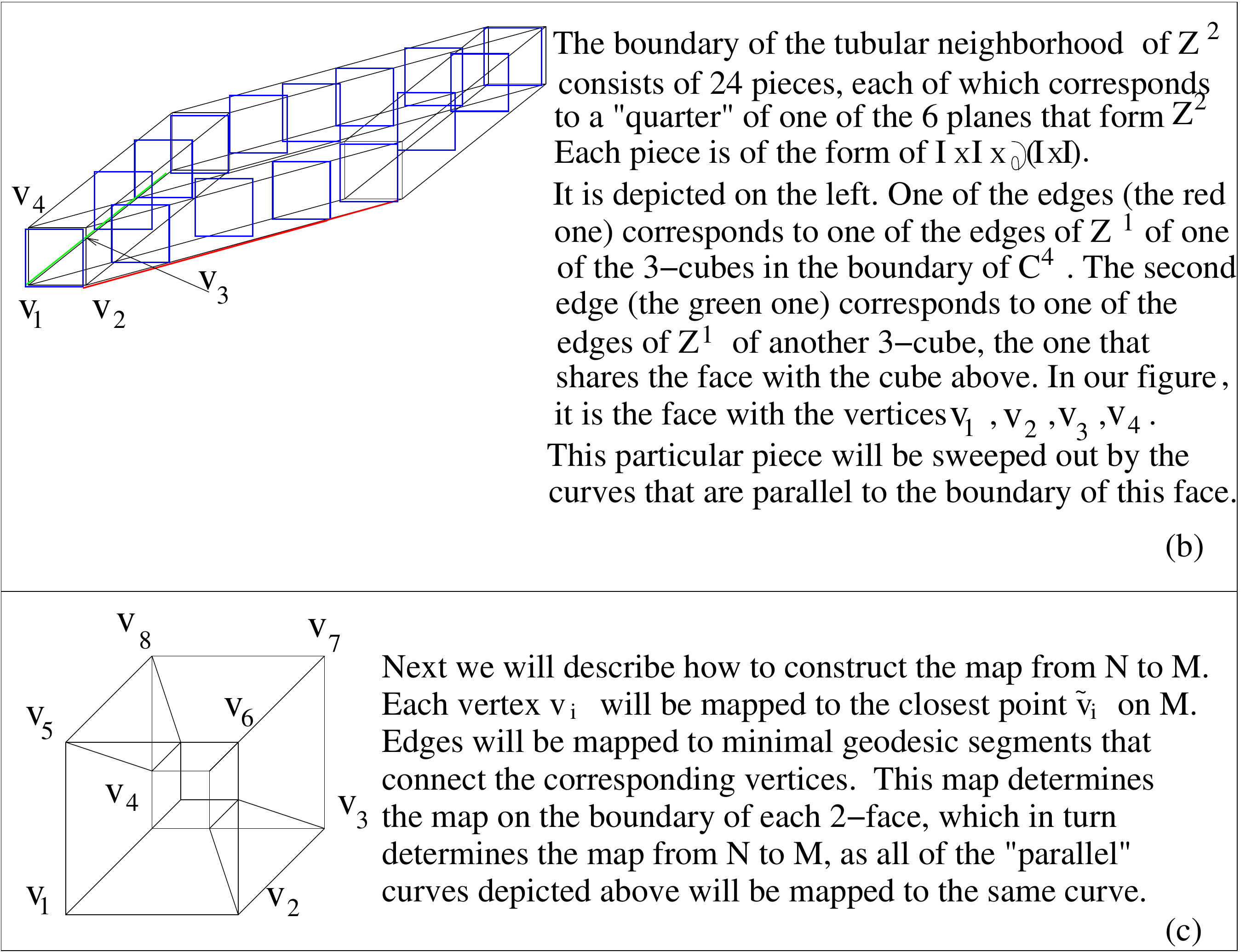}
\caption{Sweeping out the $4$-cube; dim $M=3$}
\label{fig:6}
\end{figure}

\end{definition}

Finally, let us define some deformations on the standard cube.

\begin{definition}
Let $\mu_\varepsilon:[-1,1] \to [-1,1]$ be the odd piecewise linear function defined by
\[
\mu_\varepsilon(t) =
\begin{cases}
1 & \mbox{ if } t \in [\varepsilon,1] \\
\frac{t}{\varepsilon} & \mbox{ if } t \in [-\varepsilon,\varepsilon] \\
-1 & \mbox{ if } t \in [-1,-\varepsilon]
\end{cases}
\]
sending $[\varepsilon,1]$ to~$\{ 1 \}$ and $[-1,-\varepsilon]$ to~$\{ -1 \}$.

Consider the map ${\rho}_\varepsilon:{X}_{1,\varepsilon}^{n+1} \to (C^{n+1})^{(p)}$ defined by
\begin{equation} \label{eq:rho}
{\rho}_\varepsilon(x) = (\mu_\varepsilon(x_1),\cdots,\mu_\varepsilon(x_{n+1})).
\end{equation}
By definition, every point of~${X}_{1,\varepsilon}^{n+1}$ has at least $n-p$ coordinates bounded below by~$\varepsilon$ in absolute value, which are sent to~$\pm 1$ by~$\mu_\varepsilon$.
This shows that the map~${\rho}_\varepsilon$ takes values in~$(C^{n+1})^{(p)}$.
Observe also that the map~${\rho}_\varepsilon$ fixes the vertices of~$C^{n+1}$ and sends every edge of~$C^{n+1}$ to itself.
We will refer to~${\rho}_\varepsilon$ as the ``retraction" of~${X}_{1,\varepsilon}^{n+1}$ onto~$(C^{n+1})^{(p)}$.
Note that the ``retraction"~${\rho}_\varepsilon$ extends to a degree one map 
\begin{equation} \label{eq:rho-bar}
\bar{{\rho}}_\varepsilon:C^{n+1} \to C^{n+1}.
\end{equation}
\end{definition}

In the following sections, we will fix~$\varepsilon = \frac{1}{2}$ and drop the subscripts.
For instance, we will write $Y^n$ for~$Y_\varepsilon^n$, $X_i^{n+1}$ for~$X_{i,\varepsilon}^{n+1}$, $\theta$ for~$\theta_\varepsilon$ and $\bar{\rho}$ for $\bar{\rho}_\varepsilon$.  

\medskip

We conclude this section with the following result.

\begin{proposition} 
The cubical $n$-complex $Y_\varepsilon^n$ is an $n$-pseudomanifold with boundary lying in~$\partial C^{n+1}$.
\end{proposition}

First recall the general definition of a pseudomanifold.

\begin{definition} \label{def:pseudomanifold}
An $n$-pseudomanifold with boundary is a simplicial $n$-complex~$P$ such that
\begin{itemize}
\item every simplex of~$P$ is a face of some $n$-simplex of~$P$;
\item every $(n-1)$-simplex of~$P$ is the face of at most two $n$-simplices of~$P$;
\item given two $n$-simplices of~$P$, there exists a sequence of $n$-simplices of~$P$ with two consecutive $n$-simplices having an $(n-1)$-face in common that starts at one of them and ends at the other.
\end{itemize}
The boundary~$\partial P$ of an $n$-pseudomanifold~$P$ is the simplicial $(n-1)$-subcomplex of~$P$ formed of the $(n-1)$-simplices of~$P$ which are the faces of exactly one $n$-simplex of~$P$.
\end{definition}

\begin{proof}
The cubical $n$-complex~$Y_\varepsilon^n$ decomposes into a union of cubical $n$-simplices of two kinds
\[
[-\varepsilon,\varepsilon]^p \times  \{ \pm \varepsilon \} \times [\varepsilon,1]^{n-p} \quad \mbox{and } \quad [-\varepsilon,\varepsilon]^p \times  \{ \pm \varepsilon \} \times [-1,-\varepsilon]^{n-p}
\]
up to factor permutations; see~\eqref{eq:Y}.
The boundary of the first kind of cubical $n$-simplices is a union of $(n-1)$-faces of the following three forms
\[
[-\varepsilon,\varepsilon]^{p-1} \times \{ \pm \varepsilon \} \times  \{ \pm \varepsilon \} \times [\varepsilon,1]^{n-p}, %\quad
[-\varepsilon,\varepsilon]^p \times  \{ \pm \varepsilon \} \times \{ \varepsilon \} \times [\varepsilon,1]^{n-p-1}, %\quad
[-\varepsilon,\varepsilon]^p \times  \{ \pm \varepsilon \} \times  \{ 1 \} \times [\varepsilon,1]^{n-p-1}
\]
up to factor permutations.
The same holds for the boundary of the second kind of cubical $n$-simplices.
The $(n-1)$-faces involved in these unions have exactly two components of the form of a singleton  $\{ \pm \varepsilon \}$ or $ \{ \pm 1 \}$, with at most one singleton of the form~$\{ \pm 1 \}$.
The $(n-1)$-faces with no singleton of the form~$\{ \pm 1 \}$ appear in the boundary decomposition of exactly two cubical $n$-simplices.
For instance, $[-\varepsilon,\varepsilon]^{p-1} \times \{ \varepsilon \} \times  \{ - \varepsilon \} \times [\varepsilon,1]^{n-p}$ appears in the boundary decomposition of $[-\varepsilon,\varepsilon]^p \times \{ - \varepsilon \} \times [\varepsilon,1]^{n-p}$ and $[-\varepsilon,\varepsilon]^{p-1} \times  \{ \varepsilon \} \times [-\varepsilon,\varepsilon] \times [\varepsilon,1]^{n-p}$.
(Recall that every point in an open $n$-face of~$Y_\varepsilon^n$ has exactly $p+1$ coordinates bounded by~$\varepsilon$ in absolute value.)
The $(n-1)$-faces with exactly one singleton of the form~$\{ \pm 1 \}$ lie in~$\partial C^{n+1}$ and appear in the boundary decomposition of exactly one cubical $n$-simplex.
For instance, $[-\varepsilon,\varepsilon]^p \times \{ \varepsilon \} \times  \{ 1 \} \times [\varepsilon,1]^{n-p-1}$ appears in the boundary decomposition of $[-\varepsilon,\varepsilon]^p \times \{ \varepsilon \} \times [\varepsilon,1]^{n-p}$.
Furthermore, each $k$-face of~$Y_\varepsilon^n$ lies in a cubical $n$-simplex of~$Y_\varepsilon^n$.
Also, it is not difficult to see that given two cubical $n$-simplices of~$Y_\varepsilon^n$, there exists a sequence of cubical $n$-simplices of~$Y_\varepsilon^n$ with two consecutive cubical $n$-simplices having one $(n-1)$-face in common.
This shows that $Y_\varepsilon^n$ is a cubical $n$-pseudomanifold with boundary lying in~$\partial C^{n+1}$.
\end{proof}

\section{Filling radius and homology $1$-waist} \label{sec:FR1}

We establish a lower bound on the filling radius of a closed Riemannian manifold in terms of its homology $1$-waist and derive Theorem~\ref{theo:A}.

\medskip

Let us recall the notion of filling radius introduced by M.~Gromov in~\cite{gro83} to established systolic inequalities on essential manifolds.

\begin{definition} \label{def:FR}
Let $M$ be a closed $n$-manifold with a Riemannian metric~$g$.
Denote by~$d_g$ the distance on~$M$ induced by the Riemannian metric~$g$.
The map 
\[
i:(M,d_g) \hookrightarrow (L^\infty(M),|| \cdot ||)
\]
defined by $i(x)( \cdot ) = d_g(x, \cdot)$ is an embedding from the metric space~$(M,d_g)$ into the Banach space $L^\infty(M)$ of bounded functions on~$M$ endowed with the sup-norm~$|| \cdot ||$.
This natural embedding, also called the Kuratowski embedding, is an isometry between metric spaces.
We will consider $M$ isometrically embedded into~$L^\infty(M)$.

The \emph{filling radius} of~$M$ with a Riemannian metric~$g$, denoted by $\FR(M)$, is the infimum of the positive reals~$\nu$ such that 
\[
(i_\nu)_*([M]) = 0 \in H_n(U_\nu(M))
\]
where $i_\nu: M \hookrightarrow U_\nu(M)$ is the inclusion into the $\nu$-neighborhood of~$M$ in~$L^\infty(M)$, and $[M] \in H_n(M)$ is the fundamental class of~$M$. % with respect to some homology coefficients.
Unless specified otherwise, 
Here, the homology coefficients are in~$\Z$ if $M$ is orientable, and in~$\Z_2$ otherwise.
%Later, we will also consider the \emph{rational filling radius} of~$M$, denoted by~$\FR_\Q(M)$, with homology coefficients in~$\Q$.
\end{definition}

The filling radius of a Riemannian manifold satisfies the following fundamental bounds respectively obtained by M.~Gromov~\cite{gro83} and M.~Katz~\cite{katz83} with an improvement in the constant recently obtained in \cite{nab} .

\begin{theorem}[see \cite{gro83}, ~\cite{katz83}, \cite{nab}] \label{theo:FR}
Let $M$ be a closed Riemannian $n$-manifold.
Then
\begin{align*}
\FR(M) & \leq n \, \vol(M)^\frac{1}{n} \\
\FR(M) & \leq \tfrac{1}{3} \, \diam(M)
\end{align*}
where $c_n$ is an explicit constant depending only on~$n$.
\end{theorem}

The homology $1$-waist is related to the filling radius as follows.

\begin{theorem} \label{theo:FR1}
Let $M$ be a closed $n$-manifold.
Then every Riemannian metric on~$M$ satisfies 
\[
\FR(M) \geq c_n \, \w(M)
\]
for $c_n = \frac{1}{(n+1) \, 2^{n+1}}$.
\end{theorem}

\forget
\begin{remark}
The proof of Theorem~\ref{theo:FR1} actually works with rational coefficients and provides the stronger bound $\FR_\Q(M) \geq c_n \, \w(M)$ when $M$ is orientable; see~\eqref{eq:FRQ-FR}.
\end{remark}
\forgotten

\begin{proof}
We will work with cubical complexes instead of simplicial complexes and rely on the construction of Section~\ref{sec:simplex}.
By definition of the filling radius, the fundamental class~$[M]$ of~$M$ vanishes in the $\nu$-neighborhood~$U_\nu (M)$ of~$M$ in~$L^\infty(M)$, where $\nu > \FillRad(M)$ is very close to $\FillRad(M)$.
Therefore, there exists a compact cubical $(n+1)$-pseudomanifold $P \subseteq U_\nu(M)$ with boundary $\partial P =M$.
Subdivide~$P$ so that every cubical $(n+1)$-simplex of~$P$ has at most one $n$-face in~$\partial P$.

\medskip

Suppose that $\nu < \frac{1}{(n+1) \, 2^{n+1}} \, \w(M)$.
The usual argument to obtain a contradiction and derive a lower bound on the filling radius of~$M$ consists in constructing a retraction from~$P$ onto $\partial P=M$.
However, this may not be possible in our case.
Instead, we will construct a continuous extension~$\bar{f}$ of the identity map on~$M$ to a cubical $(n+1)$-complex containing an $n$-pseudomanifold~$N$ homologuous to~$M$ taking the fundamental classes of~$M$ and~$N$ to different homology classes in~$M$.
This will lead to a contradiction as wanted.

%Instead, we will construct a continuous extension $f':P' \to M$ of the identity map to a different cubical $(n+1)$-pseudomanifold~$P'$ with $\partial P' = \partial P = M$.

\medskip

It will be convenient to think of~$P$ as an abstract compact cubical $(n+1)$-pseudomanifold related to~$M$ through a continuous map $\sigma:P \to U_\nu(M)$ whose restriction $\sigma: \partial P \to M$ to~$\partial P$ satisfies
\begin{equation} \label{eq:sigma}
\sigma_*([\partial P]) = [M] \in H_n(M)
%\footnote[2]{When working with rational coefficients in case $M$ is orientable, the map $\sigma: \partial P \to M$ represents a nonzero multiple of the fundamental class of~$M$, that is, $\sigma_*([\partial P]) = d \, [M]$ for some nonzero integer~$d$}.
\end{equation}
Deforming the map~$\sigma$, we can assume that $\sigma$ takes every edge of~$\partial P$ to a minimizing segment of~$M$.
Denote by~$P^k$ the $k$-skeleton of~$P$.
Subdividing $P$ if necessary, we can further assume that the images by~$\sigma$ of the cycles of the natural sweepout of the cubical simplices of~$\partial P$ are of length  less than $\varepsilon < \min \{ \w(M)-2\nu, \frac{1}{2} \inj(M) \}$; see Section~\ref{sec:simplex} where $p=1$.

\medskip

We first define a map $f:P^0 \to M$ which agrees with~$\sigma$ on the vertices of~$\partial P$, by sending each vertex $p_i \in P^0$ to a nearest point of~$\sigma(p_i)$ in~$M \subseteq L^\infty(M)$, as we wish.
Since the inclusion $i:M \hookrightarrow U_\nu(M)$ is distance-preserving, every pair $p_i$,~$p_j$ of adjacent vertices of~$P$ satisfies
\[
d_M(f(p_i),f(p_j)) \leq d_{L^\infty}(f(p_i),\sigma(p_i)) + d_{L^\infty}(\sigma(p_i),\sigma(p_j)) + d_{L^\infty}(\sigma(p_j),f(p_j)) < \delta
\]
with $\delta=2 \nu + \varepsilon < \frac{1}{(n+1) \, 2^{n}} \, \w(M)$.
We extend the map $f$ to~$P^1$ by taking the edges of~$P$ to minimizing segments joining the images of their endpoints, as we wish. 
Observe that the map~$f$ agrees with~$\sigma$ on the edges of~$\partial P$.
By construction, the lengths of the images of the edges of~$P^1$ are less than~$\delta$.

\medskip

Let $Q \subseteq P$ be the neighborhood of~$P^1$ in~$P$ composed of the pieces $X_1^{n+1} \subseteq C^{n+1}$ corresponding to the cubical $(n+1)$-simplices of~$P$; see Section~\ref{sec:simplex} where $p=1$.
Put together, the ``retractions" $\rho:X_1^{n+1} \to (C^{n+1})^{(1)}$ defined in~\eqref{eq:rho} with $p=1$ give rise to a ``retraction" $r:Q \to P^1$.
Denote by $\bar{f}:Q \to M$ the composition of $r:Q \to P^1$ with $f:P^1 \to M$, that is, $\bar{f}=f \circ r$.
Deform $\sigma:\partial P \to M$ into $\bar{\sigma}:\partial P \to M$ so that the restriction of~$\bar{\sigma}$ to each cubical $n$-simplex of~$\partial P$ agrees with $\sigma \circ \bar{\rho}$, where $\bar{\rho}$ is the extension of the ``retraction"~$\rho$ defined in~\eqref{eq:rho-bar} with $p=1$.
By construction, the map $\bar{f}:Q \to M$ agrees with $\bar{\sigma}$ on~$\partial P \cap Q$ and can be extended to a map 
\begin{equation} \label{eq:f-bar}
\bar{f}:Q \cup \partial P \to M
\end{equation}
which agrees with~$\bar{\sigma}$ on~$\partial P$; see Figure~\ref{fig:3}.
Furthermore, the map~$\bar{f}$ takes every edge of~$P^1$ to a segment of length less than~$\delta$.

\begin{figure}[!htbp]
\centering
\includegraphics[width=16cm]{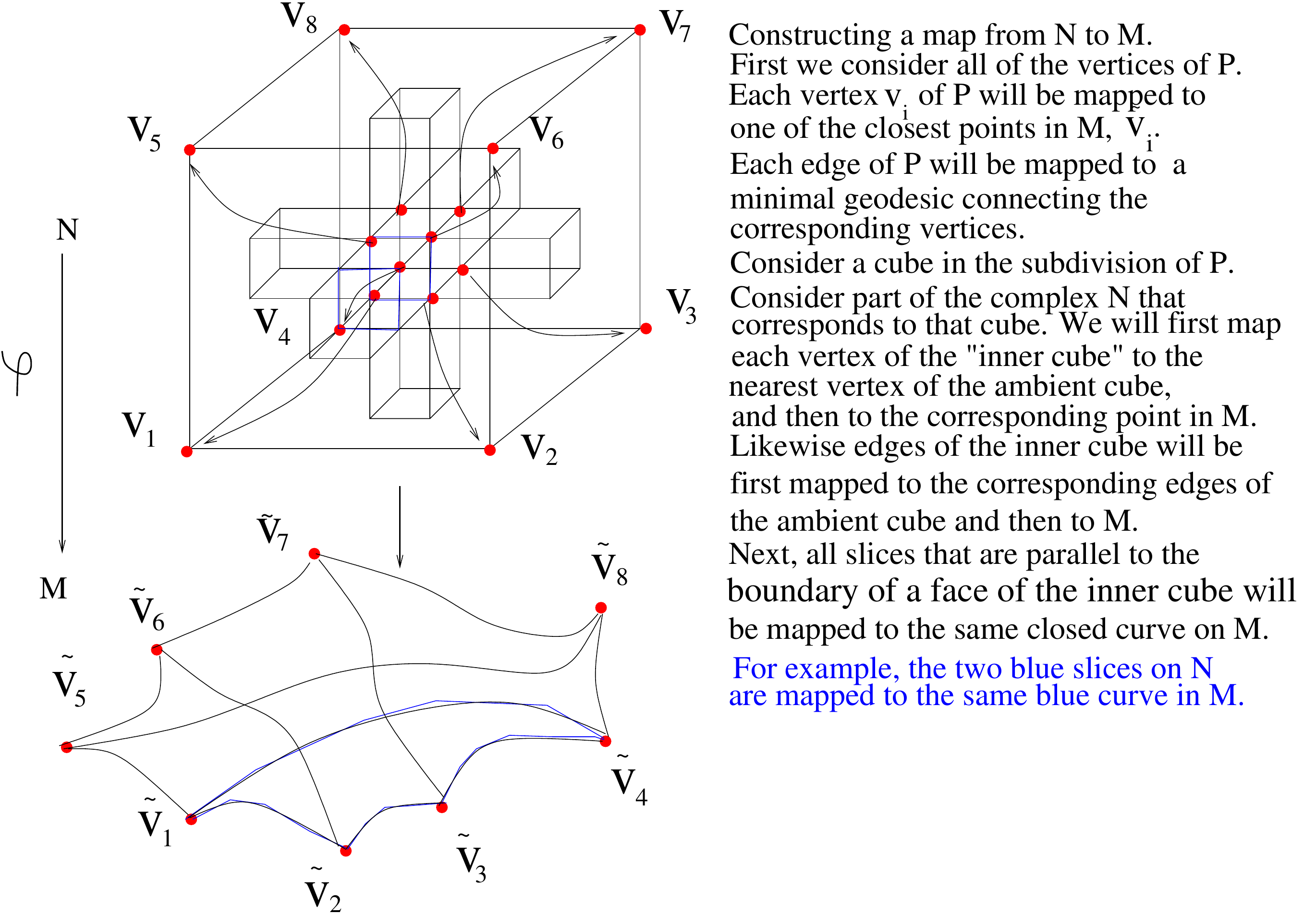}
\caption{Mapping~$Q$ to~$M$}
\label{fig:3}
\end{figure}

\medskip

The cubical $n$-pseudomanifolds~$Y^n \subseteq C^{n+1}$ pasted together according to the assembling pattern of the cubical $(n+1)$-simplices~$C^{n+1}$ of the pseudomanifold~$P$, see Section~\ref{sec:simplex} where $p=1$ (and Figure~\ref{fig:2}), form a compact cubical $n$-pseudomanifold~$N' \subseteq Q$ with boundary lying in~$\partial P$.
The map $\theta:Y^n \to Z^{n-1}$ defined in~\eqref{eq:theta} with $p=1$ gives rise to a map $h':N' \to T'$ to the finite cubical $(n-1)$-complex~$T'$ formed of the pieces~$Z^{n-1} \subseteq C^{n+1}$, where $C^{n+1}$ is a cubical $(n+1)$-simplex of~$P$.
More precisely, the restriction of~$h'$ to the pieces~$Y^n$ of~$N'$ is given by~$\theta$.
The cubical $n$-complexes $X_i^n \subseteq C^n$, where $C^n$ is a cubical $n$-simplex of~$\partial P \simeq M$, form a compact cubical $n$-pseudomanifold~$N_i'' \subseteq \partial P$ with the same boundary as~$N'$.
As previously, the map $\Theta:X_2^n \to Z^{n-2} \times [0,\frac{1}{2}]$ defined in~\eqref{eq:fhat} with $p=1$ gives rise to a map $h'':N_2'' \to T''$ to the finite cubical $(n-1)$-complex~$T''$ formed of the pieces~$Z^{n-2} \times [0,\frac{1}{2}]$ with $Z^{n-2} \subseteq C^n$, where $C^n$ is a cubical $n$-simplex of~$\partial P$.
Observe that $\Theta(x) = (\theta(x),\frac{1}{2})$ for every $x \in Y^n \cap X_2^n \simeq Y^{n-1} \subseteq C^n$.
Thus, the two maps $h'$ and~$h''$ so-defined agree on the common boundary of~$N'$ and~$N_2''$ after identifying $Z^{n-2} \subseteq Z^{n-1} \subseteq T'$ and $Z^{n-2} = Z^{n-2} \times \{ \frac{1}{2} \} \subseteq T''$, where $Z^{n-2} \subseteq C^n$ lies in~$\partial P$.
Put together, these maps give rise to a continuous map
\[
h:N \to T
\]
from the closed $n$-pseudomanifold $N = N' \cup N_2''$ lying in~$Q \cup \partial P$ to the cubical $(n-1)$-complex $T=T' \cup_S T''$ obtained by gluing~$T'$ and~$T''$ along the cubical $(n-2)$-complex~$S$ formed of the pieces $Z^{n-2} \subseteq C^n$, where $C^n$ is a cubical $n$-simplex of~$\partial P$.
(As a result of the inclusions $Z^{n-2} \subseteq Z^{n-1} \subseteq T'$ and $Z^{n-2} = Z^{n-2} \times \{ \frac{1}{2} \} \subseteq T''$, the complex~$S$ lies both in~$T'$ and in~$T''$.)
See Figure~\ref{fig:2} for a representation of~$N$.

\begin{figure}[!htbp]
\centering
\includegraphics[width=14cm]{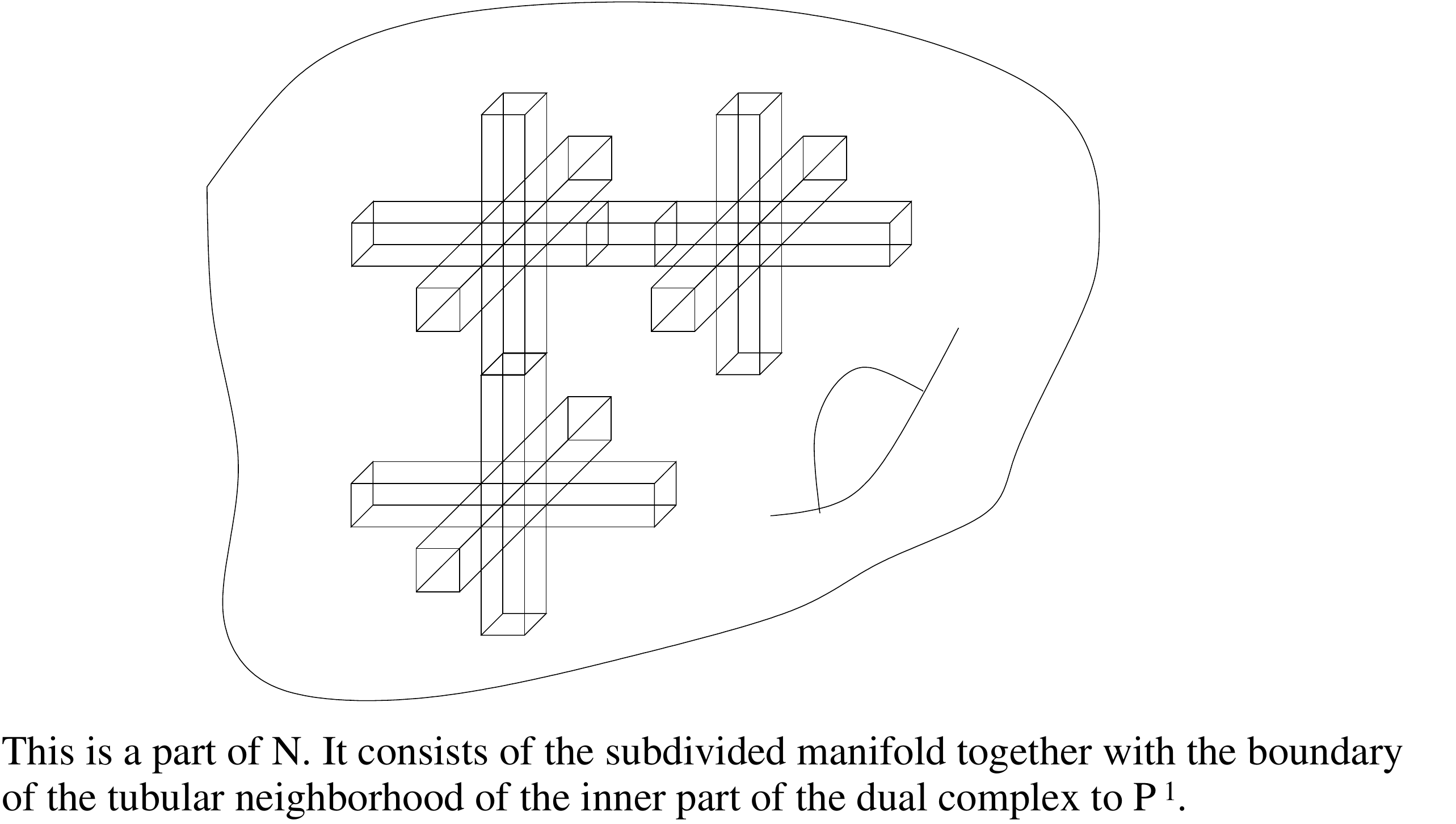}
\caption{Cubical structure of the pseudomanifold~$N$}
\label{fig:2}
\end{figure}

%The pseudomanifold~$N$ agrees with the boundary of an $(n-1)$-complex lying in the dual to the ~$1$-skeleton of the filling~$P$ of~$M$.

\medskip

By construction, every fiber~$h^{-1}(t) \subseteq N$ with $t \in T$ agrees with a fiber of~$\theta$ or~$\Theta$, and therefore is isomorphic to the $1$-skeleton of a cube of dimension at most~$n+1$; see Figure~\ref{fig:4}.
Thus, every fiber~$h^{-1}(t)$ has at most $(n+1) \, 2^n$ edges.
Furthermore, every fiber of~$h$ is sent by~$\bar{f}$ to a (possibly degenerate) cubical graph of length at most $(n+1) \, 2^n \, \delta < \w(M)$.
By definition of the homology $1$-waist, see Definition~\ref{def:width} where $p=1$, the restriction~$\bar{f}_|{N}:N \to M$ of~$\bar{f}$ to~$N$ is not of degree one.
That is,
\begin{equation} \label{eq:neq}
\bar{f}_*([N]) \neq [M] \in H_n(M).
\end{equation}

By construction, the pseudomanifolds~$N$ and~$\partial P$ are homologuous in~$Q \cup \partial P$.
Specifically, their difference as $n$-cycles bounds the pseudomanifold~$Q$.
This implies that
\begin{equation} \label{eq:f-sigmabar}
\bar{f}_*([N]) = \bar{f}_*([\partial P]) = \bar{\sigma}_*([\partial P])
\end{equation}
since $\bar{f}$ agrees with~$\bar{\sigma}$ on~$\partial P$.
Now, the map $\bar{\sigma}$ is a deformation of~$\sigma$.
Therefore,
\begin{equation} \label{eq:sigmabar}
\bar{\sigma}_*([\partial P]) = \sigma_*([\partial P]) = [M]
\end{equation}
by~\eqref{eq:sigma}.
Thus, the relations \eqref{eq:neq}, \eqref{eq:f-sigmabar} and~\eqref{eq:sigmabar} lead to a contradiction.
\forget
takes the fundamental class~$[N]$ to zero.
Therefore, there exists a continuous extension of~$f_{|N}:N \to M$ to a compact $(n+1)$-pseudomanifold~$Q'$ with boundary~$\partial Q' = N$.

\medskip

The boundaries $\partial P$, $\partial Q$ and~$\partial Q'$ decompose into
\[
\partial P = N_1'' \cup N_2'', \quad \partial Q = N' \cup N_1'', \quad \partial Q' = N' \cup N_2''.
\]
Combining the map $f:Q' \to M$ previously defined with $f:Q \to M$ yields a map $f':P' \to M$ defined on the compact $(n+1)$-pseudomanifold $P'=Q \cup_{N'} Q'$ with the same boundary as~$P$.
By construction, the map~$f'$ agrees with~$\bar{\sigma}$ on the common boundary of~$P$ and~$P'$.
This implies that the homology class $\bar{\sigma}_*([\partial P])$ vanishes.
On the other hand, since $\bar{\sigma}$ is a deformation of~$\sigma$, we have $\bar{\sigma}_*([\partial P]) = \sigma_*([\partial P]) = [M]$, which contradicts the previous claim.
\forgotten
\end{proof}

\begin{remark}
Working with simplicial complexes instead of cubical complexes yields a better constant in Theorem~\ref{theo:FR1}, namely, $c_n=\frac{1}{(n+1)(n+2)}$.
Still, we decided to work with cubical complexes since the constructions of Section~\ref{sec:simplex} are simpler to describe in this context.
\end{remark}

Theorem~\ref{theo:A} follows from Theorem~\ref{theo:FR1} and Theorem~\ref{theo:FR}.

\forget
\section{Filling radius and modified homology $1$-waist} \label{sec:FR1bis}

Let $M$ be a triangulated compact $n$-dimensional pseudomanifold with our without boundary. We will be interested in pseudomanifolds such that
each vertex is incident to at most $c(n)$ edges for some very large constants $c(n)$ that, however, depend only on $n$. We will say that $M$ has $c(n)$-bounded local complexity, or bounded local complexity. For example, each smooth manifold can be triangulated with bounded local complexity. One can see that, for example, from H. Whitney's
proof of the existence of smooth triangulations of manifolds. One embed a manifold into a high-dimensional Euclidean space and considers a subdivision
of the Euclidean space into very small cubes. Making  a generic shift of this triangulation one ensures that the manifold intersects each face
of each cube transversally. The intersections of the manifold with faces will be very close to the intersections with the $n$-plane and, therefore, can be subdivided in a locally bounded number of simplices.

Gromov proved in \cite{gro95}, section II' of ch. $5{5\over 7}$), that every triangulated closed $n$-dimensional pseudomanifold $Q$ such that each vertex is incident to at most $c$ simplices can be
filled by $(n+1)$-dimensional pseudomanifold $R$  that each vertex is incident
to at most $f(c,n)$ edges for some function $f$. In other words,a pseudomanifold with locally bounded complexity is a boundary of a pseudomanifold with
boundary with locally bounded complexity.

We will be interested in pseudomanifolds in Banach spaces. Observe, that if a closed pseudomanifold $M$ with bounded local complexity $Q$ is contained in
a small metric ball $B$ in a Banach space, and we represented it as a boundary $Q$ of pseudomanifold with boundary with bounded local complexity, we can project $Q$ to $B$ and
ensure that the resulting filling $Q'$ is in $B$ and has the same combinatorial structure as $Q$.

\begin{lemma} \label{prop:even} There exist positive integer numbers $c(1)<c(2)<\ldots $ such that:
\par\noindent
(a) Each smooth compact $n$-dimensional manifold can be triangulated with $c(n)$-bounded local complexity;
\par\noindent
(b) Each $n$-dimensional pseudomanifold with $c(n)$-bounded local complexity bounds $(n+1)$-dimensional pseudomanifold with  $c(n+1)$-bounded local complexity.
\par\noindent
(c) Let $(M,N)$ be a pair, where either (1) $M$ is a closed $n$-dimensional pseudomanifold and $N$ is either its $n$-dimensional sub-pseudomanifold with $c(n)$-bounded
local complexity or empty;
or (2) $M$ is a $n$-dimensional pseudomanifold with boundary, $N$ is an $(n-1)$-dimensional pseudomanifold of the boundary of $M$ (maybe, the whole $\partial M$), and $N$ has $c(n-1)$-bounded local geometry.
Then $N$ is also a subcomplex of a triangulated pseudomanifold $M'$ with $c(n)$-bounded local complexity. In the first case $M'$ is also closed, in the second it has a non-empty boundary $\partial M$, and
$N$ is a subcomplex of $\partial M'$. Moreover, if $M$ and $N$ are embedded in a Banach space $B$, then $M'$ is also embedded in $B$ and can be chosen so that its Hausdorff
distance to $M$ is arbitrarily small.
\end{lemma}

This lemma has the following immediate corollary:

\begin{proposition}
Let $M$ be a closed $n$-dimensional Riemannian manifold embedded to $L^\infty(M)$ using the Kuratowski embedding. Then for each $\epsilon>0$ $M$ bounds
a $(n+1)$-dimensional pseudomanifold $P$ with $c(n+1)$-bounded local complexity that is contained in $\FillRad(M)+\epsilon$-neighbourhood of $M$.
\end{proposition}

\begin{proof}
Endow $M$ with a triangulation so that it has $c(n)$-bounded local geometry.
The definition of the filling radius implies that $M$ is a boundary of a pseudomanifold $P_1$ contained in its $\FillRad(M)+{\epsilon\over 2}$-neighbourhood. Use part (d) of the previous proposition to replace $P_1$ by $P$ that sits in a ${\epsilon\over 2}$-neighbourhood of $P_1$, is bounded
by $M$, and has $c(n+1)$-bounded local geometry.
\end{proof}

It remain to prove the proposition:

\begin{proof}
The discussion before the text of the proposition already implies (a) and (b).
We can start from a filling of $M$ by a polyhedron $P'$ contained in
$(Fill Rad(M)+\epsilon/2)$-neighbourhood of $M$.

So, it is necessary to prove (c). We proceed by induction. The case $n=1$ is obvious, so we can assume that the proposition is already established for all dimensions less than $n$.

Now assume that the triangulation of $M$ is very fine (but still has a bounded local complexity), and proceed in a standard way: we remove small neighbourhoods of all vertices, creating for each vertex $p$ a ``hole" bounded
by the pseudomanifold $Lk(p)$, where $Lk(p)$ denotes the link of $p$, then remove small neighbourhoods of all edges $e$ creating
``holes" $Lk(e)\times e$, and so on. At the end each $n$-simplex will be truncated to a certain polytope (called permutohedron).
We can triangulate it into a certain fixed way. The resulting polyhedra will automatically have bounded local complexity.

Now we will start filling the holes $Lk'(a\ face)\times the\ face$ starting from the faces of the highest codimension, where the face is already endowed with
a triangulation with bounded local complexity obtained on previous steps. Note that here $Lk'$ coincides with $Lk$ on the first two steps
(links of faces of codimension $1$ and $2$) but on later steps will
be a certain modification of $Lk$ -see below. (Note that when we remove
``stuff" from $M$, we also remove it from $Lk$. When we fill back
modifications of some of the removed parts of $M$, we also (sometimes)
add to $Lk$. By the time we processed all lower dimensions, $Lk$ becomes
a new closed pseudomanifold, that we denote $Lk'$ and call a modified link.)
However, unlike $Lk$,
$Lk'$ will always have bounded local complexity.
Each (modified)link lies in a very small ball
and is a pseudomanifold. Its intersection with $N$ is either empty or a connected submanifold with boundary with bounded local geometry.
%Applying the induction assumption we can replace it by a closed pseudomanifold $\Sigma$ with bounded local complexity. Its part in $N$ is kept intact (with the exception of possibly passing to a subdivision with bounded local complexity).
As $Lk'$ is a pseudomanifold with bounded local complexity, it can be filled by a very small (in the metric of ambient Banach space $B$) pseudomanifold with bounded local complexity one dimension higher by induction assumption keeping its part in $N$ intact.

It remains to note that when we move on to faces of lower dimension (with higher dimensional links), we observe that each link already has been changed in the spirit of our construction: the top dimensional face (that are not in $N$) were truncated to
permutohedra, and stars of lower dimensional faces were replaced by insertions with bounded local complexity. The resulting pseudomanifolds will be the modified links $Lk'$ that were are going to modify at the coming step.

\end{proof}

\begin{remark}
\par\noindent
1. The proposition above is about triangulated manifolds. However,
we can fix any way to subdivide simplices into cubes, and will immediately see
that it will be true for pseudomanifolds subdivided into cubes;
\par\noindent
2. This proposition yields a control on the local complexity of the tringulation or the cubical structure of the filling, but it does not provide a control on the total number of  $n$-simplices/cubical $n$-simplices of the filling (which may be arbitrarily large).
\end{remark}

\section{Filling radius and modified homology $1$-waist} \label{sec:FR1bis}

We are interested in triangulated compact $n$-pseudomanifolds~$M$ with or without boundary such that each vertex is incident to at most $c(n)$ edges for some very large constants $c(n)$ depending only on~$n$. 
In this case, we say that $M$ has \emph{$c(n)$-bounded local complexity}, or \emph{bounded local complexity}. 

\medskip

For example, each smooth manifold can be triangulated with bounded local complexity. This follows, for example, from H.~Whitney's
proof of the existence of smooth triangulations of manifolds. 
First, embed a manifold into a high-dimensional Euclidean space and consider a subdivision of the Euclidean space into very small cubes. 
Making  a generic shift of this triangulation, one ensures that the manifold intersects each face of each cube transversely. 
If the cubes are small enough, the intersections of the manifold with the cubes are very close to the intersections with an $n$-plane and, therefore, can be subdivided in a locally bounded number of simplices.

\medskip

Gromov proved in~\cite[{\S $5{5\over 7}$.II'}]{gro99} that every triangulated closed $n$-pseudomanifold such that each vertex is incident to at most~$c$ simplices can be filled by a compact $(n+1)$-pseudomanifold such that each vertex is incident to at most $f(c,n)$ edges for some function~$f$. 
In other words, a closed pseudomanifold with locally bounded complexity is the boundary of a compact pseudomanifold with boundary with locally bounded complexity.

\medskip

The above discussion can be summarized by the following lemma.

\begin{lemma} \label{prop:even} 
There exist positive integer numbers $c(1)<c(2)<\ldots $ such that
\begin{enumerate}
\item each smooth compact $n$-manifold can be triangulated with $c(n)$-bounded local complexity;
\item each closed $n$-pseudomanifold with $c(n)$-bounded local complexity bounds a compact $(n+1)$-pseudomanifold with $c(n+1)$-bounded local complexity.
\end{enumerate}
\end{lemma}

\begin{remark}
Bounded local complexity could also be defined for cubical structures instead of simplicial structures.
Since every simplex can be decomposed into simplices and every cubical simplex can be decomposed into simplices, the two notions are equivalent.
Namely, an $n$-pseudomanifold~$M$ admits a triangulation with $c(n)$-bounded local simplicial complexity if and only if it admits a cubical structure with $c'(n)$-bounded local cubical complexity.
\end{remark}

\begin{proposition}
Let $P$ be a compact $(n+1)$-pseudomanifold whose boundary has $c(n)$-bounded local complexity.
Then there exists a compact $(n+1)$-pseudomanifold~$Q$ with the same boundary as~$P$ such that $Q$ has $c(n+1)$-bounded local complexity.

Moreover, if $P$ is embedded in a Banach space~$E$ then $Q$ is also embedded in~$E$ and can be chosen so that its Hausdorff distance to~$P$ is arbitrarily small.
\end{proposition}

\begin{proof}
We proceed by induction. The case $n=1$ is obvious, so we can assume that the proposition is already established for all dimensions less than $n+1$.

Denote by $s_i$ the maximal number of $(n+1)$-simplices of~$P$ incident to an $i$-face of~$P$.
Now, assume that the triangulation of $P$ is very fine (but still has bounded local complexity), and proceed in a standard way.
First, we remove a small neighbourhood of each vertex $p \in P \setminus \partial P$ by truncating every $(n+1)$-simplex~$\Delta$ of~$P$ containing~$p$ by an $n$-plane separating~$p$ from the other vertices of~$\Delta$.
This deletes the possible singularity at~$p$.
Denote by~$P_0$ the resulting $(n+1)$-pseudomanifold with the decomposition into convex polytope cells obtained by truncating the polytope cells of~$P$.
The vertices of~$P_0$ not in~$\partial P$ lie in the interior of the edges of~$P$.
Thus, the number of $(n+1)$-cells of~$P_0$ around each vertex of~$P_0$ is bounded by~$s_1$.
We refer to the faces of~$P_0$ lying in a face of~$P$ of the same dimension as original faces and the others as new faces.
Since $P$ is an $(n+1)$-pseudomanifold, the new faces created by truncation around each vertex~$p$ form a closed $n$-pseudomanifold~$Lk(p)$ homeomorphic to the link of~$p$ in~$P$.
By construction,
\[
P_0 = P \setminus {\rm int} \left( \coprod_{p \in P \setminus \partial P} {\rm cone}(Lk(p)) \times \{p\} \right)
\]
where $p$ runs over the vertices of $P \setminus \partial P$.

Then, we remove a small neighborhood of each original edge~$e$ of~$P_0$ by truncating every $(n+1)$-cell~$\Delta$ of~$P_0$ containing~$e$ by an $n$-plane parallel to~$e$ separating~$e$ from the other vertices of~$\Delta$.
More generally, for every $i$ from~$1$ to~$n-2$, we remove by induction a small neighborhood of each original $i$-face~$f_i$ of~$P_{i-1}$ by truncating every $(n+1)$-cell~$\Delta$ of~$P_{i-1}$ containing~$f_i$ by an $n$-plane parallel to~$f_i$ separating~$f_i$ from the other vertices of~$\Delta$.
This deletes the possible singularities along~$f_i$.
Denote by~$P_i$ the resulting $(n+1)$-pseudomanifold with the decomposition into convex polytope cells obtained by truncating the polytope cells of~$P_{i-1}$.
The vertices of~$P_i$ not in~$\partial P$ lie in the interior of the $(i+1)$-faces of~$P$ and the number of~$(n+1)$-cells of~$P_i$ around each vertex of~$P_i$ is bounded by~$s_{i+1}$.
We refer to the faces of~$P_i$ lying in a face of~$P$ of the same dimension as original faces and the others as new faces.
As previously, the new faces created by truncation around each original $i$-face~$f_i$ of~$P_{i-1}$ form a closed $(n-i)$-pseudomanifold~$Lk(f_i)$ homeomorphic to the link of the $i$-face of~$P$ containing~$f_i$.
By construction,
\[
P_i = P_{i-1} \setminus {\rm int} \left( \coprod_{f_i \subseteq P_{i-1}} {\rm cone}(Lk(f_i)) \times f_i \right)
\]
where $f_i$ runs over the original $i$-faces of~$P_{i-1}$.

The number of $(n+1)$-cells of~$P'=P_{n-2}$ around each vertex of~$P'$ is bounded by~$s_{n-1}=2$ (recall that $P$ is a pseudomanifold).
Since all the $(n+1)$-cells of~$P'$ have the same combinatorial structure, we can triangulate them in a fix way.
As a result, the closed $(n+1)$-pseudomanifold~$P'$ has bounded local complexity.
The closed $(n-i)$-pseudomanifold~$Lk(f_i)$ has also bounded local complexity (after further splitting, to explain).
By induction, there exists a compact $(n-i+1)$-pseudomanifold $Lk(f_i)^+$ with bounded local complexity bounding $Lk(f_i)$.
By construction, the $(n+1)$-pseudomanifold~$P$ can be reconstructed from~$P'$ by gluing back the pieces ${\rm cone}(Lk(f_i)) \times f_i$, where $f_i$ is an original $i$-face of~$P_{i-1}$ and $i$ runs from $n-2$ to~$0$.
Instead of gluing back the pieces ${\rm cone}(Lk(f_i)) \times f_i$, we can glue the pieces $Lk(f_i)^+ \times f_i$ with the same boundary, following the same pattern.
This gives rise to a compact $(n+1)$-pseudomanifold~$Q$ with bounded local complexity with the same boundary as~$P$. 

Suppose that $P$ is embedded in a Banach space~$E$.
Since the triangulation of~$P$ is very fine, the piece $Lk(f_i) \times f_i$ removed from~$P_{i-1}$ is contained in a small metric ball~$B$ in the Banach space~$E$.
By projection onto~$B$, we can assume that the filling $Lk(f_i)^+ \times f_i$ is also contained in~$B$.
Thus, the compact $(n+1)$-pseudomanifold~$Q$ lies within an arbitrarily small distance from~$P$ in~$E$.
This finishes the proof of the proposition.
\end{proof}

This lemma has the following immediate corollary.

\begin{proposition} \label{prop:complexity}
Let $M$ be a closed $n$-Riemannian manifold embedded into $L^\infty(M)$ by the Kuratowski embedding. 
Then, for each $\epsilon>0$, the manifold~$M$ bounds
a compact $(n+1)$-pseudomanifold~$P$ with $c(n+1)$-bounded local complexity that is contained in the $(\FillRad(M)+\epsilon)$-neighbourhood of~$M$.
\end{proposition}

\begin{proof}
Endow $M$ with a triangulation with $c(n)$-bounded local complexity.
The definition of the filling radius implies that $M$ is the boundary of a compact  $(n+1)$-pseudomanifold~$P_1$ contained in its $(\FillRad(M)+{\epsilon\over 2})$-neighbourhood. 
By Lemma~\ref{lem:M'}, we can replace~$P_1$ by another compact  $(n+1)$-pseudomanifold~$P$ that sits in an ${\epsilon\over 2}$-neighbourhood of $P_1$, is bounded
by~$M$ and has $c(n+1)$-bounded local complexity.
\end{proof}

\begin{remark}
Proposition~\ref{prop:complexity} yields a control on the local complexity of the simplicial/cubical structure of the filling, but it does not provide any control on the total number of $n$-simplices/cubical $n$-simplices of the filling (which may be arbitrarily large).
\end{remark}

\forgotten

\section{Filling radius and modified homology $1$-waist} \label{sec:FR1bis}

The goal of this section is to present an extension of Theorem~\ref{theo:FR1} where the filling radius is bounded from below in terms of a modified homology $1$-waist more suited for applications in the calculus of variations from a geometric measure theory point of view.

\medskip

As a preliminary, we are interested in triangulated compact $n$-pseudomanifolds~$M$ with or without boundary such that every vertex is incident to at most $\kappa_n$ edges for some very large constants $\kappa_n$ depending only on~$n$. 
In this case, we say that $M$ has \emph{$\kappa_n$-bounded local complexity}, or \emph{bounded local complexity}. 

\medskip

For example, every smooth manifold can be triangulated with bounded local complexity. This follows, for example, from H.~Whitney's
proof of the existence of smooth triangulations of manifolds. 
First, embed a manifold into a high-dimensional Euclidean space and consider a subdivision of the Euclidean space into very small cubes. 
Making  a generic shift of this triangulation, one ensures that the manifold intersects each face of each cube transversely. 
If the cubes are small enough, the intersections of the manifold with the cubes are very close to the intersections with an $n$-plane and, therefore, can be subdivided in a locally bounded number of simplices.

\medskip

Gromov proved in~\cite[{\S $5{\frac{5}{7}}$.II'}]{gro95} that every triangulated closed $n$-pseudomanifold such that each vertex is incident to at most~$\kappa$ simplices can be filled by a compact $(n+1)$-pseudomanifold such that each vertex is incident to at most $f(\kappa,n)$ edges for some function~$f$. 
In other words, a closed pseudomanifold with locally bounded complexity is the boundary of a compact pseudomanifold with boundary with locally bounded complexity.

\medskip

The above discussion can be summarized by the following lemma.

\begin{lemma} \label{lem:even} 
There exist positive integer numbers $\kappa_1<\kappa_2<\ldots $ such that
\begin{enumerate}
\item every smooth compact $n$-manifold can be triangulated with $\kappa_n$-bounded local complexity;\label{even1}
\item every closed $n$-pseudomanifold with $\kappa_n$-bounded local complexity bounds a compact $(n+1)$-pseudomanifold with $\kappa_{n+1}$-bounded local complexity. \label{even2}
\end{enumerate}
\end{lemma}

\begin{remark} \label{rem:even}
If a closed pseudomanifold~$P$ with bounded local complexity is contained in a (small) metric ball~$B$ of a Banach space, we can assume that the filling of~$P$ with bounded local complexity given by Lemma~\ref{lem:even}.\eqref{even2} also lies in~$B$ after projection.
\end{remark}

\begin{remark}
Bounded local complexity could also be defined for cubical structures instead of simplicial structures.
Since every simplex can be decomposed into cubes and every cubical simplex can be decomposed into simplices, the two notions are equivalent.
Namely, an $n$-pseudomanifold~$M$ admits a triangulation with $\kappa_n$-bounded local simplicial complexity if and only if it admits a cubical structure with $\kappa'_n$-bounded local cubical complexity.
\end{remark}

The following result implies that the filling with bounded local complexity given by Lemma~\ref{lem:even}.\eqref{even2} can be chosen arbitrarily close to any given filling.

\begin{proposition}
Let $P$ be a compact $(n+1)$-pseudomanifold in a Banach space~$E$, whose boundary~$\partial P$ has $\kappa_n$-bounded local complexity.
Then there exists a compact $(n+1)$-pseudomanifold~$P'$ in~$E$ with the same boundary as~$P$ such that $P'$ has $\kappa_{n+1}$-bounded local complexity and is at arbitrarily small Hausdorff distance to~$P$.
\end{proposition}

\begin{proof}
First, we are going to sketch the general idea of the proof. 
%Assume that the triangulation of $\partial P$ is very fine (but still has a bounded local complexity), and 
We proceed in the following way.
We remove small neighbourhoods of all vertices, creating for each vertex $p$ a ``hole" bounded by the pseudomanifold $Lk(p)$, where $Lk(p)$ denotes the link of~$p$.
Then we remove small neighbourhoods of all edges $e$ creating ``holes" $Lk(e)\times e$, and so on. 
At the end, each $n$-simplex will be truncated to a certain polytope (called permutohedron).
We can triangulate this polytope into a certain fixed way. The resulting polyhedron will automatically have bounded local complexity.
Then we reconstruct a filling $P'$ of~$\partial P$ starting from faces with the lowest codimension and going up. At each stage, we fill each link by the bounded local
complexity filling provided by Lemma~\ref{lem:even}. As all our surgeries can be done arbitrarily closely to the original filling, our construction
almost does not affect the distance from the filling $P'$ to $\partial P$.

The actual proof goes as follows.
Without loss of generality, we can assume that $P$ is piecewise linear and that its triangulation~$\mathcal{T}$ is $\epsilon$-fine (but still has bounded local complexity).
Consider the dual polyhedral decomposition~$\mathcal{P}$ of the triangulation~$\mathcal{T}$ of~$P$.
By definition, the dual polyhedral decomposition~$\mathcal{P}$ is formed of the closed stars of the vertices of~$\mathcal{T}$ in the first barycentric subdivision~$\mathcal{T}'$ of~$\mathcal{T}$.

Let $\Delta^{n-i}$ be an $(n-i)$-face of~$\mathcal{T}$ with $i \in \{0,\dots,n\}$.
Since $P$ is a compact $(n+1)$-pseudomanifold with boundary, the link~$Lk(\Delta^{n-i})$ of~$\Delta^{n-i}$ in~$\mathcal{T}'$ is a compact $i$-pseudomanifold\footnote{In the proof of Proposition~\ref{prop:complexity}, we relax the usual definition of a pseudomanifold to a finite disjoint union of pseudomanifolds allowing a pseudomanifold to be non-connected.} with boundary if $\Delta^{n-i}$ lies in~$\partial P$ and without boundary otherwise.
Define also the closed $i$-pseudomanifold
\begin{equation} \label{eq:Lk+}
Lk_+(\Delta^{n-i}) = Lk(\Delta^{n-i}) \cup {\rm cone}(\partial Lk(\Delta^{n-i}))
\end{equation}
where ${\rm cone}(\partial Lk(\Delta^{n-i}))$ is the cone over~$\partial Lk(\Delta^{n-i})$ arising from the center~$\omega(\Delta^{n-i})$ of~$\Delta^{n-i}$.
Note that $Lk_+(\Delta^{n-i}) = Lk(\Delta^{n-i})$ if $\Delta^{n-i}$ does not lie in~$\partial P$, and that $Lk_+(\Delta^{n-i}) \subseteq \partial P$ if $\Delta^{n-i}$ lies in~$\partial P$.

%The cone~$\widehat{Lk_+}(\Delta^{n-i})$ over~$Lk_+(\Delta^{n-i})$ from~$\omega(\Delta^{n-i})$ is a compact $(i+1)$-pseudomanifold with $\partial \widehat{Lk_+}(\Delta^{n-i}) = Lk_+(\Delta^{n-i})$.
%That is,
%\[
%\widehat{Lk_+}(\Delta^{n-i}) = {\rm cone}(Lk_+(\Delta^{n-i})).
%\]

By construction,
\begin{equation} \label{eq:Lk}
Lk(\Delta^{n-i}) = \bigcup_{\Delta^{n-i} \subseteq \Delta^{n-i+1}} {\rm cone}(Lk_+(\Delta^{n-i+1}))
\end{equation}
where the union is over all the $(n-i+1)$-simplices~$\Delta^{n-i+1}$ of~$\mathcal{T}$ containing~$\Delta^{n-i}$ and 
\[
{\rm cone}(Lk_+(\Delta^{n-i+1})) \text{ is the cone over } Lk_+(\Delta^{n-i+1}) \text{ arising from } \omega(\Delta^{n-i+1}).
\]
Note that
\[
{\rm cone}(Lk(\{p\})) = {\rm Star}(p)
\]
for every vertex~$p$ of~$\mathcal{T}$.
Furthermore, ${\rm Star}(p)$ and ${\rm Star}(q)$ intersect each other if and only if $p$ and~$q$ are adjacent vertices of~$\mathcal{T}$.
In this case,
\begin{equation} \label{eq:stars}
{\rm Star}(p) \cap {\rm Star}(q) = {\rm cone}(Lk_+([p,q])).
\end{equation}

Since the pseudomanifold~$P$ is formed of the stars ${\rm Star}(p)$, it can be reconstructed from the augmented $1$-dimensional links~$Lk_+(\Delta^{n-1})$ and the centers of the faces of~$\mathcal{T}$ by following the pattern~\eqref{eq:Lk} and the relation~\eqref{eq:Lk+}. %, taking the cones over $Lk_+(\Delta^{n-i})$.
This has to be done iteratively for $i$ equals~$1$ to~$n$ until $Lk(\{p\})$ and eventually ${\rm cone}(Lk(\{p\}))$ are reconstructed.
Observe that the boundary~$\partial P$ of~$P$ is formed of the union of all the cones occurring in~\eqref{eq:Lk+}.
That is,
\[
\partial P = \bigcup_{p \text{ vertex of } \mathcal{T}} {\rm cone}(\partial Lk(\{p\}))
\]
where the union is over all vertices~$p$ of~$\mathcal{T}$ (lying in~$\partial P$).

\medskip

Now, we want to construct a different filling~$P'$ with bounded local complexity following a similar (re)-construction procedure.
Specifically, we want to define a compact $i$-pseudomanifold $\Phi(Lk(\Delta^{n-i}))$ with bounded local complexity lying at distance $\lesssim \epsilon$ from a vertex of~$\mathcal{T}$ for every $(n-i)$-face~$\Delta^{n-i}$ of~$\mathcal{T}$.
Moreover, we require that $\partial \Phi(Lk(\Delta^{n-i})) = \partial Lk(\Delta^{n-i})$ whenever $\Delta^{n-i}$ lies in~$\partial P$.
For $i=1$, let $\Phi(Lk(\Delta^{n-1})) = Lk(\Delta^{n-1})$.
As a union of $1$-pseudomanifolds, $Lk(\Delta^{n-1})$ has $2$-bounded local complexity and lies at distance at most~$\epsilon$ from~$\omega(\Delta^{n-1})$.

Suppose that $\Phi(Lk(\Delta^{n-i}))$ is defined for every $(n-i)$-face~$\Delta^{n-i}$ of~$\mathcal{T}$.
By similarity with~\eqref{eq:Lk+}, define 
\[
\Phi(Lk_+(\Delta^{n-i})) = \Phi(Lk(\Delta^{n-i})) \cup {\rm cone}(\partial \Phi(Lk(\Delta^{n-i}))).
\]
Recall that $\partial \Phi(Lk(\Delta^{n-i})) = \partial Lk(\Delta^{n-i})$ is empty if $\Delta^{n-i}$ does not lie in~$\partial P$ and is contained in~$\partial P$ otherwise.
Thus, $\Phi(Lk_+(\Delta^{n-i}))$ is a closed $i$-pseudomanifold with bounded local complexity.

Now, define $\Phi({\rm cone}(Lk_+(\Delta^{n-i})))$ as the polyhedral pseudomanifold filling of~$\Phi(Lk_+(\Delta^{n-i}))$ with bounded local complexity given by Lemma~\ref{lem:even}.\eqref{even2}.
This filling lies within distance~$\lesssim \epsilon$ from a vertex of~$\mathcal{T}$; see Remark~\ref{rem:even}.
Finally, by similarity with~\eqref{eq:Lk}, let
\[
\Phi(Lk(\Delta^{n-i-1})) = \bigcup_{\Delta^{n-i-1} \subseteq \Delta^{n-i}} \Phi({\rm cone}(Lk_+(\Delta^{n-i}))).
\]
Note that $\Phi(Lk(\Delta^{n-i-1}))$ also lies within distance~$\lesssim \epsilon$ from a vertex of~$\mathcal{T}$.

As in~\eqref{eq:stars}, $\Phi({\rm Star}(p))$ and $\Phi({\rm Star}(q))$ intersect each other if and only if $p$ and~$q$ are adjacent vertices of~$\mathcal{T}$.
In this case, their intersection
\[
\Phi({\rm Star}(p)) \cap \Phi({\rm Star}(q)) = \Phi({\rm cone}(Lk_+([p,q])))
\]
is a compact $n$-pseudomanifold.
This implies that the union 
\[
P' = \bigcup_{p \text{ vertex of } \mathcal{T}} \Phi({\rm Star}(p))
\]
over all vertices of~$\mathcal{T}$ is a compact $(n+1)$-pseudomanifold with bounded local complexity lying at distance~$\lesssim \epsilon$ from the vertex set of~$\mathcal{T}$ and so from~$P$.
Furthermore,
\[
\partial P' = \bigcup_{p \text{ vertex of } \mathcal{T}} {\rm cone}(\partial \Phi(Lk(\{p\}))) = \bigcup_{p \text{ vertex of } \mathcal{T}} {\rm cone}(\partial Lk(\{p\})) = \partial P
\]
as required.
\end{proof}

This proposition has the following immediate corollary.

\begin{proposition} \label{prop:complexity}
Let $M$ be a closed Riemannian $n$-manifold embedded into $L^\infty(M)$ by the Kuratowski embedding. 
Then, for every $\epsilon>0$, the manifold~$M$ bounds
a compact $(n+1)$-pseudomanifold~$P$ with $\kappa_{n+1}$-bounded local complexity that is contained in the $(\FillRad(M)+\epsilon)$-neighbourhood of~$M$.
\end{proposition}

\begin{remark}
Proposition~\ref{prop:complexity} yields a control on the local complexity of the simplicial/cubical structure of the filling, but it does not provide any control on the total number of $n$-simplices/cubical $n$-simplices of the filling (which may be arbitrarily large).
\end{remark}

Now, we can derive the following filling radius estimate extending Theorem~\ref{theo:FR1}.
See~\eqref{eq:defWtilde} for the definition of~$\w'_1(M)$.

\begin{theorem} \label{theo:FR1bis}
Let $M$ be a closed $n$-manifold.
Then every Riemannian metric on~$M$ satisfies 
\[
\FR(M) \geq c_n \, \w'_1(M)
\]
for $c_n = \frac{1}{2^{2n-1} (n+1)! \, \kappa_n}$.
\end{theorem}

\begin{proof}
We argue by contradiction as in the proof of Theorem~\ref{theo:FR1} using the same notations.
Suppose that there exists a continuous map $\sigma:P \to U_\nu(M) \subseteq L^\infty(M)$ defined on a compact cubical $(n+1)$-pseudomanifold~$P$ with boundary such that the restriction $\sigma:\partial P \to M$ satisfies 
\[
\sigma_*([\partial P]) = [M] \in H_n(M)
\]
with $\nu < \frac{1}{8 \kappa'_n} \w'_1(M)$, where $\kappa'_n$ is defined in terms of~$\kappa_n$ in~\eqref{eq:kappa'}.
%By Lemma~\ref{lem:even}.\eqref{even1} and Proposition~\ref{prop:complexity}, we can assume that $P$ has bounded local complexity.

\medskip

From now on, we will consider a cubical structure of~$P$ with bounded local complexity given by Proposition~\ref{prop:complexity}.
As in the proof of Theorem~\ref{theo:FR1}, we construct a map $\bar{f}:Q \cup \partial P \to M$ where $Q$ is a neighborhood of~$P^1$ in~$P$ which agrees with a deformation~$\bar{\sigma}$ of~$\sigma$.
Recall that the map~$\bar{f}$ takes every edge of~$P^1$ to a segment of length less than~$\delta=2 \nu + \varepsilon < \frac{1}{4 \kappa'_n} \w'_1(M)$.
We also construct a map $h:N \to T$ from a closed $n$-pseudomanifold~$N$ homologuous to~$\partial P$ in~$Q \cup \partial P$ to a cubical $(n-1)$-complex~$T=T' \cup_S T''$ formed of $(n-1)$-cubes glued together along their $(n-2)$-faces, where each fiber~$h^{-1}(t)$ is isomorphic to the $1$-skeleton of a cube of dimension at most~$n+1$.
If $M$ is orientable, the filling~$P$ is also orientable (modulo~$\partial P$) and so is the closed pseudomanifold~$N$.

\medskip

Let $C^{n+1}$ be a cubical $(n+1)$-simplex of~$P$.
Denote by $\Isom(C^{n+1})$ the (full) isometry group of~$C^{n+1}$.
Recall that $T' \cap C^{n+1}$ agrees with
\[
Z^{n-1} = \{ t \in [-1,1]^{n+1} \mid \mbox{ there exist } i \neq j \mbox{ such that } t_i = t_j =0 \}.
\]
A fundamental domain for the action of~$\Isom(C^{n+1})$ on~$T' \cap C^{n+1} = Z^{n-1}$ is given by
\begin{equation} \label{eq:Delta}
\Delta = \Delta^{n-1} = \{ t \in [-1,1]^{n+1} \mid 0 \leq t_1 \leq \cdots \leq t_{n-1} \leq 1 \mbox{ and } t_n = t_{n+1} = 0 \}.
\end{equation}
Its orbit gives rise to a natural triangulation of~$Z^{n-1}$ with 
\[
|\Isom(C^{n+1})|/|{\rm Stab}(\Delta^{n-1})| = 2^{n+1} (n+1)! / 2 = 2^{n} (n+1)!
\]
copies of~$\Delta^{n-1}$.
The pieces $Z^{n-2} \times [0,1] \subseteq C^{n+1}$ composing~$T''$ also lie in~$Z^{n-1}$ and the action of $\Isom(C^{n+1}) \cap {\rm Stab}(Z^{n-2} \times [0,1])$ on each of these pieces has $\Delta^{n-1}$ for fundamental domain too.
This gives rise to a triangulation of $Z^{n-2} \times [0,1]$ with
\[
|\Isom(C^{n+1}) \cap {\rm Stab}(Z^{n-2} \times [0,1])| / |{\rm Stab}(\Delta^{n-1})| = 2^{n-1} n!
\]
copies of~$\Delta^{n-1}$, which is compatible with the triangulation of~$Z^{n-1}$.
Denote by $q':Z^{n-1} \to \Delta^{n-1}$ and by $q'':Z^{n-2} \times [0,1] \to \Delta^{n-1}$ the quotient maps.
Since $T=T' \cup T''$ is made of copies of~$Z^{n-1}$ and $Z^{n-2} \times [0,1]$ glued together, there exists a surjective continuous map $\jmath:T \to \Delta^{n-1}$ whose restriction to each copy of~$Z_T$ agrees with the quotient maps~$q'$ or~$q''$, where $Z_T = Z^{n-1}$ or $Z^{n-2} \times [0,1]$.
(It does not matter how the copies of~$Z_T$ are isometrically identified as long as $Z^{n-2} \times \{1\}$ lies in $T' \cap T''$ since, at the end, we take the quotient by the isometry group of~$Z^{n-1}$ or~$Z^{n-2}$.)
Note that the restriction of every fiber of $\jmath:T \to \Delta$ to any copy of~$Z_T$ coincides with an orbit of the isometry group of~$Z^{n-1}$ or~$Z^{n-2}$.
By construction, the complex~$T$ is tiled with copies~$\Delta_i$ of~$\Delta$ such that $\jmath_{|\Delta_i}: \Delta_i \to \Delta$ is a diffeomorphism.
Denote by~$\mathcal{T}_T$ the corresponding triangulation of~$T$.

\forget
Let $\Delta=\Delta^{n-1}$ be the $(n-1)$-simplex obtained as the quotient of the standard $(n-1)$-cube~$C^{n-1}$ by its (full) isometry group $\Isom(C^{n-1})$.
A fundamental domain of the action of~$\Isom(C^{n-1})$ on~$C^{n-1} = [-1,1]^{n-1}$ is given by
\begin{equation} \label{eq:Delta}
\Delta = \{ x \in [-1,1]^{n-1} \mid 0 \leq x_1 \leq \cdots \leq x_{n-1} \leq 1 \}.
\end{equation}
Its orbit gives rise to a natural triangulation~$\mathcal{T}$ of~$C^{n-1}$ with $|\Isom(C^{n-1})| = 2^{n} (n+1)!$ copies of~$\Delta$.
Denote by $q:C^{n-1} \to \Delta^{n-1}$ the quotient map.
There exists a continuous surjective map $\jmath:T \to \Delta$ whose restriction to each cubical $(n-1)$-simplex of~$T$ agrees with the quotient map~$q$.
(It does not matter how the $(n-1)$-cubes of~$T$ are isometrically identified with the standard $(n-1)$-cube~$C^{n-1}$ since, at the end, we take the quotient by the isometry group of~$C^{n-1}$.)
Note that the restriction of every fiber of~$\jmath:T \to \Delta$ to any cubical $(n-1)$-simplex~$C$ of~$T$ coincides with an orbit of the isometry group of~$C$.
By construction, the complex~$T$ is tiled with copies~$\Delta_i$ of~$\Delta$ such that $\jmath_{|\Delta_i}: \Delta_i \to \Delta$ is a diffeomorphism.
Denote by~$\mathcal{T}_T$ the corresponding triangulation of~$T$.
\forgotten{}

\medskip

Since the map $h:N \to T$ is a submersion away from the inverse image of the $(n-2)$-skeleton of~$\mathcal{T}_T$, the composition
\[
\hbar:N \overset{h}{\to} T \overset{\jmath}{\to} \Delta
\]
is a submersion over the interior~$\mathring{\Delta}$ of~$\Delta$.
%away from the inverse image of~$\partial \Delta$. 
Moreover, every fiber of~$\hbar$ over~$\mathring{\Delta}$ is composed of exactly $2^{n} (n+1)! \, |P| + 2^{n-1} n! \, |\partial P|$ disjoint simple loops, where 
%$2^{n-1} (n-1)!$ is the order of the isometry group of~$C^{n-1}$ and 
$|P|$ is the number of cubical $(n+1)$-simplices of~$P$ and $|\partial P|$ is the number of cubical $n$-simplices of~$\partial P$.
Furthermore, the image under $\bar{f}_{|N}:N \to M$ of each of these loops is of length at most~$4\delta$.

\medskip

Suppose that $M$ is orientable.
%In this case, the pseudomanifold~$N$ is also orientable.
Fix an orientation on~$N$ and~$\Delta$.
Since $\hbar:N \to \Delta$ is a submersion away from the inverse image of~$\partial \Delta$, we can define in a unique way an orientation on the fibers~$\hbar^{-1}(x)$, with $x \in \mathring{\Delta}$, so that $\hbar^* \omega_\Delta \wedge \omega_{\hbar^{-1}(x)}$ is positive, where $\omega_\Delta$ is a positive volume form on~$\Delta$ and $\omega_{\hbar^{-1}(x)}$ is a volume form on~$\hbar^{-1}(x)$ defining its orientation. % on~$\hbar^{-1}(x)$.
If $M$ is nonorientable, we only consider unoriented cycles and there is no need to define an orientation on the fibers of~$\hbar$.

\medskip

%Furthermore, the fibers varies continuously with $x \in \Delta \setminus \partial \Delta$ in the piecewise smooth topology.
The family of $1$-cycles $\Xi_x = \hbar^{-1}(x) \subseteq N$ with $x \in \mathring{\Delta}$ extends by continuity to a family of $1$-cycles parameterized by~$\Delta$ as follows.
Specifically, we want to define~$\Xi_{x_0}$ for $x_0 \in \partial \Delta$.
Fix $t_0 \in \jmath^{-1}(x_0)$.
For a small enough neighborhood~$\mathcal{U}$ of~$t_0$ in~$T$, the points of $\jmath^{-1}(x) \cap \mathcal{U}$ converge to~$t_0$ as $x \in \mathring{\Delta}$ goes to~$x_0$.
Moreover, the cardinality~$k_{t_0}$ of $\jmath^{-1}(x) \cap \mathcal{U}$ is bounded by~$2^{n} (n+1)!$ times the number of pieces~$Z_T$ of~$T$ containing~$t_0$.
Since each of these pieces lies in a cube~$C^{n+1}$ corresponding to a cubical $(n+1)$-simplex of~$P$ or a cubical $n$-simplex of~$\partial P$, it follows from the $\kappa_n$-bounded local complexity of~$P$ that
%the number of cubical $(n+1)$-simplices of~$P$ containing~$t_0$ is bounded by~$\kappa_n$.
%times the number of cubical $(n-1)$-simplices of~$Z^{n-1}$ in~$C^{n+1}$ (or cubical $(n-2)$-simplices of~$Z^{n-2}$ in~$C^{n}$, whichever is larger).
\begin{equation} \label{eq:kappa'}
k_{t_0} = |\jmath^{-1}(x) \cap \mathcal{U}| \leq \kappa_n' := 2^{n} (n+1)! \, \kappa_n.
\end{equation}
Furthermore, we have the following claim.

\begin{claim} \label{claim:pairs}
Suppose that $t_0 \notin Z^{n-2} \times \{0\}$.
For $x \in \mathring{\Delta}$ close enough to~$x_0$, the points of $\jmath^{-1}(x) \cap \, \mathcal{U}$ can be partitioned into pairs $\{t_i,\bar{t}_i\}$ whose inverse images $h^{-1}(t_i)$ and $h^{-1}(\bar{t}_i)$ converge to~$h^{-1}(t_0)$ with opposite orientations as $x$ goes to~$x_0$.
\end{claim}

\begin{proof}
By symmetry, without loss of generality, we can assume that $t_0$ lies in (the boundary of) a copy of $\Delta \subseteq C^{n+1}$; see~\eqref{eq:Delta}.
Let us examine three mutually disjoint cases.

\medskip

{\it Case 1.} Suppose that $(t_0)_{n-1} =1$ and $0 < (t_0)_1 < \cdots < (t_0)_{n-2} <1$.
Thus, $t_0$ lies in an open $(n-2)$-face~$F$ of~$\partial \Delta$.
%Suppose that $t_0$ lies in an open $n$-face~$F$ of~$P \setminus \partial P$.
%That is, $(t_0)_{n-1} =1$ and $(t_0)_{n-2} <1$.
Since $P$ and $\partial P$ are cubical pseudomanifolds (with and without boundary) and $t_0 \notin Z^{n-2} \times \{0\}$, there are exactly two pieces $Z_T$ and~$\bar{Z}_T$ of~$T$ containing the $(n-2)$-face~$F$ of~$\partial \Delta$ with $\Delta \subseteq Z_T$.
Denote by $\bar{\Delta} \subseteq \bar{Z}_T$ the symmetric of~$\Delta$ with respect to~$F$.
%cubical $(n+1)$-simplices~$C$ and~$\bar{C}$ of~$P$ containing~$F$ as an $n$-face.
The set $\jmath^{-1}(x) \cap \mathcal{U}$ is formed of exactly two points $t \in \Delta$ and $\bar{t} \in \bar{\Delta}$ symmetric with respect to~$F$.
By symmetry, the loops $h^{-1}(t)$ and~$h^{-1}(\bar{t})$ have opposite orientation at the limit when $x$ goes to~$x_0$.

\medskip

{\it Case 2.} Suppose that $(t_0)_1=0$ (so Case~1 is not satisfied).
Denote by~$\sigma$ the symmetry of~$C^{n+1}$ with respect to the hyperplane $\{t_1=0\}$ of~$C^{n+1}$.
The set $\jmath^{-1}(x) \cap \mathcal{U} \cap C^{n+1}$ decomposes into two subsets $\Sigma_+ \subseteq \{t_1 >0 \}$ and $\Sigma_- \subseteq \{t_1 <0 \}$ symmetric with respect to~$\sigma$.
By symmetry, the loops $h^{-1}(t)$ and~$\sigma(h^{-1}(t))$ have opposite orientation at the limit when $x$ goes to~$x_0$.

\medskip

{\it Case 3.} Suppose that $(t_0)_1>0$ and that Case~1 is not satisfied.
As $t_0 \in \partial \Delta^{n-1}$, 
%Let $C\simeq [-1,1]^{n-1}$ be a cubical $(n-1)$-simplex of~$T$ containing~$x_0$.
%By assumption, the point~$x_0$ does not lie in any $(n-2)$-face of~$T$.
%By symmetry, we can assume that $x_0$ lies in~$\partial \Delta$, where $\Delta$ is defined in~\eqref{eq:Delta}.
there exist disjoint subsets $I_1,\dots,I_k \subseteq \{1,\dots,n-1\}$ with $|I_i| \geq 2$ such that $(t_0)_p = (t_0)_q$ for every $p,q \in I_i$.
The isotropy subgroup of~$t_0$ for the action of~$\Isom(C^{n+1})$ is isometric to the product $\Gamma=S(I_1) \times \cdots \times S(I_k)$, where the symmetry groups~$S(I_i)$ act by isometries on~$C^{n+1}$ by permuting the coordinates with index in~$I_i$.
Furthermore, the points of $\Sigma = \jmath^{-1}(x) \cap \mathcal{U} \cap C^{n+1}$ form a free orbit of~$\Gamma$.
Fix a point $s \in \Sigma$ and a bijection $\theta:\Gamma_+ \to \Gamma_-$ between the orientation-preserving and orientation-reversing isometries of~$\Gamma$.
This bijection gives rise to a partition of~$\Sigma$ into pairs of points $\{t_i,\bar{t}_i\}$, where $t_i = \sigma_i(s)$ and $\bar{t}_i = \theta(\sigma_i)(s)$ for some $\sigma_i \in \Gamma_+$.
By symmetry, the loops $h^{-1}(t_i)$ and~$h^{-1}(\bar{t}_i)$ have opposite orientation at the limit when $x$ goes to~$x_0$.
\end{proof}

For every $x \in \Delta$, we have the following decomposition of~$\hbar^{-1}(x)$ into connected components
\begin{equation} \label{eq:hbar(x)}
\hbar^{-1}(x) = \bigcup_{t \, \in \, \jmath^{-1}(x)} h^{-1}(t).
\end{equation}
Observe also that if $t_0 \in Z^{n-2} \times \{0\}$, then the fiber~$h^{-1}(t)$ converges to a point as $t$ goes to~$t_0$.

\medskip

We can now define the $1$-cycle~$\Xi_{x_0} \subseteq N$ so that its restriction to~$h^{-1}(t_0)$ agrees with the limit of the simple loops corresponding to the fibers of~$h$ over the $k_{t_0}$ points of $\jmath^{-1}(x) \cap \mathcal{U}$ as $x \in \mathring{\Delta}$ goes to~$x_0$.
This completely characterizes~$\Xi_{x_0}$ for each of its connected components since its support lies in~$\hbar^{-1}(x_0)$.
By Claim~\ref{claim:pairs} and the observation following~\eqref{eq:hbar(x)}, the $1$-cycle~$\Xi_{x_0}$ is equivalent to the zero $1$-cycle.
See Figure~\ref{fig:sm2} and Figure~\ref{fig:sm1}.

\begin{figure}[!htbp]
\centering
\includegraphics[width=14cm]{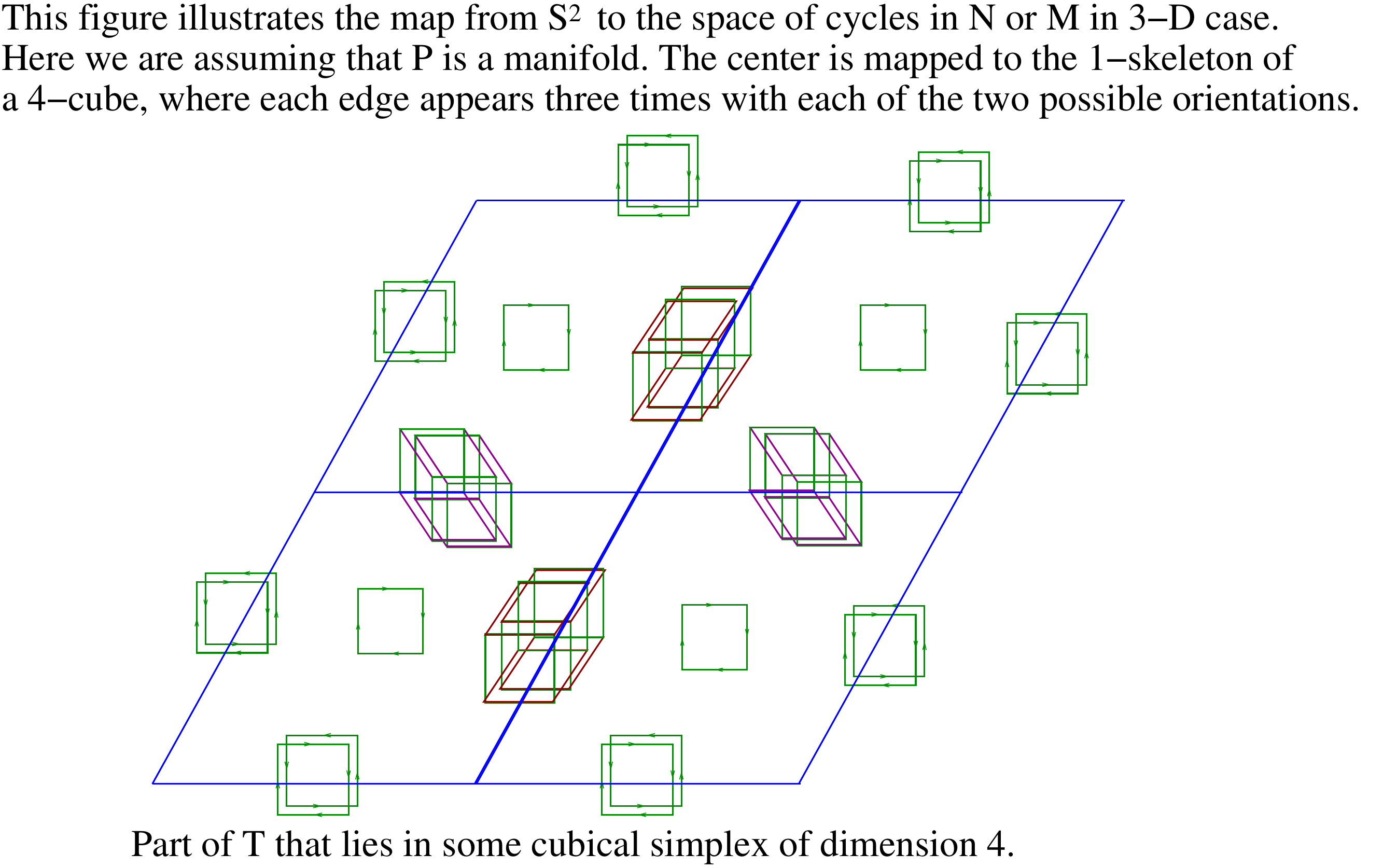}
\caption{The $S^2$-family of $1$-cycles~$\Xi_x$}
\label{fig:sm2}
\end{figure}

\begin{figure}[!htbp]
\centering
\includegraphics[width=14cm]{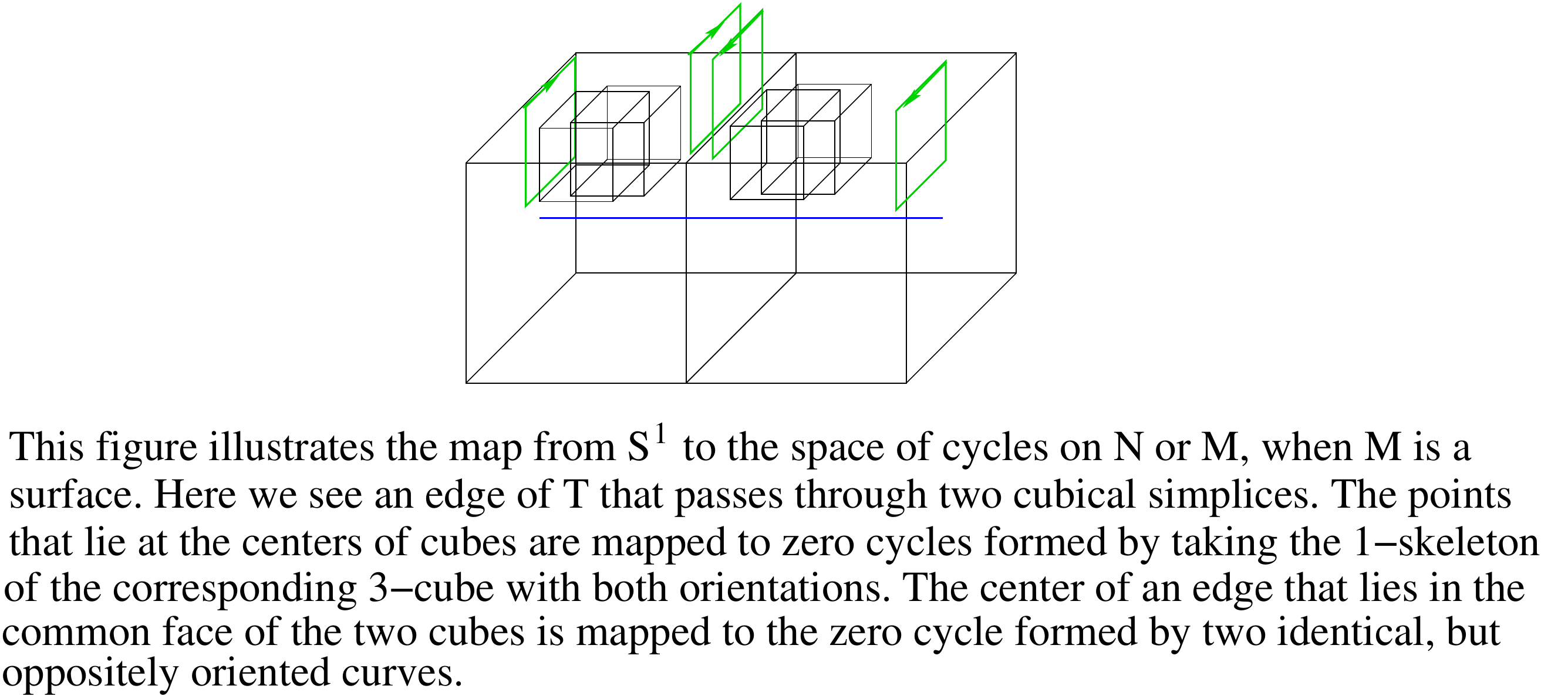}
\caption{The $S^1$-family of $1$-cycles~$\Xi_x$}
\label{fig:sm1}
\end{figure}

\medskip

As noticed before, the length of the image under $\bar{f}_{|N}:N \to M$ of each fiber~$h^{-1}(t)$ with $t \in \jmath^{-1}(x)$ and $x \in \mathring{\Delta}$ is at most~$4\delta$.
Thus, the maximal length of the image of a connected component of~$\Xi_x$ (counted according to its geometric multiplicity) is at most~$4 \delta \kappa'_n$; see Figure~\ref{fig:multiplicity}.
That is, for every $x \in \Delta$, we have
\[
\max_{C \, \subseteq \, \Xi_x} \length(\bar{f}_{|C}) \leq 4 \delta \kappa'_n
\]
where $C$ runs over the connected components of the $1$-cycle~$\Xi_x$.

\begin{figure}[!htbp]
\centering
\includegraphics[width=14cm]{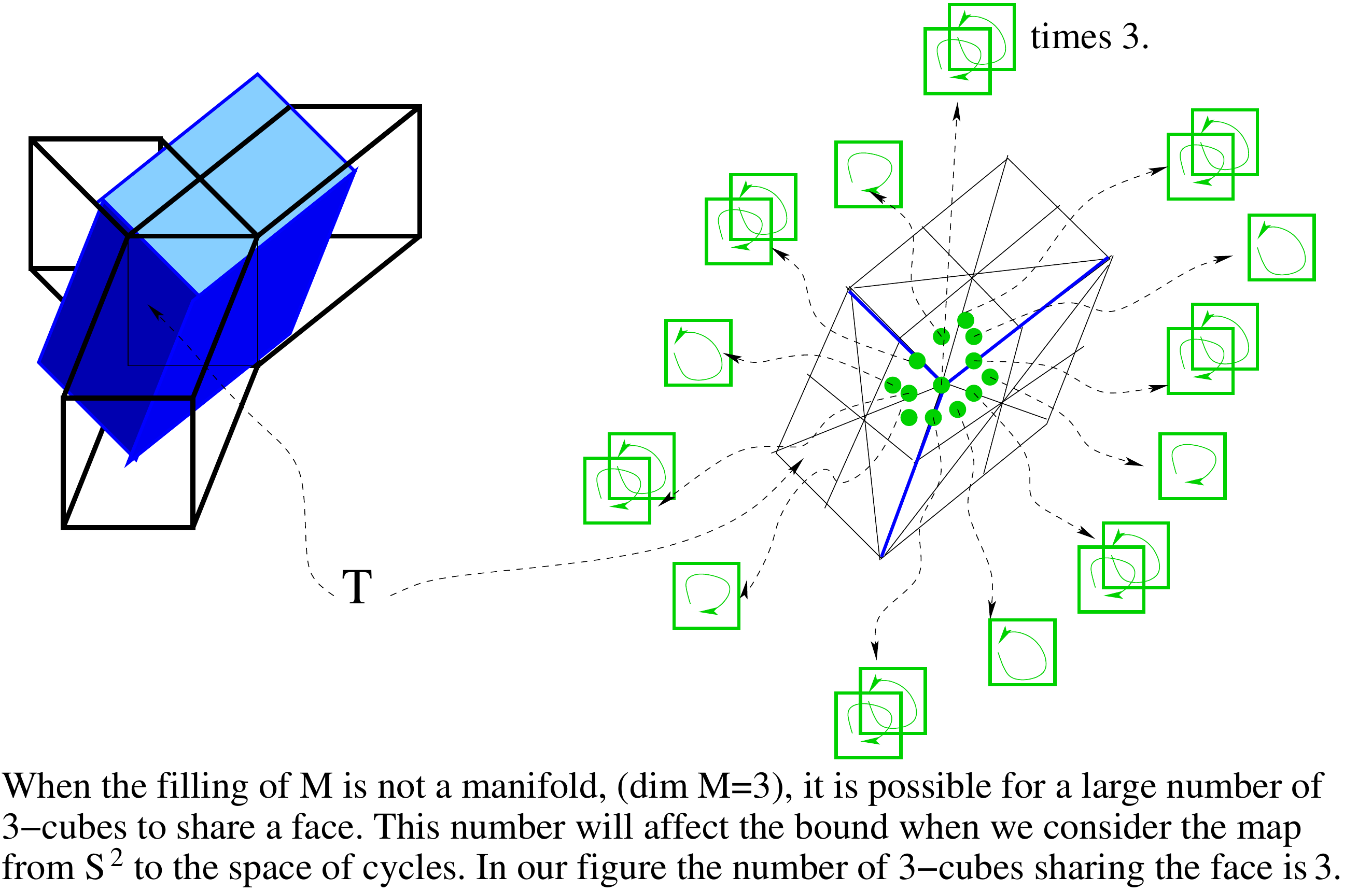}
\caption{Multiple cubes sharing a face}
\label{fig:multiplicity}
\end{figure}

\medskip

Let us extend this family~$\Xi_x$ of $1$-cycles to $\partial \Delta \times [0,1]$.
By Claim~\ref{claim:pairs} (and the observation following~\eqref{eq:hbar(x)}), the $1$-cycle~$\Xi_x$ with $x \in \partial \Delta$ can be seen as a graph where the (algebraic) sum of each edge~$[a,b]$ vanishes.
Denote by~$m$ the midpoint of~$[a,b]$.
For $(x,t) \in \partial \Delta \times [0,1]$, define the $1$-cycle~$\Xi_{x,t}$ by replacing each edge~$[a,b]$ of~$\Xi_x$ with $[a,a_t] \cup [b_t,b]$ (keeping the same multiplicity), where $a_t=ta+(1-t)m$ and $b_t=tb+(1-t)m$ in barycentric coordinates; see Figure~\ref{fig:collar}.
Observe that $\Xi_{x,t}$ is a continuous family of $1$-cycles which agrees with $\Xi_x$ for $t=0$ and with a union of points corresponding to the vertices of~$\Xi_x$ for~$t=1$.
Thus, we obtain a family of $1$-cycles~$\Xi_u$ with $u \in B^{n-1}=\Delta \cup (\partial \Delta \times [0,1])$, where the peripheral cycles~$\Xi_u$ for $u \in \partial B^{n-1}$ are unions of points and 
\begin{equation} \label{eq:Wtilde}
\max_{C \, \subseteq \, \Xi_u} \length(\bar{f}_{|C}) \leq 4 \delta \kappa'_n < \w'_1(M)
\end{equation}
for every $u \in B^{n-1}$.

\begin{figure}[!htbp]
\centering
\includegraphics[width=16cm]{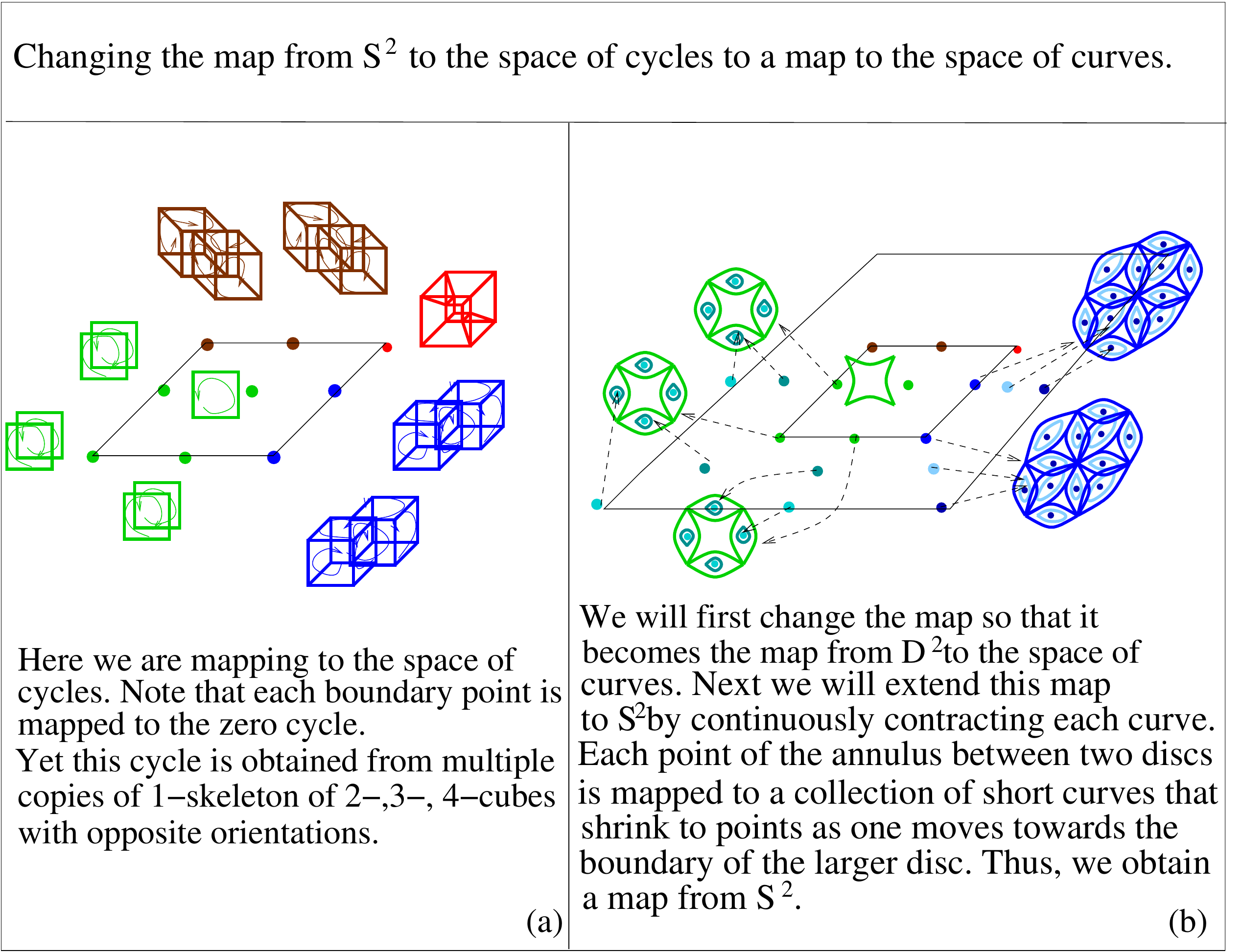}
\caption{Extension of the family of cycles to the collar of~$\Delta$}
\label{fig:collar}
\end{figure}

\medskip

By construction, the inverse image under $h:N \to T$ of the interior~$\mathring{\Delta}_i$ of an $(n-1)$-simplex~$\Delta_i$ of~$\mathcal{T}_T$ identifies with the product $\mathring{\Delta}_i \times S^1$, where each factor $\{x\} \times S^1$ agrees with the restriction of~$\Xi_x$ to~$h^{-1}(\mathring{\Delta}_i)$ for every $x \in \mathring{\Delta}_i$.
This implies that the family of~$1$-cycles~$\Xi_u$ induces a nontrivial class in 
\[
\pi_{n-1}(\mathcal{Z}_1(N;G),\{0\}) \simeq H_n(N;G) \simeq G
\]
under the Almgren isomorphism~\cite{alm}, where $G=\Z$ if $M$ (and so~$N$) is orientable and $G=\Z_2$ otherwise.
By definition of~$\w'_1(M)$, see~\eqref{eq:defWtilde}, we derive from~\eqref{eq:Wtilde} that the image by~$\bar{f}$ of the family of $1$-cycles~$\Xi_u$ does not represent the fundamental class of~$M$. %induces a trivial class in homotopy.
Thus,
\[
\bar{f}_*([N]) \neq [M] \in H_n(M;G).
\]

On the other hand, since the pseudomanifolds~$N$ and~$\partial P$ are homologuous in~$Q \cup \partial P$ and the map~$\bar{f}:Q \cup \partial P \to M$ agrees with a deformation of~$\sigma$ on~$\partial P$, we deduce that
\[
\bar{f}_*([N]) = \bar{f}_*([\partial P]) = \sigma_*([\partial P]) = [M].
\]
Hence a contradiction.
\end{proof}

\begin{remark}
In the proof of Theorem~\ref{theo:FR1bis}, we did not define the family~$\Xi_x$ of $1$-cycles of~$N$ with $x \in S^{n-1}$ as the inverse images of some map $N \to S^{n-1}$, but as a perturbation/extension of the family given by the inverse images of $\hbar:N \to \Delta^{n-1}$.
It would be possible to do so by pushing apart the $1$-cycles and inserting small new $1$-cycles along the lines of the examples given in Subsection~\ref{subsec:gmt}.
This would allow us to replace $\w'_1(M)$ in Theorem~\ref{theo:FR1bis} with $\bar{\w}_1(M)$ defined in~\eqref{eq:wp}.
\end{remark}

We can refine the inequality of Theorem~\ref{theo:FR1bis} by considering the invariant~$\w''_1(M)$ instead of~$\w'_1(M)$; see~\eqref{eq:Wtilde2}.

\begin{theorem} \label{theo:FR1ter}
Let $M$ be a closed $n$-manifold with $n \geq 3$.
Then every Riemannian metric on~$M$ satisfies 
\[
\FR(M) \geq c_n \, \w''_1(M)
\]
for some explicit positive constant $c_n$ depending only on~$n$.
\end{theorem}

\begin{proof}
We argue as in the proof of Theorem~\ref{theo:FR1bis}, using the same notations and pointing out only the differences.
Since $n \geq 3$, by slightly perturbing~$\bar{f}$, we can assume that the restriction of~$\bar{f}$ to~$P^1$ is an embedding into~$M$.
By construction, the images by~$\bar{f}$ of the fibers~$h^{-1}(t)$ with $t \in T \setminus T''$ lie in small cubes~$\bar{f}(C^n)$ of~$M$ with $C^n \subseteq \partial P$.
Furthermore, these images are pairwise disjoint in~$M$ and do not intersect the graph $\bar{f}(P^1)$.
Likewise, the fibers $h^{-1}(t)$ with $t \in T \setminus T''$ are isomorphic to the $1$-skeleton of a cube of dimension at most~$n+1$, and their images by~$\bar{f}$ lie in the graph~$\bar{f}(P^1)$.
We can slightly modify~$\bar{f}$ so that the images of the fibers $h^{-1}(t)$ with $t \in T \setminus T''$ are disjoint in~$M$.
Perturbing also the map $\jmath:T \to \Delta$, we can further assume the following.
For every $t \in \jmath^{-1}(x)$, denote by~$C^{n-1}_t$ the $(n-1)$-cube of~$T$ containing~$t$ (or one of them).
For every $t' \in \jmath^{-1}(x)$ not lying in an $(n-1)$-cube of~$T$ intersecting $C^{n-1}_t$, the image of~$h^{-1}(t')$ is disjoint from the image of~$h^{-1}(t)$.
Now, by the $\kappa_n$-bounded local complexity of~$P$, the number of fibers~$h^{-1}(t')$ lying in an $(n-1)$-cube of~$T$ intersecting~$C^{n-1}_t$ is bounded by an explicit constant depending only on~$n$.
As a result, we obtain a family~$\Xi'_u$ of $1$-cycles of~$N$ with $u \in S^{n-1}$ such that every connected component~$C'$ of~$\bar{f}(\Xi'_u)$ satisfies
\[
\length(C') \leq \kappa_n'' \, \delta
\]
for some explicit constant~$\kappa_n''$ depending only on~$n$.
Since the $1$-cycle family~$\Xi'_u$ is a deformation of the original $1$-cycle family~$\Xi_u$, the homotopy class it induces in $\pi_1(\mathcal{Z}_1(N;G),\{0\})$ is nontrivial.
We conclude as in the proof of Theorem~\ref{theo:FR1bis}.
\end{proof}

\section{Filling radius and relative homotopy $1$-waist} \label{sec:FRhomotopy}

We adapt the argument of Section~\ref{sec:FR1} to establish a lower bound on the filling radius of a closed Riemannian manifold in terms of its relative homotopy $1$-waist and derive Theorem~\ref{theo:B}.

\begin{theorem} \label{theo:FR2}
Fix $k \leq n-1$.
Let $M$ be a closed $n$-manifold and $\Phi:M \to K$ be a continuous map to a CW-complex~$K$ with $\pi_i(K) = 0$ for every $i \geq k+1$.
Suppose that $\Phi_*([M]) \neq 0 \in H_n(K;G)$ for some homology coefficient group~$G$.
Then every Riemannian metric on~$M$ satisfies 
\[
\FR(M) \geq c_n \, \w_{1,k}(M,\Phi)
\]
for some explicit positive constant~$c_n$ depending only on~$n$.
\end{theorem}

\begin{proof}
We are going to give a proof by contradiction. Initially, we argue as in the proof of Theorem~\ref{theo:FR1} using the same notations.
Suppose that there exists a continuous map $\sigma:P \to U_\nu(M) \subseteq L^\infty(M)$ defined on a compact cubical $(n+1)$-pseudomanifold such that the restriction $\sigma:\partial P \to M$ satisfies 
\[
\sigma_*([\partial P]) = [M] \in H_n(M;G)
\]
with $\nu < \frac{1}{(k+1) \, 2^{k+1}} \w_{1,k}(M,\Phi)$.
As in the proof of Theorem~\ref{theo:FR1}, we construct a map $f:P^1 \to M$ which agrees with~$\sigma$ on the $1$-skeleton of~$\partial P$ so that the lengths of the images of the edges of~$P^1$ are less than $\delta=2 \nu + \varepsilon < \frac{1}{(k+1) \, 2^{k}} \w_{1,k}(M,\Phi)$.

\medskip

Let~$P_*^{k+1}$ be the cubical $(k+1)$-complex formed of the cubical $(k+1)$-simplices of~$P$ not lying in~$\partial P$.
Denote by~$Q^{k+1} \subseteq P_*^{k+1}$ the cubical $(k+1)$-complex formed of the pieces $X_1^{k+1} \subseteq C^{k+1}$, where $C^{k+1}$ is a cubical $(k+1)$-simplex of~$P$ not lying in~$\partial P$.
We define a continuous map $\bar{f}:Q^{k+1} \cup \partial P \to M$ from $f:P^1 \to M$ which coincides with the deformation~$\bar{\sigma}:\partial P \to M$ of~$\sigma$ and takes every edge of~$P^1$ to a segment of length less than~$\delta$; see the proof of Theorem~\ref{theo:FR1} for the details of the construction.

\medskip

Let us extend this map to the $(k+1)$-skeleton~$P^{k+1}$ of~$P$.
Define the cubical $k$-complex~$R'$ and the cubical $(k-1)$-complex~$T'$ by respectively pasting together the pieces $Y^k \subseteq C^{k+1}$ and the pieces~$Z^{k-1} \subseteq C^{k+1}$, where $C^{k+1}$ is a cubical $(k+1)$-simplex of~$P$ not lying in~$\partial P$.
Denote by $h':R' \to T'$ the map whose restriction to~$Y^k$ agrees with the map $\theta:Y^k \to Z^{k-1}$ defined in~\eqref{eq:theta} with $p=1$.
Similarly, define the cubical $k$-complex~$R''$ and the cubical $(k-1)$-complex~$T''$ by respectively pasting together the pieces $X_2^k \subseteq C^k$ and the pieces~$Z^{k-2} \times [0,\frac{1}{2}]$ with $Z^{k-2} \subseteq C^k$, where $C^k$ is a cubical $k$-simplex of~$\partial P$.
Denote by $h'':R'' \to T''$ the map whose restriction to~$X_2^k$ agrees with the map $\Theta:X_2^k \to Z^{k-2} \times [0,\frac{1}{2}]$ defined in~\eqref{eq:fhat} with $p=1$.
The two maps~$h'$ and~$h''$ so-defined agree on the intersection $R' \cap R''$ formed of the pieces~$Y^{k-1} \subseteq C^k$ lying in~$\partial P$, after identifying $Z^{k-2} \subseteq Z^{k-1} \subseteq T'$ and $Z^{k-2} = Z^{k-2} \times \{ \frac{1}{2} \} \subseteq T''$, where $Z^{k-2} \subseteq C^k$ lies in~$\partial P$.
Put together, these maps give rise to a continuous map
\[
h:R \to T
\]
from the cubical $k$-complex~$R=R' \cup R''$ lying in~$P^{k+1}$ to the cubical $(k-1)$-complex $T=T' \cup_S T''$ obtained by gluing~$T'$ and~$T''$ along the cubical $(k-2)$-complex~$S$ formed of the pieces $Z^{k-2} \subseteq C^k$ lying in~$\partial P$.

\medskip

Consider the composite map 
\[
F=\Phi \circ \bar{f}: Q^{k+1} \cup \partial P \to K
\]
extending $\Phi \circ \bar{\sigma}: \partial P \to K$.
%(Recall that $Q^{k+1}$ lies in~$ P_*^{k+1}$.)
By construction, every fiber~$h^{-1}(t) \subseteq R$ with $t \in T$ agrees with a fiber of~$\theta$ or~$\Theta$, and is isomorphic to the $1$-skeleton of a cube of dimension at most~$k+1$.
Thus, every fiber of~$h$ is sent by~$\bar{f}$ to a graph of length at most $(k+1) \, 2^k \, \delta < \w_{1,k}(M,\Phi)$.
By definition of the $\Phi$-relative homotopy $k$-waist, see Definition~\ref{def:width}, the restriction $F_{|R}: R \to K$ is homotopic to a map 
\[
R \overset{h}{\to} T \to K
\]
which factors out through~$h$.
Thus, the map $F:Q^{k+1} \cup \partial P \to K$ extends to
\[
R \times [0,1] / \! \! \sim
\]
where $(x,1) \sim (y,1)$ if and only if $h(x) = h(y)$.
Since the complement of the interior of~$Q^{k+1}$ in~$P_*^{k+1}$ is homeomorphic to $R \times [0,1] / \! \! \sim$, this yields a map $F:P_*^{k+1} \cup \partial P \to K$ defined in particular on the $(k+1)$-skeleton~$P^{k+1}$ of~$P$.

\medskip

Now, since $\pi_i(K)=0$ for every $i \geq k+1$, the map~$F$ further extends into $F:P \to K$.
Recall that the restriction of~$F$ to~$\partial P$ agrees with~$\Phi \circ \bar{\sigma}$.
Therefore, the homology class 
\[
(\Phi \circ \bar{\sigma})_*([\partial P]) = \Phi_*([M]) \in H_n(K;G)
\]
is trivial.
Hence a contradiction.
\end{proof}

Theorem~\ref{theo:B} follows from Theorem~\ref{theo:FR2} and Theorem~\ref{theo:FR}. 

\forget
\medskip

Let us conclude this section with an example of a homotopy $1$-sweepout which does not proceed from a homology $1$-sweepout, following a construction of I.~Babenko~\cite[\S5.3]{bab06} in systolic geometry.

\begin{example}
Consider the cellular $2$-complex
\[
X = \RP^2 \cup_f B^3
\]
obtained by attaching to~$\RP^2$ a $3$-cell~$B^3$ along a map $f:S^2 \to \RP^2$ representing $2\alpha \in \pi_2(\RP^2)$, where $\alpha$ is a generator of~$\pi_2(\RP^2) \simeq \Z$.
Note that $\pi_1(X) \simeq \pi_1(\RP^2) \simeq \Z_2$ and $H_3(X;\Z) \simeq 0$.
The cellular classifying map of~$X$ decomposes as 
\[
X \overset{\varphi}{\longrightarrow} M \overset{\Phi}{\longrightarrow} K
\]
with $M= \RP^3$ and $K=K(\Z_2,1) =\RP^\infty$, where $\Phi:M \to K$ is the inclusion map.
The composite map $\Phi \circ \varphi:X \to K$ is not homotopic to a map $X \to T \to K$ which factors out through a simplicial $2$-complex~$T$.
Otherwise, the classifying map $\Phi \circ \varphi$ could be deformed into a map to the $2$-skeleton~$\RP^2$ of~$\RP^\infty$ giving rise to an extension of~$f$ to~$B^3$, which is impossible.
Thus, every slicing of~$X$ gives rise to a homotopy $(1,2)$-sweepout of~$M$.
\end{example}
\forgotten

\section{Filling radius, Urysohn width and $1$-waist} \label{sec:UW}

Using the filling estimate established in Theorem~\ref{theo:FR1}, we show that the filling radius of a closed Riemannian manifold is roughly equal to its homology $1$-waist.

\medskip

We need to introduce the following notion related to the $1$-waist.

\begin{definition} \label{def:urysohn}
The \emph{Urysohn width} of a closed Riemannian $n$-manifold~$M$, denoted by~$\UW(M)$, is the infimum of the distances~$\delta$ such that there exists a continuous map $M \to X$ from~$M$ to a simplicial $(n-1)$-complex whose fibers have diameters less than~$\delta$.
Strictly speaking, this definition corresponds to the notion of Urysohn $(n-1)$-width.
\end{definition}

Though the homology $1$-waist~$\w(M)$ and the Urysohn width~$\UW(M)$ are defined in similar terms, the two notions present some differences.
First, the homology $1$-waist measures the $1$-polyhedron length, while the Urysohn width measures the $1$-polyhedron diameter. % of one-cycles and, more generally, fiber sets.
Second, the homology $1$-waist is concerned with homology $1$-sweepouts made of $1$-polyhedra which may intersect each other, while, by definition, the fibers involved in the definition of the Urysohn width are disjoint.
Still, the two notions are connected through the filling radius estimate of Theorem~\ref{theo:FR1} and the general bound
%The filling radius and the Urysohn width are related through the following general bound, see~\cite[Appendix~1]{gro83},
\begin{equation} \label{eq:FR-UW}
\FR(M) \leq \tfrac{1}{2} \, \UW(M)
\end{equation}
obtained in~\cite[Appendix~1]{gro83} satisfied by every closed pseudomanifold.

\medskip

The following result extends the bound~\eqref{eq:FR-UW} to the homology $1$-waist~$\w(M)$. % replacing the filling radius with the rational filling radius.

\begin{proposition} \label{prop:FRQ}
Every closed Riemannian $n$-manifold~$M$ satisfies 
\[
\FR(M) \leq \tfrac{1}{2} \, \w(M).
\]
\end{proposition}

\begin{proof}
By definition of the homology $1$-waist, see Definition~\ref{def:width} where $p=1$, there exist a continuous map $h:N \to T$ from a closed $n$-pseudomanifold~$N$ to a finite simplicial $(n-1)$-complex~$T$ and a degree one map $\varphi:N \to M$ 
%with 
%\[
%\varphi_*([N]) \neq 0 \in H_n(M;\Z)
%\]
such that 
\begin{equation} \label{eq:eta}
\length \, \varphi_{|h^{-1}(t)} < \w(M) + 2\varepsilon
\end{equation}
for every $t \in T$, where $\varepsilon >0$ is any given positive real.
Slightly perturbing~$\varphi$ if necessary, we can always assume that $\varphi$ is piecewise smooth.
Consider the metric~$g_N= \varphi^*g_M+ \lambda^2 g_0$ on~$N$, where $\varphi^*g_M$ is the pull-back of the metric~$g_M$ on~$M$ under~$\varphi$, and $g_0$ is a fixed metric on~$N$ with $\lambda>0$ arbitrarily small.
The following chain of inequalities holds
\[
\FR(M) \leq \FR(N) \leq \tfrac{1}{2} \, \UW(N) \leq \tfrac{1}{2} \, \w(M) + \varepsilon
\]
where each inequality can be justified as follows.
By~\cite[p.~6]{gro83}, the filling radius does not increase under $1$-Lipschitz degree one maps.
Applying this result to the contracting map $\varphi:(N,g_N) \to (M,g_M)$ yields the first inequality.
The second inequality is given by~\eqref{eq:FR-UW}.
By construction, the diameter of the fibers~$h^{-1}(t)$ of $h:N \to T$ is less than~$\w(M)+2\varepsilon$, see~\eqref{eq:eta}, which implies the third inequality.
Now, letting $\varepsilon$ go to zero in the previous inequality chain, we obtain the relation $\FR(M) \leq \frac{1}{2} \, \w(M)$.
\end{proof}

Theorem~\ref{theo:C} follows from Theorem~\ref{theo:FR1} and Proposition~\ref{prop:FRQ}.

\section{Hypersphericity, filling radius and Urysohn width} \label{sec:HS}

Using the filling estimate established in Theorem~\ref{theo:FR1}, we derive that the filling radius and the hypersphericity of a closed orientable Riemannian manifold can be arbitrarily far apart.

\begin{definition} \label{def:hypersphericity}
The \emph{hypersphericity} of a closed orientable Riemannian $n$-manifold~$M$, denoted by~$\HS(M)$, is the supremum of the radii~$R$ such that there exists a $1$-Lipschitz map $M \to S^n(R)$ of nonzero degree from~$M$ to the standard $n$-sphere~$S^n(R)$ of radius~$R$.
%That is,
%\[
%\HS(M) = \sup \{ R \mid \mbox{there exists } M \to S^n(R) \mbox{ $1$-Lipschitz with nonzero degree} \}.
%\]
%Similarly, the \emph{degree one hypersphericity} of~$M$, denoted by~$\HS_1(M)$, is the supremum of the radii~$R$ such that there exists a $1$-Lipschitz map $M \to S^n(R)$ of degree one.
%\[
%\HS_1(M) = \sup \{ R \mid \mbox{there exists } M \to S^n(R) \mbox{ $1$-Lipschitz of degree one} \}.
%\]
%Note that $\HS_1(M) \leq \HS(M)$.
\end{definition}

The hypersphericity is related to the filling radius and the Urysohn width through the following inequalities.

\begin{proposition} \label{prop:HS-FR-UW}
Let $M$ be a closed orientable $n$-manifold.
Then
\[
\tfrac{1}{2} \arccos(-\tfrac{1}{n+1}) \, \HS(M) \leq \FR(M) \leq \tfrac{1}{2} \, \UW(M).
\]
\end{proposition}

\begin{proof}
The second inequality comes from~\eqref{eq:FR-UW}.
For the first inequality, it is convenient to work with the rational filling radius, which is defined in a similar way as the standard filling radius, see Definition~\ref{def:FR}, except that the homology coefficients are in~$\Q$.
%In order to avoid technical difficulties related to torsion in homology, it is sometimes more convenient to work with rational filling radius; see Definition~\ref{def:FR}.
It follows from abstract nonsense using the relation $H_n(X;\Q) \simeq H_n(X;\Z) \otimes \Q$ for simplicial complexes~$X$ given by the universal coefficient theorem for homology that 
\[
\FR_\Q(M) \leq \FR(M).
\]
By~\cite[p.~6]{gro83}, the rational filling radius does not increase under $1$-Lipschitz maps of nonzero degree.
Thus, 
\[
\FR_\Q(S^n) \, \HS(M) \leq \FR_\Q(M).
\]
Now, the filling radius of the standard sphere~$S^n$ has been computed in~\cite{katz83} and the argument extends to the rational filling radius.
More specifically,
\[
\FR_\Q(S^n) = \tfrac{1}{2} \arccos(-\tfrac{1}{n+1}).
\]
Hence the proposition.
\end{proof}

\begin{remark}
A direct inequality between the hypersphericity and the Urysohn width can be found in~\cite{gro88}, \cite[2.12$\frac{1}{2}_+$]{gro99} and~\cite{guth05}.
%\[
%\HS(M) \leq \tfrac{\pi}{2} \, \UW(M).
%\]
\end{remark}

Proposition~8.3 of~\cite{BS10}, combined with the previous proposition or remark, shows that for every Riemannian metric on~$S^2$, these geometric invariants are roughly the same
\[
\HS(S^2) \simeq \FR(S^2) \simeq \UW(S^2).
\]
In higher dimension, examples showing that the hypersphericity and the Urysohn width can be arbitrarily far apart have first been constructed in~\cite[\S5]{guth05}.
Using the relationship between the filling radius and the homology $1$-waist established in Theorem~\ref{theo:FR1}, we show that a similar phenomenon occurs between the hypersphericity and the filling radius.

\begin{theorem}
There exists a sequence~$(g_i)$ of Riemannian metrics on~$S^4$ with arbitrarily small hypersphericity and filling radius bounded away from zero.
\end{theorem}

\begin{proof}
We argue as in~\cite[\S5]{guth05}.
Let $\HP^2$ be the quaternionic projective plane (of dimension~$8$) with the standard homogeneous metric.
Consider a sequence $S_i^4 \subseteq \HP^2$ of $4$-spheres with their induced metrics, representing $\HP^1 \simeq S^4$ in homology, which Gromov-Hausdorff converges to~$\HP^2$.
Such an approximating sequence exists; see~\cite{FO} and~\cite[\S5]{guth05}.
The girth of the inclusion maps $f_i:S_i^4 \to \HP^2$ tend to zero.
That is, there exists a finite open cover~$\{ U_k^i \}$ of~$\HP^2$ such that every preimage $f_i^{-1}(U_k^i)$ has an arbitrarily small radius for $i$ large enough.
By~\cite[Lemma~5.2]{guth05}, every $1$-Lipschitz map $S_i^4 \to S^4(R)$ to the round sphere of radius~$R$ is homotopic to a map which factors out through
\begin{equation} \label{eq:split}
S_i^4 \overset{f_i}{\longrightarrow} \HP^2 \to S^4(R)
\end{equation}
for $i$ large enough.
This implies that the map $S_i^4 \to S^4(R)$ has zero degree.
Otherwise, the induced homomorphism $H^4(S^4(R)) \to H^4(S_i^4)$ would be nonzero.
In particular, the homomorphism $H^*(S^4(R)) \to H^*(\HP^2)$ induced by the second map in~\eqref{eq:split} takes the fundamental cohomology class~$\alpha \in H^4(S^4(R))$ to a nonzero class $\omega \in H^4(\HP^2)$.
By naturality of the cup product, this homomorphism takes the product $\alpha \cup \alpha \in H^8(S^4(R))$ to the product $\omega \cup \omega \in H^8(\HP^2)$, which is impossible since $\alpha \cup \alpha = 0$ and $\omega \cup \omega \neq 0$.
Therefore, the hypersphericity of~$S_i^4$ tends to zero.

\medskip

On the other hand, suppose that the filling radius of~$S_i^4$ is not bounded away from zero.
Using the relationship between the filling radius and the homology $1$-waist established in Theorem~\ref{theo:FR1}, there exist a continuous map $h_i:N_i \to T_i$ from a closed orientable $4$-pseudomanifold~$N_i$ to a finite simplicial $3$-complex~$T_i$ and a degree one map $\varphi:N_i \to S_i^4$ 
%with 
%\[
%(\varphi_i)_*([N_i]) \neq 0 \in H_4(S_i^4)
%\]
such that 
\begin{equation} \label{eq:epsilon}
\length \, {\varphi_i}_{|h_i^{-1}(t)} < \varepsilon
\end{equation}
for every $t \in T_i$, where $\varepsilon >0$ is any given positive real.
Define a metric~$g_N$ on~$N$ by pulling-back the metric on~$M$ as in the proof of Proposition~\ref{prop:FRQ} so that the map $\varphi_i:N_i \to S_i^4$ is contracting.
By construction, the length of the fibers $h_i^{-1}(t)$ is less than~$\varepsilon$; see~\eqref{eq:epsilon}.
Thus, the girth of $h_i:N_i \to T_i$ is less than~$\varepsilon$.
By~\cite[Lemma~5.2]{guth05}, the contracting map $f_i \circ \varphi_i:N_i \to \HP^2$ is homothetic to a map which factors out through
\[
N_i \overset{h_i}{\longrightarrow} T_i \to \HP^2.
\]
Since $T_i$ is a simplicial $3$-complex, this implies that the homology class $(f_i \circ \varphi_i)_*([N_i])$ vanishes in~$H_4(\HP^2)$, which contradicts the injectivity of the homomorphisms $(\varphi_i)_*$ and~$(f_i)_*$ in homology.
\end{proof}

\section{Filling radius and homology $p$-waist} \label{sec:pwidth}

We establish a lower bound on the filling radius of a closed Riemannian manifold in terms of its homology $p$-waist and its homological filling functions.
As a consequence, we derive Theorem~\ref{theo:A}.

\medskip

Let us first introduce the following replacement transformation of a map defined on a cubical complex based on the notion of homological filling functions; see Definition~\ref{def:FH}.
For $\varepsilon >0$ small enough, define $\FHb_k^\varepsilon(v) = \FHb_k(v) + \varepsilon$.
In order to keep the notations simple and despite the risk of confusion, we will continue to write $\FHb_k$ for~$\FHb_k^\varepsilon$ in the following proposition.

\begin{proposition} \label{prop:cage}
Let $M$ be a closed Riemannian $n$-manifold and $p$ be a positive
number with $p\leq n$.
For every cubical $i$-complex~$K^i$ with $i \leq p$ and every continuous map $f:K^1 \to M$ defined on the $1$-skeleton of~$K^i$, sending every edge of~$K^1$ to a minimizing segment of~$M$ of length at most~$\delta>0$, there exists a continuous extension $F:\mathcal{X}^i \to M$ of $f:K^1 \to M$ defined on a cubical $i$-complex~$\mathcal{X}^i$ containing~$K^1$ with 
\begin{equation} \label{eq:volp}
\vol_j(F_{|C}) \leq \FHb_{j-1} \circ \cdots \circ \FHb_1(\delta)
\end{equation}
for every cubical $j$-simplex~$C$ of~$\mathcal{X}^i$ with $j \leq p$.
The extension $F:\mathcal{X}^i \to M$ is called the \emph{$\mathcal{R}$-transformation}\footnote{$\mathcal{R}$ stands for ``replacement".} of $f:K^1 \to M$ modeled on~$K^i$ (or simply the $\mathcal{R}$-transformation of~$f$ if the model space~$K^i$ is a cube or is implicit).

Furthermore, the $\mathcal{R}$-transformation can be defined so as to satisfy the following properties:
\begin{enumerate}
%\item (One-dimensional case) If $i=1$ then $\mathcal{X}^1=K^1$, and the map $F:K^1 \to M$ takes every vertex of~$K^1$ to its image by~$f$ and every edge of~$K^1$ to a minimizing segment of~$M$. \label{assign1}
%\item (Volume bound) Assume that for some $\delta>0$, the length of the image in $M$ of each edge of $K^i$ under $f$ does not exceed $\delta$. Then, for each cubical $j$-simplex~$C$ in $K^i$,
%\begin{equation} \label{eq:volp}
%\vol_j(F_{|C}) \leq \FHb_{j-1} \circ \cdots \circ \FHb_1(\delta).
%\end{equation} \label{assign2}
\item (Triviality) If %$f$ is an inclusion (i.e. a $1-1$ map), and 
$f:K^i \to M$ is continuous on~$K^i$ and if the volume bound
\begin{equation} \label{eq:fC}
\vol_j(f_{|C}) \leq \FHb_{j-1} \circ \cdots \circ \FHb_1(\delta)
\end{equation}
holds for each cubical $j$-simplex~$C$ of~$K^i$ with $j \leq i$, then $\mathcal{X}^i=K^i$ and $F=f$. \label{assign3}
\item (Coherence) If ${K}_j^{i_j}$ (with $j=1,2$) are two cubical $i_j$-complexes with ${K}_1^{i_1} \subseteq {K}_2^{i_2}$ and $f_j:{K}_j^1 \to M$ are two continuous maps which coincide on~${K}_1^1$, where $i_j \leq p$, then $\mathcal{X}_1^{i_1} \subseteq \mathcal{X}_2^{i_2}$ and the two corresponding maps $F_j:\mathcal{X}_j^{i_j} \to M$ coincide on~$\mathcal{X}_1^{i_1}$. \label{assign4}
\item (Commutation with the boundary operator $\partial$)
If $K^i$ is a closed $i$-cube and $e_1,\ldots, e_{2i}$
denote its $(i-1)$-faces, then $\mathcal{X}^i=\Sigma^i$ is a compact $i$-pseudomanifold whose boundary~$\partial \Sigma^i$ agrees with the union of the
$2i$ cubical $(i-1)$-complexes/pseudomanifolds $\mathcal{Y}_j$ corresponding to the domains of the $\mathcal{R}$-transformations of the restrictions $f_{|e_j^1}:e_j^1 \to M$ modeled on~$e_j$ for $j=1,\ldots, 2i$. \label{assign5}
%The restriction of $F$ to $Y_j$ coincides
%with the map assigned by $A$ to $(e_j, f|_{e_j})$.
%Moreover, $\mathcal{Y}_{j_1}$ and $\mathcal{Y}_{j_2}$ intersect if and only $e_{j_1}$ and $e_{j_2}$ intersect. In this case, $\mathcal{Y}_{j_1}\cap \mathcal{Y}_{j_2}$ is the cubical $(i-2)$-complex corresponding to the domain of the $\mathcal{R}$-transformation of the restriction $f_{|e_{j_1}\cap \, e_{j_2}}:e_{j_1}\cap e_{j_2} \to M$. 
\end{enumerate}
\forget
assigns to each pair $(K^i, f)$ 
a pair $(\mathcal{X}^i,F)$ and satisfies the five axioms below,
where $i\in\{1,\ldots, p\}$, $K^i$, $\mathcal{X}^i$ are $i$-dimensional
cubical complexes, $f:K^i\to M$ and $F:\mathcal{X}^i\to M$
are continuous maps:
\par\noindent
(1) (One-dimensional case) If $i=1$, $\mathcal{X}^i=K^i$, the restriction
of $F$ to the $0$-skeleton of $K^1$ coincides with $f$, and each
edge of $K^1$ is mapped to a minimizing geodesic between its
endpoints.
\par\noindent
(2) (Volume bound) Assume that for some $\delta$ the length
of the image of each edge of $K^i$ in $M$ under $f$ does not exceed $\delta$. Then for each $j$-dimensional cubic cell $C$
in $K^i$
\begin{equation} \label{eq:volp}
\vol_i(F(C)) \leq \FHb_{j-1} \circ \cdots \circ \FHb_1(\delta).
\end{equation}
\par\noindent
(3) (Triviality) If %$f$ is an inclusion (i.e. a $1-1$ map), and 
for each $j$-dimensional cubic cell $C$ in $K^i$
$\vol_i(F(C)) \leq \FHb_{j-1} \circ \cdots \circ \FHb_1(\delta)$,
then $\mathcal{X}^i=K^i$ and $F=f$.
\par\noindent
%For every positive integer~$p$ and every continuous map $f:{K}^p \to M$ from a cubical $p$-simplex~${K}^p$ to~$M$ whose length of the image of the edges of~${K}^p$ is less than~$\delta$, we can associate a continuous map $F:\Sigma^p \to M$ from a compact simplicial $p$-pseudomanifold~$\Sigma^p$ with boundary to~$M$ such that
%(2) (Volume bound) Assume that for some $\delta$ the length
%of the image of each edge of $K^i$ in $M$ under $f$ does not exceed $\delta$. Then
%\begin{equation} \label{eq:volp}
%\vol_i(F) \leq \FHb_{i-1} \circ \cdots \circ \FHb_1(\delta).
%\end{equation}
%(Remark: It is easy to see that if $f$ is an inclusion (3) would automatically follow from (2).)
%\par\noindent
%Furthermore, this replacement procedure is coherent. That is,
%\begin{enumerate}
%\item 
(4) (Coherence) If ${K}_i^{i_j}$ (with $j=1,2$) are two cubical $p_i$-simplices with ${K}_1^{i_1} \subseteq {K}_2^{i_2}$ and $f_j:{K}_i^{i_j} \to M$ are two continuous maps which coincide on~${K}_1^{i_1}$, then $\Sigma_1^{i_1} \subseteq \Sigma_2^{i_2}$ and the two corresponding maps $F_i:\Sigma_i^{i_i} \to M$ coincide on~$\Sigma_1^{i_1}$. %\label{cage1}
\par\noindent
%\item 
(5) ($A$ commutes with the boundary operator $\partial$)
If $K^i$ is a closed $i$-dimensional cube, and $e_1,\ldots, e_{2i}$
its $(i-1)$-dimensional faces, then $\mathcal{X}^i$ is a pseudomanifold
that can be represented as the union of
$2i$ $(i-1)$-complexes $Y_j$ assigned by $A$ to $(e_j, f|_{e_j})$,
$j=1,\ldots, 2i$. 
%The restriction of $F$ to $Y_j$ coincides
%with the map assigned by $A$ to $(e_j, f|_{e_j})$.
Moreover, $Y_{j_1}$ and $Y_{j_2}$ intersect
if and only $e_{i_1}$ and $e_{i_2}$ intersect. In this case
$Y_{j_1}\bigcap Y_{j_2}$ is the $(j-2)$-dimensional complex
that $A$ assigns to $(e_{i_1}\bigcap e_{j_2}, f|_{e_{j_1}\bigcap e_{j_2}})$. %, and the restriction of $F$ to $Y_{j_1}\bigcap Y_{j_2}$
%is the map that $A$ assigns to $(e_{i_1}\bigcap e_{j_2}, %f|_{e_{j_1}\bigcap e_{j_2}})$.
%More generally, if $K^i$
%is a $i$-dimensional pseudomanifold with boundary $\partial K_i$,
%$X^i$ is a pseudomanifold with a boundary $\partial X^i$ composed
%of pieces assigned by $A$ to faces of $\partial K^i$ and restrictions of $f$ to these faces. The map assigned by $A$ to $(\partial K^i,f_{\partial K^i})$ coincides with the restriction of $F$ to $\partial X^i$. Somewhat informally,we can express this property as
%$\partial A(K^i, f)= A(\partial K^i, f|_{\partial K^i})$.

%If $\partial {K}^p = \cup_i {K}_i^{p-1}$, where ${K}_i^{p-1}$ are the $(p-1)$-faces of~${K}^p$, then $\partial \Sigma^p = \cup_i \Sigma_i^{p-1}$, where $F_i:\Sigma_i^{p-1} \to M$ is the map corresponding to $f_i=f_{|{K}_i^{p-1}}:{K}_i^{p-1} \to M$.
%\item If ${K}^p$ is a cubical simplex of~$M$ (for some cubical structure of~$M$) and $f:{K}^p \to M$ is given by the inclusion map ${K}^p \hookrightarrow M$ then $\Sigma^p={K}^p$ and the corresponding map $F:\Sigma^p \to M$ agrees with $f:{K}^p \to M$.
%\end{enumerate}
%
%In addition, if ${K}^p$ is a compact $p$-pseudomanifold so does~$\mathcal{X}^p$.
%In this case, ${K}^p$ is closed if and only if $\mathcal{X}^p$ is closed.
%Moreover, by the coherence of the replacement procedure, the cubical decomposition of the boundary of~${K}^p$ induces a decomposition of the boundary of~$\mathcal{X}^p$ into domains homeomorphic to $p$-pseudomanifolds.
\forgotten
\end{proposition}

\begin{proof}
%For every $k \geq 1$ and every continuous map $G:\Sigma_0^k \to M$ of volume at most~$v$ defined on a closed $k$-pseudomanifold~$\Sigma_0^k$, fix once and for all a continuous extension $H:\Sigma^{k+1} \to M$ of volume
%\begin{equation} \label{eq:choice}
%\vol_{k+1}(H) \leq \FH_k(v) + \varepsilon
%\end{equation}
%defined on a compact $(k+1)$-pseudomanifold~$\Sigma^{k+1}$ with $\partial \Sigma^{k+1} = \Sigma_0^k$; see Definition~\ref{def:FH}.
%To avoid burdening the argument by epsilontics, we will assume that $\varepsilon=0$.

We argue by induction on~$p$.
If $p=1$, we simply take $\mathcal{X}^1=K^1$ and \mbox{$F:K^1 \to M$} for $f:K^1 \to M$.
%the property~\eqref{assignd every edge of~$$1} yields the desired transformation~$\Phi$.
The inequality~\eqref{eq:volp} and the properties~\eqref{assign3}-\eqref{assign5} are satisfied in this case.
Suppose that the result of the proposition holds true for~$p \geq 1$. 
Let $K^{p+1}$ be a cubical $(p+1)$-complex and $f:K^1 \to M$ be a continuous map as in the proposition.
%We can assume that the $\mathcal{R}$-transformation of~$f$ modeled on the $p$-skeleton of~$K^{p+1}$ is already defined. 
Let us define the $\mathcal{R}$-transformation of the restriction of~$f$ to the $1$-skeleton of each cubical $(p+1)$-simplex~$C^{p+1} \subseteq K^{p+1}$.
%It K^{p+1 remains to define~$\Phi$ for the restriction of $f$ to each cubical $(p+1)$-simplex~$C^{p+1}$ of~$K^{p+1}$ so that~\eqref{assign2}, \eqref{assign3}, \eqref{assign4}, \eqref{assign5} hold.
%The property~\eqref{assign3} is easy to satisfy: If $C^{p+1}$ satisfies the assumption of~\eqref{assign3}, then the transformation~$\Phi$ takes the map $f_{|\partial C^{p+1}}:\partial C^{p+1} \to M$ to itself and we can define the image of $f_{|C^{p+1}}:C^{p+1} \to M$ under~$\Phi$ as $f_{|C^{p+1}}:C^{p+1} \to M$.
%%the already constructed $A$ on $(\partial C^{p+1}, f|_{\partial C^{p+1}})$ is the identity, and we can just define $A((c^{p+1},f|_{C^{p+1}}))$ as $(c^{p+1},f|_{C^{p+1}})$.
%Also, if the property~\eqref{assign5} is satisfied, then the property~\eqref{assign4} automatically follows from~\eqref{assign5} and the fact that the property~\eqref{assign4} holds for the $p$-skeleton of~$K^{p+1}$.
%%We need to define $A$ for $(p+1)$-dimensional complexes.
%%Consider first the case where ${K}^{p+1}$ is the closed $(p+1)$-cube regarded as  a cubical $(p+1)$-complex~$C^{p+1}$.
%%Let $f:C^{p+1} \to M$ be a continuous map whose length of the images of the edges of~$C^{p+1}$ are less than~$\delta$.
Denote by $C_i^p$ the $p$-faces of~$C^{p+1}$ with $1 \leq i \leq 2(p+1)$.
By induction, the $\mathcal{R}$-transformation of the restriction of~$f$ to the $1$-skeleton of~$C_i^p$ is a map $F_i:\mathcal{X}_i^p = \Sigma_i^p \to M$ defined on a compact $p$-pseudomanifold~$\Sigma_i^p$ with
\[
\vol_p(F_i) \leq \FHb_{p-1} \circ \cdots \circ \FHb_1(\delta).
\]
By coherence of the $\mathcal{R}$-transformation, the maps $F_i:\Sigma_i^p \to M$ and $F_j:\Sigma_j^p \to M$ coincide on the cubical $(p-1)$-complex given by the intersection $\Sigma_i^p \cap \Sigma_j^p$.
Thus, the maps $F_i:\Sigma_i^p \to M$ give rise to a continuous map $G:\cup_{i=1}^{2(p+1)} \Sigma_i^p \to M$ with 
\[
\vol_p(G) \leq 2(p+1) \, \FHb_{p-1} \circ \cdots \circ \FHb_1(\delta).
\]
The cubical structures of the $p$-faces~$C_i^p$ of $\partial C^{p+1} = \cup_i C_i^p$ induce compatible natural decompositions of the boundaries~$\partial \Sigma_i^p$ of the pseudomanifolds~$\Sigma_i^p$ into compact $(p-1)$-pseudomanifolds corresponding to the $(p-1)$-faces of~$C_i^p$.
%%The domains of these decompositions pairwise coincide for different intersecting $\Sigma_i^p$ and~$\Sigma_j^p$.
%A piece of the boundary corresponding to a fixed cubical $(p-1)$-simplex appears twice with opposite orientations.
Every compact $(p-1)$-pseudomanifold of these decompositions appears twice with opposite orientations.
Therefore, the sum of the boundaries of the pseudomanifolds~$\Sigma_i^p$ vanishes.
Thus, the space~$\cup_i \Sigma_i^p$ obtained by replacing the $p$-faces of $\partial C^{p+1} = \cup_i C_i^p$ with the compact $p$-pseudomanifolds~$\Sigma_i^p$ is a closed $p$-pseudomanifold~$\Sigma^p$.
By definition of the homological filling function, there exists a continuous extension $F:\Sigma^{p+1} \to M$ of $G:\Sigma^p \to M$ defined on a compact $(p+1)$-pseudomanifold~$\Sigma^{p+1}$ with $\partial \Sigma^{p+1} = \Sigma^p$ such that
\begin{equation} \label{eq:vol(p+1)F}
\vol_{p+1}(F) \leq \FHb_p \circ \cdots \circ \FHb_1(\delta). %+\epsilon,
\end{equation}
%where $\epsilon>0$ can be arbitrarily small.
%We define $F$ on $C^{p+1}$ as $H$.
%If each map $F_i:\Sigma_i^p \to M$, where $\Sigma_i^p = C_i^p$ is a face of a cubical $(p+1)$-simplex~$C^{p+1}$ of~$M$, agrees with the inclusion map \mbox{$\Sigma_i^p = C_i^p \hookrightarrow M$}, we can take $F:\mathcal{X}^{p+1} = \Sigma^{p+1} \to M$ defined on~$\Sigma^{p+1} = C^{p+1}$ as the inclusion of~$C^{p+1}$ into~$M$.
Gluing together the compact $(p+1)$-pseudomanifolds~$\Sigma^{p+1}$ corresponding to the cubical $(p+1)$-simplices~$C \subseteq K^{p+1}$ along their common faces~$\Sigma_i^p$ following the combinatorial structure of~$K^{p+1}$, we obtain a cubical $(p+1)$-complex~$\mathcal{X}^{p+1}$ and a continuous extension $F:\mathcal{X}^{p+1} \to M$ of $f:K^1 \to M$ modeled on~$K^{p+1}$ satisfying~\eqref{assign5}.

Now, if $f$ is defined on~$K^{p+1}$ and if each cubical $j$-simplex~$C \subseteq K^{p+1}$ satisfies~\eqref{eq:fC}, then $\Sigma_i^p = C_i^p$ and $F_i:C_i^p \to M$ agrees with~$f_{|C_i^p}:C_i^p \to M$ by induction.
In this case, we can take \mbox{$f:K^{p+1} \to M$} for its $\mathcal{R}$-transformation since $f_{|C^{p+1}}:C^{p+1} \to M$ satisfies the volume bound~\eqref{eq:vol(p+1)F}.
Thus, the property~\eqref{assign3} is satisfied.
The property~\eqref{assign4} follows by induction (it holds for the $p$-skeleton of~$K^{p+1}$) and construction.
%When ${K}^{p+1}$ is an arbitrary finite cubical $(p+1)$-complex, we apply the previous construction to the restriction of~$f$ to each cubical $k$-simplex~$C^k$ of~${K}^{p+1}$ with $k \leq p+1$.
%This gives rise to a continuous map $F:\Sigma^k \to M$ defined on a compact $k$-pseudomanifold~$\Sigma^k$ satisfying~\eqref{eq:volp} for $k=p$ with $\kappa=1$.
%By coherence of the replacement procedure, these maps put together yield the desired map $F:\mathcal{X}^{p+1} \to M$ defined on the union $\mathcal{X}^{p+1} = \cup \Sigma^k$ over all the pseudomanifolds~$\Sigma^k$ with $k \leq p+1$.
%By construction, the replacement procedure is coherent.
%If ${K}^p$ is a compact $p$-pseudomanifold, the cubical structures of the cubical $p$-simplices~$C^p$ of~${K}^p$ induce compatible natural decompositions on the boundaries of the $p$-pseudomanifold~$\Sigma^p$.
%Since the combinatorial pattern of the pieces~$\Sigma^p$ in~$\mathcal{X}^p$ follows the same pattern as the cubical simplices~$C^p$ in the $p$-pseudomanifold~${K}^p$, we deduce that $\mathcal{X}^p$ is a compact $p$-pseudomanifold.
%Moreover, $\mathcal{X}^p$ is closed if and only if ${K}^p$ is closed.
\end{proof}

\begin{remark}
The $\mathcal{R}$-transformation of the map $f:{K}^1 \to M$ modeled on~$K^p$ only depends on the choice of the filling pseudomanifolds involved in the homological filling functions.
\end{remark}

The homology $p$-waist is related to the filling radius through the homological filling function; see Definition~\ref{def:FH}.

\begin{theorem} \label{theo:FR3}
Let $M$ be a closed Riemannian $n$-manifold.
Then, for every positive integer~$p$,
\[
\w_p(M) \leq \frac{1}{2^{n-p+1}} \textstyle{\binom{n+1} {p}}^{-1} \, \FHb_{p-1} \circ \cdots \circ \FHb_1(2 \, \FR(M)).
\]
\end{theorem}

\begin{proof}
We argue as in the proof of Theorem~\ref{theo:FR1} using the same notations.
Fix a geodesic cubical structure of~$M$ of size at most~$\varepsilon>0$ such that every cubical simplex of~$M$ is an almost minimal filling of its boundary.
Suppose that there exists a continuous map $\sigma:P \to U_\nu(M) \subseteq L^\infty(M)$ defined on a compact cubical $(n+1)$-pseudomanifold such that the restriction $\sigma:\partial P \to M$ is a PL-homeomorphism, and, therefore, satisfies 
\[
\sigma_*([\partial P]) = [M] \in H_n(M)
\]
where $\nu$ is chosen so small that
\begin{equation} \label{eq:nu}
\frac{1}{2^{n-p+1}} \textstyle{\binom{n+1} {p}}^{-1} \, \FHb_{p-1} \circ \cdots \circ \FHb_1(2 \nu) < \w_p(M).
\end{equation}

Construct a map $f:P^1 \to M$ as in the proof of Theorem~\ref{theo:FR1} by projecting the images by~$\sigma$ of the vertices of~$P$ to their closest points in~$M$ and by sending every edge of~$P$ to a segment of~$M$. The length of these segments is at most~$\delta:= 2\nu+\varepsilon$.
%where $\varepsilon>0$ is the size of the cubical structure of~$P$.
Taking a sufficiently fine subdivision of~$P$, we can assume without loss of generality that the inequality~\eqref{eq:nu} is still satisfied with~$\delta$ replacing~$2\nu$.
Note that $f$ agrees with~$\sigma$ on~$P^1 \cap \partial P$.
%and that the image under $\sigma$ of each cubical $n$-simplex of~$\partial P$  is so small that it is $(1+\epsilon)$-bi-Lipshitz homeomorphic to a small cube in the Euclidean space.
%, and each $i$-dimensional cube is an almost minimal filling of its boundary. 
%As a result, the $\mathcal{R}$-transformation defined in Proposition~\ref{prop:cage} leaves the restriction of~$\sigma$ to any subcomplex of~$\partial P$ unchanged by the property~\eqref{assign3}.
Now, denote by $Q \subseteq P$ the neighborhood of the $p$-skeleton~$P^{(p)}$ of~$P$ composed of the pieces ${X}_1^{n+1} \subseteq C^{n+1}$ corresponding to the cubical $(n+1)$-simplices~$C^{n+1}$ of~$P$; see Section~\ref{sec:simplex}.
Define a ``retraction" $r:Q \to P^{(p)}$ by putting together the ``retractions" $\rho:X_1^{n+1} \to (C^{n+1})^{(p)}$ described in~\eqref{eq:rho}.
Deform $\sigma:\partial P \to M$ into $\bar{\sigma}:\partial P \to M$ so that the restriction of~$\bar{\sigma}$ to each cubical $n$-simplex of~$\partial P$ agrees with $\sigma \circ \bar{\rho}$, where $\bar{\rho}$ is the extension of the ``retraction"~$\rho$ defined in~\eqref{eq:rho-bar}.
Define
\[
\bar{f}:P^1 \cup \partial P \to M
\]
which agrees with $f \circ r:P^1 \to M$ on~$P^1$ and with $\bar{\sigma}:\partial P \to M$ on~$\partial P$.
Contrary to the proof of Theorem~\ref{theo:FR1}, see~\eqref{eq:f-bar}, it may not be possible to extend~$\bar{f}:P^1 \cup \partial P \to M$ to~$Q \cup \partial P$, or even to $P^{(p)} \cup \partial P$, when $p>1$.
Instead, we will consider the $\mathcal{R}$-transformation of $\bar{f}:P^1 \cup \partial P \to M$ modeled on~$P^{(p)} \cup \partial P$ to carry on the argument.

%Given a continuous extension of $f:P^1 \cup \partial P \to M$ to the $p$-skeleton of $P$, {\it i.e.}, a continuous map $P^{(p)} \to M$, one can can construct a continuous map
%\[
%\bar{f}:Q \cup \partial P \to M
%\]
%as in the proof of Theorem~\ref{theo:FR1}, see~\eqref{eq:f-bar}, which agrees with $f \circ r:Q \to M$ on~$Q$ and with the deformation $\bar{\sigma}:\partial P \to M$ of~$\sigma$ on~$\partial P$. The problem is that it may not be possible to extend~$f:P^1 \cup \partial P \to M$ to~$P^{(p)}$. Instead, one can perform the replacement procedure from the proof of Proposition~\ref{prop:cage}: We inductively assign to each cubical $i$-simplex~$C$ of~$P^{(p)}$ an almost minimal filling in~$M$ of the already constructed (map of the) pseudomanifold that corresponds to~$\partial C$. As a result, $\bar f$ is defined on $\bar Q\cup \partial M$, where $\bar Q$ is a ``pseudomanifold complex" with the same combinatorial structure as~$Q$. That is, each cubical simplex of~$Q$ is replaced by a pseudomanifold with boundary, yet all these pseudomanifolds are assembled in the same way as the cubical simplices of~$Q$.

%, and takes every edge of~$P^1$ to a segment of length less than~$\delta$, with $\delta=2\nu + \varepsilon$.
%(By subdividing~$P$ if necessary, we can assume that the size~$\varepsilon$ of the images of the cubical simplices of~$P$ is so small that the inequality~\eqref{eq:nu} is still satisfied with~$\delta$ replacing~$2\nu$.)

\medskip

The compact cubical $n$-pseudomanifold~$N' \subseteq Q$ with boundary lying in~$\partial P$ is composed of the cubical $n$-pseudomanifolds ${Y}^n \subseteq C^{n+1}$, where $C^{n+1}$ is a cubical $(n+1)$-simplex of~$P$; see Section~\ref{sec:simplex}. 
%Replacing, if necessary, by a very fine triangulation, we can assume without any loss of generality that each $n$-dimensional cube is so small that it is $(1+\epsilon)$-bi-Lipshitz homeomorphic to a cube in the Euclidean space, and
%each $i$-dimensional cube is an almost minimal filling of its boundary.
Similarly, the cubical $(n-p)$-complex~${T}'$ is composed of the pieces ${Z}^{n-p} \subseteq C^{n+1}$, where $C^{n+1}$ is a cubical $(n+1)$-simplex of~$P$; see Section~\ref{sec:simplex}.
Denote by $\hbar':N' \to {T}'$ the map whose restriction to~${Y}^n$ agrees with the map ${\theta}:{Y}^n \to {Z}^{n-p}$ defined in~\eqref{eq:theta}.
The cubical $n$-complexes ${X}_i^n \subseteq C^n$, where $C^n$ is a cubical $n$-simplex of~$\partial P \simeq M$, form a cubical $n$-pseudomanifold $N_i' \subseteq \partial P$ with the same boundary as~$N_i''$.
Also, the pieces ${Z}^{n-p-1} \times [0,\frac{1}{2}]$ with ${Z}^{n-p-1} \subseteq C^n$, where $C^n$ is a cubical $n$-simplex of~$\partial P$, form a finite cubical $(n-p)$-complex~${T}''$.
Denote by $\hbar'':N_2'' \to {T}''$ the continuous map whose restriction to~${X}_2^n$ agrees with the map ${\Theta}:{X}_2^n \to {Z}^{n-p-1} \times [0,\frac{1}{2}]$ defined in~\eqref{eq:fhat}.
The two maps $\hbar'$ and~$\hbar''$ so-defined agree on the common boundary of~$N'$ and~$N_2''$ after identifying ${Z}^{n-p-1} \subseteq {Z}^{n-p} \subseteq {T}'$ and ${Z}^{n-p-1} \subseteq {Z}^{n-p-1} \times \{ \frac{1}{2} \} \subseteq {T}''$, where ${Z}^{n-p-1} \subseteq C^n$ lies in~$\partial P$.
Put together, these maps give rise to a continuous map
\[
\hbar:N \to {T}
\]
from the closed $n$-pseudomanifold~$N = N' \cup N_2''$ lying in~$Q \cup \partial P$ to the cubical $(n-p)$-complex ${T} = {T}' \cup_{{S}} {T}''$ obtained by gluing~${T}'$ and~${T}''$ along the cubical $(n-p-1)$-complex~${S}$ formed of the pieces ${Z}^{n-p-1} \subseteq C^n$, where $C^n$ is a cubical $n$-simplex of~$\partial P$.

\medskip

By construction, every fiber $\hbar^{-1}(t) \subseteq N'$ with $t \in T'$ agrees with a fiber of~${\theta}$, and therefore is isomorphic to the $p$-skeleton of a cube of dimension at most~$n+1$.
Moreover, the retraction $r:Q \to P^{(p)}$ sends every fiber~$\hbar^{-1}(t)$ with $t$ lying in the interior~$\mathring{\tau}$ of a cubical simplex~$\tau$ of~${T}'$ to the same cubical $p$-complex~$C_\tau \subseteq C^{n+1}$, preserving the cubical structure.
(Note that $C_\tau$ is only composed of $p$-cubes~${K}$ glued together.)
In particular, the preimage~$\hbar^{-1}(\mathring{\tau}) \subseteq N'$ of the interior of a cubical simplex~$\tau$ of~${T}'$ decomposes as $\hbar^{-1}(\mathring{\tau}) \simeq \mathring{\tau} \times C_\tau$ and the map~$\hbar:N' \to T'$ takes $(t,x) \in \mathring{\tau} \times C_\tau \subseteq N'$ to $\hbar(t,x)=t \in {T}'$.
Therefore, the compact $n$-pseudomanifold~$N'$ with boundary lying in~$\partial P \simeq M$ decomposes as the union
\begin{equation} \label{eq:N}
N' = \bigcup_{\tau \, \subseteq \, T'} \bigcup_{{K} \, \subseteq \, C_\tau} \tau \times {K} 
\end{equation}
over the cubical simplices~$\tau$ of~$T'$ and the $p$-cubes~${K}$ of~$C_\tau$, where $\tau_1 \times {K}_1$ is attached to $\tau_2 \times {K}_2$ along $(\tau_1 \cap \tau_2) \times ({K}_1 \cap {K}_2)$.
Note that if $\tau_1 \subseteq \tau_2$ then $C_{\tau_2} \subseteq C_{\tau_1}$.

\medskip

Let ${K}$ be a $p$-cube of~$P$.
By Proposition~\ref{prop:cage}, the $\mathcal{R}$-transformation $\bar{F}_{K}:\Sigma_{K} \to M$ of $\bar{f}_{|{K}^1}:{K}^1 \to M$ modeled on~$K$ is defined on a compact $p$-pseudomanifold~$\Sigma_{K}$ with boundary and satisfies
\begin{equation} \label{eq:volF1}
\vol_p(\bar{F}_{K}) \leq \FHb_{p-1} \circ \cdots \circ \FHb_1(\delta).
\end{equation}
Since the image under $\bar{f}=\bar{\sigma}$ of each cubical simplex of~$\partial P$ is an almost minimal filling of the image of its boundary, it follows from the property~\eqref{assign3} of the $\mathcal{R}$-transformation that for every $p$-cube $K \subseteq \partial P$, the pseudomanifold~$\Sigma_K$ is equal to~$K$ and the map $\bar{F}_{K}:\Sigma_K = K \to M$ agrees with $\bar{f}:K \to M$.
Replacing every $p$-cube~${K}$ with~$\Sigma_{K}$ in~\eqref{eq:N}, we obtain a simplicial $n$-complex
\[
V' = \bigcup_{\tau \, \subseteq \, T'} \bigcup_{{K} \, \subseteq \, C_\tau} \tau \times \Sigma_{K}
\]
where $\tau_1 \times \Sigma_{{K}_1}$ is attached to $\tau_2 \times \Sigma_{{K}_2}$ along $(\tau_1 \cap \tau_2) \times ( \Sigma_{{K}_1} \cap \Sigma_{{K}_2})$ whenever $\tau_1 \times {K}_1$ is attached to $\tau_2 \times {K}_2$ in~\eqref{eq:N}.
By the coherence property~\eqref{assign4} of Proposition~\ref{prop:cage}, the pseudomanifold structure of~$N'$ carries over to~$V'$.
More precisely, $V'$ is a compact $n$-pseudomanifold with the same boundary as~$N'$.
Thus, the union $V=V' \cup N_2''$ is a closed $n$-pseudomanifold.

\medskip

For every cubical simplex~$\tau$ of~$T'$, define
\begin{equation} \label{eq:X_tau}
{X}_\tau = \bigcup_{{K} \, \subseteq \, C_\tau} \Sigma_{{K}}
\end{equation}
as the union of the compact $p$-pseudomanifolds~$\Sigma_{{K}}$ corresponding to the $p$-cubes~${K}$ of~$C_\tau$, where $\Sigma_{{K}_1}$ is attached to~$\Sigma_{{K}_2}$ along~$\Sigma_{{K}_1 \cap {K}_2}$; see Proposition~\ref{prop:cage}.
The map 
\[
h:V \to T
\]
which agrees with $\hbar'':N_2'' \to T''$ on~$N_2''$ and sends $(t,x) \in \mathring{\tau} \times {X}_\tau$ to $t \in T$ for every cubical simplex~$\tau$ of~$T'$ is well defined and continuous.
By the coherence property~\eqref{assign4} of Proposition~\ref{prop:cage}, the previously defined maps $\bar{F}_{K}:\Sigma_{K} \to M$ put together induce a map $\bar{F}_\tau:{X}_\tau \to M$.
Now, consider the map 
\[
F:V \to M
\]
which coincides with $\bar{f}: \partial P \to M$ on~$N''_2 \subseteq \partial P$ and whose restriction to each fiber $h^{-1}(t) \simeq \{ t \} \times {X}_\tau$ with $t \in \mathring{\tau}$ agrees with $\bar{F}_\tau:{X}_\tau \simeq h^{-1}(t) \to M$.
By the coherence property and since $\bar{F}_K = \bar{f}_{K}$ for every $p$-cube $K \subseteq \partial P$, the map $F:V \to M$ is well defined and continuous.

\medskip

For every cubical simplex~$\tau$ of~$T'$, recall that the cubical $p$-complex~$C_\tau$ lies in~$C^{n+1}$.
Since $C^{n+1}$ has at most $k=2^{n-p+1} \binom{n+1} {p}$ faces of dimension~$p$, the cubical complex~${X}_\tau \simeq h^{-1}(t)$ for $t \in \mathring{\tau}$ is composed of at most $k$ pseudomanifolds~$\Sigma_{K}$; see~\eqref{eq:X_tau}.
Since the volume of the restriction of~$F$ to each of these pseudomanifolds satisfies~\eqref{eq:volF1}, we obtain
\begin{equation*} %\label{eq:volF}
\vol_p [F(h^{-1}(t))] \leq 2^{n-p+1} \textstyle{\binom{n+1} {p}} \, \FHb_{p-1} \circ \cdots \circ \FHb_1(\delta) < W_p(M)
\end{equation*}
for every $t \in \mathring{\tau}$, where the second inequality follows from the filling radius assumption; see~\eqref{eq:nu}.
Thus, by definition of the homology $p$-waist, see Definition~\ref{def:width}, the map $F:V \to M$ satisfies
\begin{equation} \label{eq:Vneq}
F_*([V]) \neq [M] \in H_n(M).
\end{equation}

\medskip

In a different direction, the retraction~$r:Q \to P^{(p)}$ takes every piece $\tau \times C_\tau$ of~$N$, where $\tau$ is a cubical simplex of~$T'$, see~\eqref{eq:N}, to the cubical $p$-complex~$C_\tau$.
Thus, the map $\bar{f}:Q \to M$ defined as $\bar{f} = f \circ r$ takes $\tau \times C_\tau$ to the image of~$C_\tau$ in~$M$ by~$f$.
Similarly, the map $F:V \to M$ takes $\tau \times {X}_\tau$ to the image of~${X}_\tau$ in~$M$ by~$\bar{F}_\tau$, where $\tau$ is a cubical simplex of~$T'$.
%; see~\eqref{eq:V}.
Therefore, the contribution of the pieces $\tau \times {X}_\tau$ to the image by~$F_*$ of the fundamental class of~$V$ is trivial.
Now, by construction, the map $\bar{f}: \partial P \to M$ induces a degree one map in relative homology between every cubical $n$-complex~$X_2^n$ of~$N_2''$ and the cubical $n$-simplex~$C^n$ of~$\partial P$ containing~$X_2^n$.
Since the map~$F$ agrees with~$\bar{f}$ on~$N_2''$, we deduce that
\[
F_*([V]) = [M] \in H_n(M).
\]
Hence a contradiction with~\eqref{eq:Vneq}.
\end{proof}

\begin{remark} \label{rem:improve}
Working this simplicial complexes instead of cubical complexes yields a better quantitative estimate in Theorem~\ref{theo:FR3}.
Namely,
\[
\w_p(M) \leq %\frac{1}{2^{n-p+1}}
\textstyle{\binom{n+1} {p}}^{-1} \, \FH_{p-1}(p \, \FH_{p-2}(\cdots (3 \, \FH_2(2 \, \FH_1(6\, \FR(M)))) \cdots )).
\]
\end{remark}

Since the homological filling functions are nondecreasing, Theorem~\ref{theo:E} follows from Theorem~\ref{theo:FR3} and Theorem~\ref{theo:FR}.

\end{document}